\documentclass[a4paper,11pt,final]{article}

\usepackage[T1]{fontenc}
\usepackage{amsthm}
\usepackage{amsfonts}
\usepackage{graphicx}

\usepackage[a4paper,scale=0.79]{geometry}

\newtheorem{Theorem}{Theorem}[section]
\newtheorem{Definition}{Definition}[section]
\newtheorem{Lemma}{Lemma}[section]

\begin{document}

\title{Implicit monotone difference methods
for scalar conservation laws with source terms}

\author{Michael Breu\ss{} \thanks{Institute for Mathematics, 
Brandenburg University of Technology Cottbus-Senftenberg, Platz der Deutschen Einheit 1, 03046 Cottbus, Germany ({\tt breuss@b-tu.de})} 
\and Andreas Kleefeld \thanks{Forschungszentrum J\"ulich GmbH, Institute for Advanced Simulation, J\"ulich Supercomputing Centre,
Wilhelm-Johnen-Stra\ss{}e, 52425 J\"ulich, Germany ({\tt a.kleefeld@fz-juelich.de})}}

\maketitle

\begin{abstract}
In this article, a concept of implicit 
methods for scalar conservation laws in one or 
more spatial dimensions allowing also for
source terms of various types is presented. This material is a significant extension 
of previous work of the first author \cite{breuss02}. Implicit notions are developed that are
centered around a monotonicity criterion. 
We demonstrate a connection between a numerical scheme and a discrete entropy inequality, 
which is based on a classical approach by Crandall and Majda.
Additionally, three implicit methods are investigated using 
the developed notions. Next, we conduct a convergence proof which is not based on a classical compactness argument.
Finally, the theoretical results are confirmed by various numerical tests.
\end{abstract}

\section{Introduction}
\label{breuss-section-1}

This article deals with the entropy solution of hyperbolic conservation laws
in the sense of Kru\v{z}kov.
Specifically, we allow the numerical methods
to act within the two most general
settings, that is (i) smooth fluxes together with
non-linear sources and (ii) continuous fluxes and
sources depending both on space and time.
The corresponding analytical existence and uniqueness
results for these cases are given
within a number of papers of Kru\v{z}kov and
his co-workers, see for example \cite{benkru96,kruhil74,krupan91}
and the references therein. 

This paper represents a significant extension 
of the work by Breu\ss{} \cite{breuss02}, where implicit methods are
considered for homogeneous scalar equations in one dimension.
To our knowledge, the combination of the developed
concept of implicit methods
for both mentioned general problems
together with the application of corresponding schemes
on problems belonging to both classes is new.
Accordingly, the main contribution of this paper
is the extension of the rigorously validated range
of applicability of finite difference methods.

The encountered difficulties for the described task have already been discussed 
in the introduction of Breu\ss{} \cite{breuss02}. Summarizing, information that is propagated 
with infinite speed may take place provided that a flux function of a nonlinear 
conservation law is not Lipschitz continuous as it is accepted in setting (ii). A detailed 
one-dimensional example is given by Kru\v{z}kov and Panov in \cite{krupan91} (see also \cite{breuss02}), 
where the exact solution is known. This example shows that 
a rarefaction wave extending to infinity after 
arbitrarily small time takes place. Additionally, this example has a pole for $u=0$ and the solution
domain is infinite although an initial condition with compact support is given.

Two direct conclusions emerge from this example. At first,
the Courant-Friedrichs-Lewy (CFL) number would be effectively zero provided an 
explicit scheme is used. Additionally, the Kuznetsov approach for convergence \cite{kuznetsov76} is not employable, because it
relies on a suitable error estimation which explicitly 
uses the Lipschitz continuity of the flux and 
the boundedness of the domain of the solution. At second, a variety of other well-established approaches for 
the convergence of numerical methods are not applicable. For instance, one approach is based on Helly's theorem
which uses the compactness
of the function space of bounded variations (BV). This is employed in the convergence proofs of Total Variation Diminishing (TVD)
methods. But using the BV concept, this function space is only compact (see LeVeque \cite{leveque92}) provided 
a fixed compact space-time-domain
containing the solution is used. Hence, the
compactness property of this function space is 
unfortunately not applicable in the discussed case. The same is true for explicit monotone methods as it is the case in the fundamental work of 
Crandall and Majda \cite{cramaj80}. They used the properties
of this function space to obtain a compactness
argument. Especially, in that work the sources are
also assumed to be essentially bounded and
BV-stable integrable functions depending on space and time.
In another important approach introduced by DiPerna \cite{diperna85} measure valued
solutions are used, where the
compactness of the domain of the solution both
in space and time is assumed which has already been discussed in \cite{coqflo93}. 

From this discussion it should be clear that we need to employ implicit schemes such that
the convergence strategy is different from the before mentioned 
ones. Therefore, we use the monotonicity of 
implicit methods to obtain a discrete comparison principle. This suffices to 
guarantee the convergence of such methods to the entropy solution
in the sense of Kru\v{z}kov.

Therefore, in this paper the monotonicity property of an implicit scheme is investigated (see \cite{godrav91,leveque92} for the discussion for explicit schemes).
Hence, it is indispensable to avoid any derivative of the flux. As we will show, we construct a monotonicity notion that is based 
on a comparison of data sets using an induction principle. 

The application of this monotonicity notion on three
implicit variations of well-known monotone explicit schemes
is investigated. One would expect, that implicit
schemes are generally capable to capture all effects described
by a conservation law even for continuous fluxes
and general sources, because 
in the implicit case the numerical characteristics include all the
characteristics of the differential equation.
However, while our monotonicity 
investigations of an implicit upwind scheme and an
implicit Godunov-type method yield the expected
results, the investigation of the implicit variation of
the traditional Lax-Friedrichs scheme shows, that the scheme
is only monotone even in the full implicit case
if the flux is Lipschitz continuous. Furthermore, the restriction
on the admissible Lipschitz constant of the flux is not
depending on the number of spatial dimensions.
This interesting result which is new to our
knowledge is explored via a simple experiment 
using a two-dimensional linear advection equation.

This article consists of five additional sections. In Section \ref{breuss-section-2},
we briefly review the
two most general theoretical results on solutions of
conservation laws available
to our knowledge, namely the existence and
uniqueness results established in \cite{benkru96} 
and \cite{kruzhkov70}.
In the next section, we introduce the notions for implicit methods that are
centered around monotonicity. 
The given detailed convergence proof is an extension of the strategy given in Breu\ss{} \cite{breuss02}.
Section \ref{breuss-section-4} presents the investigation of three numerical methods with respect to their
monotonicity. Additionally, for these
methods the proofs of convergence towards 
the entropy solution are given.
Finally, we present the results
of various numerical tests in Section \ref{breuss-section-5} followed by a short summary and conclusive remarks in Section \ref{breuss-section-6}.

\section{The setting}
\label{breuss-section-2}

Within this section, we define the two mathematical
scenarios of interest, i.e. we briefly review the type
of problems considered in \cite{benkru96} and \cite{kruzhkov70}.

\subsection*{Scenario 1}
\label{breuss-subsection-2-1}

The Cauchy problem under consideration is
\begin{eqnarray}
\frac{\partial}{\partial t} u \left( \mathbf{x}, t \right)
+
\sum_{l=1}^d
\frac{\partial}{\partial x_l}
f_l \left( u \left( \mathbf{x}, t \right) \right)
\, = \, q
& \, \textrm{ on } \, & \mathbb{R}^d \times \left(0,T\right)
\,,
\label{breuss-1}\\
u \left( \mathbf{x}, 0 \right)
\, = \, u_0 \left( \mathbf{x} \right)
& \, \textrm{ on } \, & \mathbb{R}^d
\,,
\label{breuss-2}
\end{eqnarray}
where $T$ is a fixed positive number. 
Concerning the flux functions we generally assume
\begin{equation}
f_l(u) \in C \left( \mathbb{R}; \, \mathbb{R} \right)
\,, \quad
l=1, \, \ldots, \, d
\,.
\label{breuss-3}
\end{equation}
In order to apply 
the uniqueness theorem given in \cite{benkru96}, the fluxes
are additionally supposed to satisfy the growth conditions
\begin{displaymath}
\left|
f_l(u)-f_l(\hat u)
\right|
\leq
\omega_l (u-\hat u)
\;
\textrm{ a.e. for }\;
u \geq \hat u \;
\textrm{ and for }\;
l=1, \ldots, d
\,,
\end{displaymath}
with the moduli of continuity $\omega_l$ featuring
\begin{displaymath}
\omega_1(0)= \ldots = \omega_d(0)=0
\quad \textrm{and} \quad
\liminf_{r \to 0}
r^{1-d} \prod_{l=1}^d \omega_l(r) < \infty
\,.
\end{displaymath}
Note that these conditions on the fluxes are more general
than the usually assumed Lipschitz continuity.
The initial condition shall satisfy
\begin{equation}
u_0 \in L^\infty_{loc} \left( \mathbb{R}^d; \, \mathbb{R} \right)
\,,
\label{breuss-4}
\end{equation}
and for the source term we consider
\begin{equation}
q \equiv q \left( \mathbf{x}, t \right)
\in L^1_{loc} \left( \mathbb{R}^d \times (0,T); \, \mathbb{R} \right)
\,,
\label{breuss-5}
\end{equation}
\begin{equation}
q (\cdot,t) \in L^\infty \left( \mathbb{R}^d; \, \mathbb{R} \right)
\; \textrm{ for a.e. } \; t \in (0,T)
\; \textrm{ and } \;
\int_0^T \| q(\cdot, t) \|_\infty \, \mathrm{d}t<\infty
\,.
\label{breuss-6}
\end{equation}
Under the conditions (\ref{breuss-3}) -- (\ref{breuss-6}),
B\'enilan and Kru\v{z}kov \cite{benkru96}
proved uniqueness of the entropy solution of 
(\ref{breuss-1}) -- (\ref{breuss-2}).

Because the solution of the Cauchy problem
generally develops discontinuities even if 
$u_0$ is smooth, it is often considered in its weak form, i.e.
\begin{eqnarray}
\lefteqn{
\int_0^\infty
\int_{\mathbb{R}^d}
\left[
u \left( \mathbf{x},t \right)
\phi_t \left( \mathbf{x},t \right)
+
\sum_{l=1}^d
f_l(u \left( \mathbf{x},t \right) )
\frac{\partial}{\partial x_l}
\phi \left( \mathbf{x},t \right)
\right] \,
\mathrm{d}\mathbf{x}\,\mathrm{d}t
}\nonumber\\
\quad \quad & = &
-\int_{\mathbb{R}^d}
u_0 \left( \mathbf{x}\right)
\phi_0 \left( \mathbf{x} \right) \, \mathrm{d}\mathbf{x}
-
\int_0^\infty
\int_{\mathbb{R}^d}
q \left( \mathbf{x},t \right)
\phi \left( \mathbf{x},t \right) \, \mathrm{d}\mathbf{x}\,\mathrm{d}t
\; \; \;
\forall \phi \in C_0^\infty \left( \mathbb{R}^{d+1}; 
\, \mathbb{R}\right)
\,.
\label{breuss-7}
\end{eqnarray}
It is well-known that weak solutions are in general not unique, 
see for example \cite{leveque92} and the references therein.
In order to ensure uniqueness, a so-called entropy condition
has to be introduced. The already mentioned entropy condition
due to Kru\v{z}kov \cite{benkru96}
which guarantees the uniqueness of a solution of 
(\ref{breuss-1}) -- (\ref{breuss-2}) takes the form
\begin{eqnarray}
\lefteqn{
\int_0^\infty
\int_{\mathbb{R}^d}
\Biggl[
\mathopen{|}
u \left( \mathbf{x},t \right)
 - k
\mathclose{|}
\phi_t \left( \mathbf{x},t\right)
\Biggr. }
\nonumber\\
\lefteqn{ 
\quad \quad
\Biggl.
+
\sum_{l=1}^d
\mathrm{sgn} \left( u\left( \mathbf{x},t \right)
-k \right)
\left[ f_l (u \left( \mathbf{x},t \right) )
-f_l (k) \right]
\frac{\partial}{\partial x_l}
\phi \left( \mathbf{x},t \right)
\Biggr] \,
\mathrm{d}\mathbf{x}\,\mathrm{d}t
}
\nonumber\\
& \geq &
-
\int_{\mathbb{R}^d}
\mathopen{|}
u_0 \left( \mathbf{x} \right)
 - k
\mathclose{|}
\phi_0 \left( \mathbf{x} \right)
\, \mathrm{d}\mathbf{x} \nonumber\\
&& \quad
-
\int_0^\infty
\int_{\mathbb{R}^d}
\mathrm{sgn} \left[ u \left( \mathbf{x},t \right)
-k \right]
q \left( \mathbf{x},t \right)
\phi \left( \mathbf{x},t \right)
\, \mathrm{d}\mathbf{x}\,\mathrm{d}t \label{breuss-8}\\
&&
\textrm{for all } 
\phi \in C_0^\infty \left( \mathbb{R}^{d+1}; \, \mathbb{R} \right)
\textrm{ with }
\phi \geq 0
\textrm{ and for all } k \in \mathbb{R} 
\,.
\nonumber
\end{eqnarray}

\subsection*{Scenario 2}
\label{breuss-subsection-2-2}

The {\it Scenario 2} deals with the Cauchy problem
\begin{eqnarray}
\frac{\partial}{\partial t} u \left( \mathbf{x}, t \right)
+
\sum_{l=1}^d
\frac{d}{d x_l}
f_l \left( \mathbf{x} ,t, u \left( \mathbf{x}, t \right) \right)
\, = \, q
& \, \textrm{ on } \, & \mathbb{R}^d \times \left(0,T\right)
\,,
\label{breuss-9}\\
u \left( \mathbf{x}, 0 \right)
\, = \, u_0 \left( \mathbf{x} \right)
& \, \textrm{ on } \, & \mathbb{R}^d
\,,
\label{breuss-10}
\end{eqnarray}
where $T$ is a fixed positive number
and with
\begin{displaymath}
\frac{d}{d x_l} f_l 
\equiv
f_{l_{x_l}}+f_{l_u}u_{x_l}
\,.
\end{displaymath} 
In comparison to {\it Scenario 1},
we impose different assumptions on the fluxes and 
the source terms. As in (\ref{breuss-4}),
there is no particular condition imposed
on the initial data.
The flux functions are now assumed to satisfy
\begin{equation}
f_l(\mathbf{x},t,u) 
\in
C^1 \left( 
\mathbb{R}^d \times \mathbb{R}_+ \times \mathbb{R}; \, \mathbb{R}\right)
\,, \quad
l=1, \, \ldots, \, d
\,.
\label{breuss-11}
\end{equation}
As source terms we consider functions
\begin{equation}
q \equiv q \left( \mathbf{x}, t , u(\mathbf{x},t) \right)
\in C^1
\left(
\mathbb{R}^d \times \mathbb{R}_+ \times \mathbb{R}; \, \mathbb{R}
\right)
\,.
\label{breuss-12}
\end{equation}
Under the conditions (\ref{breuss-11}) and
(\ref{breuss-12}), Kru\v{z}kov \cite{kruzhkov70}
proved the uniqueness of the entropy solution of
(\ref{breuss-9}) -- (\ref{breuss-10}).
Comparing the weak formulation of this problem with
the weak formulation (\ref{breuss-7}), we have to substitute
\begin{equation}
\int_0^\infty
\int_{\mathbb{R}^d}
q \left( \mathbf{x},t , u(\mathbf{x},t) \right)
\phi \left( \mathbf{x},t \right) \, \mathrm{d}\mathbf{x}\,\mathrm{d}t
\quad
\textrm{for}
\quad
\int_0^\infty
\int_{\mathbb{R}^d}
q \left( \mathbf{x},t\right)
\phi \left( \mathbf{x},t \right) \, \mathrm{d}\mathbf{x}\,\mathrm{d}t
\,.
\label{breuss-13}
\end{equation}
The assumptions (\ref{breuss-11}) and (\ref{breuss-12}) yield the form of the
Kru\v{z}kov entropy condition as
\begin{eqnarray}
\lefteqn{
\int_0^\infty
\int_{\mathbb{R}^d}
\Biggl[
\mathopen{|}
u \left( \mathbf{x},t \right)
 - k
\mathclose{|}
\phi_t \left( \mathbf{x},t\right)
\Biggr. }
\nonumber\\
\lefteqn{ 
\quad
\Biggl.
+
\sum_{l=1}^d
\mathrm{sgn} \left( u\left( \mathbf{x},t \right)
-k \right)
\left[ 
f_l (\mathbf{x},t, u \left( \mathbf{x},t \right) )
-f_l ( \mathbf{x},t,k) 
\right]
\frac{\partial}{\partial x_l}
\phi \left( \mathbf{x},t \right)
\Biggr] \,
\mathrm{d}\mathbf{x}\,\mathrm{d}t
}
\nonumber\\
& \quad \quad \geq &
-
\int_{\mathbb{R}^d}
\mathopen{|}
u_0 \left( \mathbf{x} \right)
 - k
\mathclose{|}
\phi_0 \left( \mathbf{x} \right)
\, \mathrm{d}\mathbf{x} \nonumber\\
&& \quad
-
\int_0^\infty
\int_{\mathbb{R}^d}
\sum_{l=1}^d
\mathrm{sgn} \left[ u \left( \mathbf{x},t \right)
-k \right]
\left[
q \left( \mathbf{x},t, u(\mathbf{x},t) \right)
-
f_{l_{x_l}} \left( \mathbf{x}, t, k \right)
\right]
\phi \left( \mathbf{x},t \right)
\, \mathrm{d}\mathbf{x}\,\mathrm{d}t 
\label{breuss-14} \\
&&
\textrm{for all } 
\phi \in C_0^\infty \left( \mathbb{R}^{d+1}; \, \mathbb{R} \right)
\textrm{ with }
\phi \geq 0
\textrm{ and for all } k \in \mathbb{R} 
\,.
\nonumber
\end{eqnarray}

\section{Numerical methods}
\label{breuss-section-3}

We first describe the implicit notions, followed
by the proofs of the involved
Lemmas and Theorems in a separate section.
For the sake of brevity,
we discuss only {\it Scenario 1} in detail, since 
the techniques which have to be used with respect to 
{\it Scenario 2} are identical.
The proper conceptual extension to {\it Scenario 2}
is described within additional remarks.

\subsection{A concept of implicit methods}
\label{breuss-subsection-3-1} 

Since we want to describe numerical methods
in $d$ spatial dimensions, we spend some effort
on a general notation.

Because we investigate finite difference methods, we have to
introduce grid points. For simplicity, we consider grids which are
equidistant with respect to the individual
$d$ spatial dimensions as well as to time,
i.e. we employ grid spacings $\Delta x_l$ corresponding
to the space dimensions $l=1, \ldots, d$,
and $\Delta t$ corresponding to time.

Since this results in a countable number of grid points, we
introduce a {\it linear numbering} $J$ of the spatial grid points
\begin{displaymath}
J= \left\{
0,1,2, \ldots \right\}
\,.
\end{displaymath}
We also define a bijective mapping
\begin{eqnarray}
\tilde J \quad : 
\quad J & \longrightarrow & \mathbb{R}^d \nonumber\\
i & \longrightarrow & 
\left(
i_1 \Delta x_1, \,
i_2 \Delta x_2, \,
\ldots, \,
i_d \Delta x_d 
\right)^\mathrm{T} 
\quad
\textrm{with}
\quad
\left(
i_1, \, i_2, \, \ldots, \, i_d 
\right)^\mathrm{T}
\in 
\mathbb{Z}^d
\,.
\nonumber
\end{eqnarray}
In order to describe the indices within the stencil of
a numerical method, we define the index $i \pm \delta l$ via
\begin{displaymath}
i \pm \delta l
\; 
\stackrel{\displaystyle{\tilde J}}{\longrightarrow}
\;
\left( 
i_1 \Delta x_1, \,
i_2 \Delta x_2, \,
\ldots, \,
\left(i_l \pm 1 \right) \Delta x_l, \,
\ldots, \,
i_d \Delta x_d
\right)^\mathrm{T}
\,.
\end{displaymath}
Let $u_j^k$ and $q_j^k$ denote the value of the numerical solution
and the value of the source term
at the point with the index $j \in J$ at the time level
$k \Delta t$, respectively.
With these notations, we consider {\it conservative}
implicit methods in the form
\begin{equation}
u_j^{n+1}
=
u_j^n
-
\sum_{l=1}^d
\frac{\Delta t}{\Delta x_l}
\left\{
g_l \left( u_j^{n+1}, \, u_{j + \delta l}^{n+1} \right)
-
g_l \left( u_{j-\delta l}^{n+1}, \, u_j^{n+1} \right)
\right\}
+
\Delta t q_j^{n+1}
\,.
\label{breuss-15}
\end{equation}
We assume that the numerical flux functions $g_l$ introduced
in (\ref{breuss-15}) are {\it consistent}, i.e.
\begin{displaymath} 
g_l(v, \, v) = f_l(v) \textrm{ holds for all } v \in \mathbb{R}
\textrm{ and for all } l=1, \, \ldots, \, d
\,.
\end{displaymath}
In the case of {\it Scenario 2}, we simply add
arguments $(\mathbf{x}_j,t^{n+1})$ within the fluxes;
we will not do this explicitly in the following.

The key to nonlinear stability is the notion
of monotonicity.\vspace{1.0ex}

\begin{Definition}[{\rm Monotonicity}]
\label{definition-breuss-1}
Let two data sequences
\begin{displaymath}
v^n = \left\{ v_j^n \right\}_{j \in J}
\quad \textrm{and} \quad
w^n = \left\{ w_j^n \right\}_{j \in J}
\end{displaymath}
be given. Let the investigated consistent and conservative
numerical method produce 
new sequences of data $v^{n+1}$ and $w^{n+1}$
from the given data $v^n$ and $w^n$, respectively.
Then the numerical method is monotone iff
the implication
\begin{equation}
v^n \geq w^n \quad
\Rightarrow \quad
v^{n+1} \geq w^{n+1}
\label{breuss-16}
\end{equation}
holds in the sense of the comparison
of components.\vspace{1.0ex}\\
\end{Definition}
It is useful to define $H$ and $\tilde H_l$ 
using $\underline{d}=\left\{ 1, \, \ldots, \, d \right\}$ via
\begin{eqnarray}
u_j^{n+1}
& = &
H \left( 
l \in \underline{d}, \,
u_{j-\delta l}^{n+1}, \,
u_j^{n+1}, \,
u_{j+\delta l}^{n+1}, \,
u_j^n \right) \nonumber\\
& = &
u_j^n
-
\sum_{l=1}^d
\frac{\Delta t}{\Delta x_l}
\left\{
g_l \left( u_j^{n+1}, \, u_{j + \delta l}^{n+1} \right)
-
g_l \left( u_{j-\delta l}^{n+1}, \, u_j^{n+1} \right)
\right\}
+
\Delta t q_j^{n+1}
\nonumber\\
& = &
u_j^n
+
\sum_{l=1}^d
\tilde H_l 
\left( u_{j-\delta l}^{n+1}, \, u_j^{n+1},
\, u_{j + \delta l}^{n+1} \right)
+
\Delta t q_j^{n+1}
\,.
\label{breuss-17}
\end{eqnarray}

\begin{Theorem}[{\rm Monotonicity of implicit methods}]
\label{theorem-breuss-1}
Let $a$, $b$ and $c$ be arbitrarily chosen
but fixed real numbers.
A consistent and conservative implicit method
of type (\ref{breuss-15}) 
is monotone iff  
for all spatial dimensions 
$l \in \left\{ 1, \, \ldots, \, d \right\}$ holds
\begin{eqnarray}
&&
\tilde H_l \left( a + \Delta a,b,c \right)
\quad \geq \quad
\tilde H_l \left( a,b,c \right)
\quad  \forall \, \Delta a \geq 0
\,,
\label{breuss-18}\\
&&
\tilde H_l \left( a,b,c+ \Delta c \right)
\quad \geq \quad
\tilde H_l \left( a,b,c \right)
\quad \forall \, \Delta c \geq 0
\,.
\label{breuss-19}
\end{eqnarray}
\end{Theorem}
Note that we have omitted the condition
\begin{displaymath}
H \left( 
l \in \underline{d}, \,
u_{j-\delta l}^{n+1}, \,
u_j^{n+1}, \,
u_{j+\delta l}^{n+1}, \,
s + \Delta s \right)
\geq
H \left( 
l \in \underline{d}, \,
u_{j-\delta l}^{n+1}, \,
u_j^{n+1}, \,
u_{j+\delta l}^{n+1}, \,
s \right)
\end{displaymath}
for all $j \in J$ and all $\Delta s \geq 0$,
since this condition is redundant. This is due to the form of
the method (\ref{breuss-15}).
Additionally, note that the monotonicity property 
does not depend on the exact nature of the source terms,
i.e. both {\it Scenario 1} and {\it Scenario 2}
are included within the range of applicability of
Theorem \ref{theorem-breuss-1}.
\vspace{1.0ex}

\begin{Theorem}[{\rm $L_\infty$-Stability}]
\label{theorem-breuss-2}
Let an implicit method of the form (\ref{breuss-15})
be given, which is also conservative and monotone.
Then the numerical solution is $L_\infty$-stable
over any finite time interval $[0, T]$.\vspace{0.5ex}\\
\end{Theorem}
The following Definition is useful for
proving convergence towards the entropy solution.
\vspace{1.0ex}

\begin{Definition}[{\rm Consistency with the entropy condition}]
\label{definition-breuss-2}
An implicit numerical scheme of type (\ref{breuss-15})
is consistent
with the entropy condition of Kru\v{z}kov, if there exist
for all $l= 1, \, \ldots, \, d$
numerical entropy fluxes $G_l$ which satisfy for all
$k \in \mathbb{R}$ the following assertions: 
\begin{enumerate}
\item
Consistency with the entropy flux of Kru\v{z}kov
\begin{equation}
G_l(v,v;k) = F_l(v;k) \textrm{ } \forall v 
\;
\textrm{with} 
\;
F_l(v;k)=\mathrm{sgn} (v-k)
\left[ f_l(v)-f_l(k) \right] 
\,.
\label{breuss-20}
\end{equation}
\item
Validity of a discrete entropy inequality
\begin{eqnarray}
\lefteqn{
\frac{U \left( u_j^{n+1};k \right)
-U \left( u_j^n;k \right)}{\Delta t}}\nonumber\\
& \leq &
-
\sum_{l=1}^d
\frac{G_l \left( u_j^{n+1},u_{j+\delta l}^{n+1};k \right)
-G_l \left( u_{j-\delta l}^{n+1},u_j^{n+1};k \right)}{\Delta x_l}
\nonumber\\
&&
+
\mathrm{sgn}
\left[
u_j^{n+1} - k
\right]
q_j^{n+1}\,,
\label{breuss-21}
\end{eqnarray}
where $U(v;k)= \left| v-k \right|$ is chosen
due to Kru\v{z}kov.\vspace{0.5ex}\\
\end{enumerate}
\end{Definition}
In the sequel, we define
\begin{equation}
a \vee b := \max(a,b) \quad \textrm{and} \quad
a \wedge b := \min(a,b)\,.
\label{breuss-22}
\end{equation}
The important connection 
between the numerical entropy fluxes $G_l$
and the numerical flux functions $g_l$ is now established which are
based on a variation of a procedure employed by Crandall
and Majda \cite{cramaj80}.
\vspace{1.0ex}

\begin{Lemma}\label{lemma-breuss-1}
Let a consistent and conservative numerical scheme 
of type (\ref{breuss-15}) be given
with numerical flux functions $g_l$, $l=1, \, \ldots, \, d$.
Then the numerical entropy fluxes defined by
\begin{equation}
G_l(v,w;k) \; := \;
g_l(v \vee k, w \vee k)-g_l(v \wedge k, w \wedge k)
\label{breuss-23}
\end{equation}
are consistent with the entropy 
fluxes of Kru\v{z}kov.\vspace{0.5ex}\\
\end{Lemma}
One can now prove the following result, partly by a
variation of the procedure given in \cite{cramaj80}.
We introduce the source term within the proof.
\vspace{0.5ex}
\begin{Theorem}\label{theorem-breuss-3}
Let an implicit scheme of the form (\ref{breuss-15})
be given, which is also 
consistent, conservative, and monotone.
Then the scheme is also consistent
with the entropy condition of Kru\v{z}kov.\vspace{1.0ex}\\
\end{Theorem}
Under the same assumptions, we prove convergence of the
corresponding numerical approximation
to the entropy solution. We want to do this later
when we concretely investigate numerical schemes. 

\subsection{Proofs}
\label{breuss-subsection-3-2}

We first want to prove Theorem \ref{theorem-breuss-1}.
The idea of the proof can be sketched as follows.
Let two sequences $w^n$ and $w^{n+1}$ be
given, where $w^{n+1}$ results
from an application of a considered method on $w^n$.
Then, a positive change in a given value $w_j^n$
inspires a positive change in $w_j^{n+1}$.
Secondly, a positive change in $w_j^{n+1}$ 
inspires positive changes in $w_{j \pm \delta l}^{n+1}$
for all $l$, thus creating no oscillations.
Thirdly, concerning an arbitrary index $i$, positive
changes in $w_{i \pm \delta l}^{n+1}$ result in positive
changes in $w_i^{n+1}$. Since the index $j$ used in the second
argument is chosen arbitrarily,
this is the same argument as the third one for
$j \in \left\{ i \pm \delta l; \, l=1, \, \ldots, \, d \right\}$. 
If and only if these conditions are fulfilled by 
a considered method, the method is monotone.

In order to give the proof of
Theorem \ref{theorem-breuss-1}
a convenient structure, we first give the
following Lemma.\vspace{1.0ex}

\begin{Lemma}\label{lemma-breuss-2}
Let a consistent and conservative implicit
method of the form (\ref{breuss-15})
be given, which satisfies the
conditions (\ref{breuss-18}) and (\ref{breuss-19}).
Furthermore, let two sequences 
$v^n = \left\{ v_j^n \right\}_{j \in J}$
and $w^n = \left\{ w_j^n \right\}_{j \in J}$
be given. Then from
\begin{displaymath}
\exists \, i \in J: \; v^n_i >w^n_i
\quad \textrm{and}
\quad \forall \, j \in \mathbf{J} \: (j \neq i):
\; v^n_j = w^n_j
\end{displaymath}
follows $v^{n+1} \geq w^{n+1}$ 
in the sense of the comparison of components.\vspace{0.5ex}
\end{Lemma}
\begin{proof}(of Lemma \ref{lemma-breuss-2})\\
By the assumption of the Lemma there exists
an index $i \in J$ so that $v_i^n > w_i^n$ holds.
Without restriction of generality we choose $i=0$.
The proof of the assertion follows by induction
over suitable subsets of $J$.
\vspace{1.0ex}\\
Let us introduce these subsets. Therefore, let
$J_m$ denote a subset of $J$ containing
$m$ elements with
\begin{displaymath}
\forall \, m_0 \in J_m \, 
\exists \, m_1 \in J_m \,(m_0 \neq m_1):
\;
\Bigl[
\left\{ 
m_0
\right\} 
\cap
\left\{
p \in J_m \; ; \;
p = m_1 \pm \delta l
\textrm{, }
l = 1, \, \ldots, \, d
\right\}
\Bigr]
\neq \emptyset
\end{displaymath}
for $m \geq 2$, thus the elements of $J_m$ are
indices of neighboring points.
\vspace{1.0ex}\\
{\it Beginning of the induction:} $m=1$\\
As indicated, we choose without restriction of
generality $J_1= \left\{ 0 \right\}$.
The statement is true because of the
form of the method (\ref{breuss-15}), so that
\begin{eqnarray}
\lefteqn{
H \left( 
l \in \underline{d}, \,
w_{-\delta l}^{n+1}, \,
w_0^{n+1}, \,
w_{\delta l}^{n+1}, \,
s + \Delta s \right)
} \nonumber\\
& \geq &
H \left(
l \in \underline{d}, \,
w_{-\delta l}^{n+1}, \,
w_0^{n+1}, \,
w_{\delta l}^{n+1}, \,
s \right)
\quad \forall \, \Delta s \geq 0
\quad
\textrm{holds.}
\nonumber
\end{eqnarray}
\vspace{0.5ex}\\
{\it Assumption:}\\
The statement is true for arbitrary but fixed $m>1$.
\vspace{0.5ex}\\
{\it Induction step:} $m \mapsto m+1$\\
Let the statement be true for the subsets
$\left\{ v_i^{n+1} \right\}_{i \in J_m}$ and
$\left\{ w_i^{n+1} \right\}_{i \in J_m}$
of the sequences $v^{n+1}$ and $w^{n+1}$. 
In particular, it holds
\begin{eqnarray}
&&
v_{\tilde m}^{n+1} \geq w_{\tilde m}^{n+1}
\quad
\textrm{for an index}
\quad
\tilde m \in J_m
\nonumber\\
& \textrm{with} &
\Bigl[
\left\{ i \in J ; \, i= \tilde m \pm \delta l \textrm{, }
l=1,\, \ldots, \, d \right\} 
\cap
\left( J \setminus J_m \right)
\Bigr]
\;
\neq 
\;
\emptyset
\nonumber
\end{eqnarray}
which is otherwise chosen arbitrarily, i.e.
we consider an index $\tilde m$ corresponding
to a grid point with at least one
neighbor having an index not in $J_m$.

Without restriction on generality, let us choose
a particular index $l_m$ corresponding to the
situation
\begin{displaymath}
\tilde m \in J_m
\quad
\textrm{and}
\quad
\tilde m + \delta l_m \notin J_m
\,.
\end{displaymath}
Since by construction the sequences $v^{n+1}$ and $w^{n+1}$
are identical outside the considered subsets, it holds
\begin{displaymath}
\tilde H_{l_m} 
\left( 
v_{\tilde m}^{n+1}, \,
w_{\tilde m + \delta l_m}^{n+1}, \,
w_{\tilde m + 2\delta l_m}^{n+1}
\right)
\;
\geq
\;
\tilde H_{l_m} 
\left( 
w_{\tilde m}^{n+1}, \,
w_{\tilde m + \delta l_m}^{n+1}, \,
w_{\tilde m + 2\delta l_m}^{n+1}
\right)
\end{displaymath}
by (\ref{breuss-18}). If the index 
$\tilde m + 2 \delta l_m$ is already in $J_m$, we
estimate
\begin{displaymath}
\tilde H_{l_m} 
\left( 
v_{\tilde m}^{n+1}, \,
w_{\tilde m + \delta l_m}^{n+1}, \,
v_{\tilde m + 2\delta l_m}^{n+1}
\right)
\geq 
\tilde H_{l_m} 
\left( 
w_{\tilde m}^{n+1}, \,
w_{\tilde m + \delta l_m}^{n+1}, \,
w_{\tilde m + 2\delta l_m}^{n+1}
\right)
\end{displaymath}
by also using (\ref{breuss-19}).
The case $\tilde m \in J_m$
and $\tilde m - \delta l_m \notin J_m$
can be handled analogously.
 
By defining 
\begin{displaymath}
J_{m+1}:= J_m \cup \left\{ \tilde m +\delta l_m \right\}
\quad
\textrm{or}
\quad
J_{m+1}:= J_m \cup \left\{ \tilde m -\delta l_m \right\}
\end{displaymath}
corresponding to the situation under consideration,
it follows $v_i^{n+1} \geq w_i^{n+1}$
for all $i \in J_{m+1}$.
Since $\tilde m$ and $l_m$ were chosen arbitrarily
within the framework of the construction,
the procedure is well-defined and the proof is finished.
\end{proof}

\begin{proof}(of Theorem \ref{theorem-breuss-1})\\
Let again two sequences $v^n,w^n$ be given, which are mapped
on sequences $v^{n+1}$ and $w^{n+1}$
by application of the considered 
consistent and conservative numerical method,
respectively. 
\vspace{1.0ex}\\
"$\Rightarrow$":\\
Let the method be monotone in the sense of 
Definition \ref{definition-breuss-1}.
Let $v^n \geq w^n$ hold in the sense of
comparison of components.
By the assumed monotonicity of the scheme
follows $v^{n+1} \geq w^{n+1}$.
It remains to verify the validity of the conditions
(\ref{breuss-18}) and (\ref{breuss-19}).
\vspace{0.5ex}\\
{\it To condition (\ref{breuss-18}):}\\
Let $l \in \left\{ 1, \, \ldots, \, d \right\}$
be chosen arbitrarily but fixed. Accordingly, let
an arbitrarily chosen but fixed index $i$ 
and a corresponding set of values
\begin{displaymath}
\left\{ a,b,c \right\}
\subset
w^{n+1}
\quad
\textrm{be given with}
\quad
\left( w_{i-\delta l}^{n+1}, \,
w_i^{n+1}, \,
w_{i+\delta l}^{n+1} \right)
=
\left( a,b,c \right)
\,.
\end{displaymath}
Assume that for $\Delta a \geq 0$ it does not hold in general
\begin{displaymath}
\tilde H_l \left( a+\Delta a,b,c \right)
\quad \geq \quad 
\tilde H_l \left(a,b,c \right)
\,.
\end{displaymath}
Then there exist two tuples $(a_1,b_1,c)$
and $(a_2,b_2,c)$ with
$a_1 > a_2$ and
\begin{equation}
\tilde H_l \left( a_1,b_1,c \right)
<
\tilde H_l \left( a_2,b_2,c \right)
\,.
\label{breuss-24}
\end{equation}
Since we investigate the general situation, we may well
assume equality of the remainder of the sequences under 
consideration, thus the only resulting change by application
of the method originates from (\ref{breuss-24}).
By (\ref{breuss-15})
it follows that $b_1<b_2$ has in general to be valid.
On the other hand there is $(a_1,b_1) \geq (a_2,b_2)$
in the sense of comparison of components by the assumed monotonicity
of the method, and so the assumption
is wrong and the validity of (\ref{breuss-18}) is
verified.\vspace{0.5ex}\\
{\it To condition (\ref{breuss-19}):}\\
The proof can be done analogously.
\vspace{1.0ex}\\
"$\Leftarrow$":\\
Next, the validity of the monotonicity condition
(\ref{breuss-16}) under the assumptions
(\ref{breuss-18}) and (\ref{breuss-19})
is proven.
Therefore, we define the set
\begin{displaymath}
\hat J^n
:=
\left\{
i \in J \; ; \;
v_i^n > w_i^n, \,
v_i^n \in v^n, \,
w_i^n \in w^n
\right\}
\,.
\end{displaymath}
There are only a few possibilities for the
composition of $\hat J^n$:
It may consist of the empty set or a finite or
infinite subset of the index set $J$ containing the
indices of all spatial grid points.
Since we have to take into account all these cases,
we define
\begin{displaymath}
\hat J_m^n
:=
\left\{
\hat J^n
\; ; \;
\sharp \left( \hat J^n \right) = m
\right\}
\,.
\end{displaymath}
The proof of the assertion follows by 
induction over $m \geq 1$ concerning these sets. Note that the case $m=0$ is trivial.
\vspace{1.0ex}\\
{\it Beginning of the induction:} $\hat J^n= \hat J^n_1$.\\
Let $i$ be the index in the arbitrarily chosen but fixed index
set $\hat J^n_1$. Then the validity of the monotonicity
condition follows by application
of Lemma \ref{lemma-breuss-2}.
\vspace{1.0ex}\\
{\it Assumption:} The assertion holds for all subsets of
$\hat J^n = \hat J^n_m$ for an arbitrarily chosen but fixed
number $m>1$.
\vspace{1.0ex}\\
{\it Induction step:} $m \mapsto m+1$\\ 
Now we consider
$\hat J^n_{m+1}$ with
$\hat J_m^n \subset \hat J^n_{m+1}$.
We define two particular indices $m_1$, $m_2$ with
\begin{displaymath}
m_1 \in \hat J_m^n 
\quad
\textrm{and}
\quad
m_2 \in 
\left( \hat J^n_{m+1} \setminus \hat J^n_m \right)
\,.
\end{displaymath}
Thereby, the index $m_1$ is chosen arbitrarily but fixed.
By the assumption of the induction,
the scheme is monotone with respect to positive
changes in values corresponding to
the index set $\hat J^n_m$.
This means in particular that a positive change
in $v_{m_1}^n$ together with positive changes in
other values corresponding to $\hat J^n_m$
leads to non-negative changes in the sequence $v^{n+1}$.

Now a simultaneous positive change in
in $v_{m_1}^n$ and $v_{m_2}^n$ is considered
while in the background there are arbitrary but fixed
positive changes in the values corresponding to
$\hat J^n_{m+1} \setminus \left\{ m_1, \, m_2 \right\}$.

Let the data resulting from positive changes in 
$v_i^n$, $i \in \hat J^n_m \setminus \left\{ m_1, \, m_2 \right\}$,
be denoted by $\bar v^{n+1}$,
i.e. $\bar v^{n+1} \geq w^{n+1}$ holds by the
assumption of the induction step.

Moreover, let $\Delta_j^1$ be a change in $\bar v_j^{n+1}$
induced by a positive change in $v_{m_1}^n$.
Thus $\Delta_j^1$ is always non-negative
by the assumption of the induction.
Analogously, let $\Delta_j^2$
a change in $\bar v_j^{n+1}$ induced by 
a positive change in $v_{m_2}^n$.
The change $\Delta_j^2$ is also non-negative
which follows analogously to the proof
of Lemma \ref{lemma-breuss-2}.

There are two possibilities to investigate
for the mutual effects of such changes in
data corresponding to an
arbitrary but fixed index $\tilde i$ and 
an accordingly arranged index 
$l_i \in \left\{ 1, \, \ldots, \, d \right\}$:
\begin{displaymath}
\tilde H_{l_i}
\left(
\bar v_{\tilde i-\delta l_i}^{n+1}+
\Delta_{\tilde i-\delta l_i}^1, \,
\bar v_{\tilde i}^{n+1}, \,
\bar v_{\tilde i+\delta l_i}^{n+1}+
\Delta_{\tilde i+\delta l_i}^2 
\right)
\stackrel{(\ref{breuss-18}), (\ref{breuss-19})}{\geq}
\tilde H_{l_i}
\left(
\bar v_{\tilde i-\delta l_i}^{n+1}, \,
\bar v_{\tilde i}^{n+1}, \,
\bar v_{\tilde i+\delta l_i}^{n+1} 
\right)
\end{displaymath}
and
\begin{displaymath}
\tilde H_{l_i}
\left(
\bar v_{\tilde i-\delta l_i}^{n+1}+
\Delta_{\tilde i-\delta l_i}^2, \,
\bar v_{\tilde i}^{n+1}, \,
\bar v_{\tilde i+\delta l_i}^{n+1}+
\Delta_{\tilde i+\delta l_i}^1
\right)
\stackrel{(\ref{breuss-18}), (\ref{breuss-19})}{\geq}
\tilde H_{l_i}
\left(
\bar v_{\tilde i-\delta l_i}^{n+1}, \,
\bar v_{\tilde i}^{n+1}, \,
\bar v_{\tilde i+\delta l_i}^{n+1}
\right)
\,.
\end{displaymath}
Note the arbitrary choice of $m_1$ and $m_2$
by a simultaneous change in the data 
corresponding to the index set 
$\hat J^n_m \setminus \left\{ m_1, \, m_2 \right\}$.
Since there are also no limitations
concerning the choices of $\hat J^n_m$ and $l_i$,
the procedure is well defined
and the proof is finished.
\end{proof}
\vspace{0.5ex}
\begin{proof}(of Theorem \ref{theorem-breuss-2})\\
Let a sequence $u^0 \in L_\infty$ be given.
We then identify the finite values
\begin{displaymath}
a := \inf_{j \in J} u_j^0
\quad
\textrm{and}
\quad
b := \sup_{j \in J} u_j^0
\,.
\end{displaymath}
Since the source terms are pointwise bounded over
the time interval $(0,T)$ --- see assumptions
(\ref{breuss-6}) and (\ref{breuss-12}), respectively ---
they are in both scenarios of interest
especially bounded by a finite number $M$ with
\begin{displaymath}
\int_0^T \| q \|_\infty \, \mathrm{d}t 
\, < \, M
\,.
\end{displaymath}
Consequently, by the assumed monotonicity follows
that the numerical solution obtained via given data
$u_0$ is bounded for all $n$ with $n \Delta t < T$
by $a^n \leq u^n \leq b^n$ with
\begin{displaymath}
a^n_j := a-M \, (>-\infty) \; \forall j \in J
\quad
\textrm{and}
\quad
b^n_j := b+M \, (<\infty) \; \forall j \in J
\,.
\end{displaymath}
\end{proof}
\vspace{0.5ex}
\begin{proof}(of Lemma \ref{lemma-breuss-1})\\
Because the numerical scheme is consistent
and conservative, the statement
\begin{displaymath}
G_l(v,v;k) =
g_l(v \vee k, v \vee k)-g_l(v \wedge k, v \wedge k) =
\mathrm{sgn}(v-k)[f_l(v)-f_l(k)]
\end{displaymath}
holds by (\ref{breuss-22})
for all $l=1, \, \ldots, \, d$ and all $k \in \mathbb{R}$.
\end{proof}
\vspace{1.0ex}

\begin{proof}(of Theorem \ref{theorem-breuss-3})\\
Since the method is assumed to be consistent and conservative,
there exist numerical flux functions $g_l$,
$l=1, \, \ldots, \, d$, so that
one can construct numerical entropy fluxes $G_l$
by applying Lemma \ref{lemma-breuss-1}.
Thereby, the consistency with the entropy fluxes
due to Kru\v{z}kov is given.
It is left to show the validity of
a discrete entropy inequality.
Therefore, let $k \in \mathbb{R}$ be chosen
arbitrarily but fixed. By using the definition of
$G_l$, we derive 
\begin{eqnarray}
\lefteqn{ -
\sum_{l=1}^d
\frac{\Delta t}{\Delta x_l}
\biggl\{
G_l \left(
u_j^{n+1}, \, u_{j+\delta l}^{n+1}; k
\right)
-
G_l \left(
u_{j-\delta l}^{n+1}, \, u_j^{n+1};k
\right)
\biggr\}
}\nonumber\\ 
& = &
H
\left( 
l \in \underline{d}, \,
u_{j-\delta l}^{n+1} \vee k, \,
u_j^{n+1} \vee k, \,
u_{j+\delta l}^{n+1} \vee k, \,
u_j^n \vee k 
\right) \nonumber\\
&&
- 
H
\left(
l \in \underline{d}, \,
u_{j-\delta l}^{n+1} \wedge k, \,
u_j^{n+1} \wedge k, \,
u_{j+\delta l}^{n+1} \wedge k, \,
u_j^n \wedge k
\right)
-
\left| u_j^n - k \right|
\,. \label{breuss-25}
\end{eqnarray}
Now we estimate the terms involving $H$
by using the monotonicity properties of the method.
It is necessary to employ a diversion of the cases
$u_j^{n+1} \geq k$ and $u_j^{n+1}<k$.
\vspace{1.0ex}\\
(a) {\it Case } $u_j^{n+1} \geq k$:
\begin{eqnarray}
\lefteqn{
H
\left( 
l \in \underline{d}, \,
u_{j-\delta l}^{n+1} \vee k, \,
u_j^{n+1} \vee k, \,
u_{j+\delta l}^{n+1} \vee k, \,
u_j^n \vee k 
\right)
}\nonumber\\
& \stackrel{\textrm{(a)}}{=} &
u_j^n \vee k
-
\sum_{l=1}^d
\frac{\Delta t}{\Delta x_l}
\left\{
g_l \left( u_j^{n+1}, \, u_{j + \delta l}^{n+1} \vee k \right)
-
g_l \left( u_{j-\delta l}^{n+1} \vee k, \, u_j^{n+1} \right)
\right\}
+
\Delta t
q_j^{n+1}
\nonumber\\
& \geq &
u_j^n
-
\sum_{l=1}^d
\frac{\Delta t}{\Delta x_l}
\left\{
g_l \left( u_j^{n+1}, \, u_{j + \delta l}^{n+1} \right)
-
g_l \left( u_{j-\delta l}^{n+1}, \, u_j^{n+1} \right)
\right\}
+
\Delta t
q_j^{n+1}
\nonumber\\
& = &
u_j^{n+1} 
\quad
\stackrel{\textrm{(a)}}{=} 
\quad
u_j^{n+1}\vee k
\,.
\nonumber
\end{eqnarray}
(b) {\it Case } $u_j^{n+1} < k$:
\begin{eqnarray}
\lefteqn{
H
\left( 
l \in \underline{d}, \,
u_{j-\delta l}^{n+1} \vee k, \,
u_j^{n+1} \vee k, \,
u_{j+\delta l}^{n+1} \vee k, \,
u_j^n \vee k 
\right)
}\nonumber\\
& \stackrel{\textrm{(b)}}{=} &
u_j^n \vee k
-
\sum_{l=1}^d
\frac{\Delta t}{\Delta x_l}
\left\{
g_l \left( k, \, u_{j + \delta l}^{n+1} \vee k \right)
-
g_l \left( u_{j-\delta l}^{n+1} \vee k, \, k \right)
\right\}
+
\Delta t
q_j^{n+1}
\nonumber\\
& \geq &
k
-
\sum_{l=1}^d
\frac{\Delta t}{\Delta x_l}
\left\{
g_l \left( k, \, k \right)
-
g_l \left( k, \, k \right)
\right\}
+
\Delta t
q_j^{n+1}
\nonumber\\
& = &
k+ \Delta t q_j^{n+1}
\quad
\stackrel{\textrm{(b)}}{=} 
\quad
u_j^{n+1}\vee k
+
\Delta t
q_j^{n+1}
\,.
\nonumber
\end{eqnarray}
(c) {\it Case } $u_j^{n+1} \geq k$:
\begin{eqnarray}
\lefteqn{
H
\left( 
l \in \underline{d}, \,
u_{j-\delta l}^{n+1} \wedge k, \,
u_j^{n+1} \wedge k, \,
u_{j+\delta l}^{n+1} \wedge k, \,
u_j^n \wedge k 
\right)
}\nonumber\\
& \stackrel{\textrm{(c)}}{=} &
u_j^n \wedge k
-
\sum_{l=1}^d
\frac{\Delta t}{\Delta x_l}
\left\{
g_l \left( k, \, u_{j + \delta l}^{n+1} \wedge k \right)
-
g_l \left( u_{j-\delta l}^{n+1} \wedge k, \, k \right)
\right\}
+
\Delta t
q_j^{n+1}
\nonumber\\
& \leq &
k
-
\sum_{l=1}^d
\frac{\Delta t}{\Delta x_l}
\left\{
g_l \left( k, \, k \right)
-
g_l \left( k, \, k \right)
\right\}
+
\Delta t
q_j^{n+1}
\nonumber\\
& = &
k + \Delta t q_j^{n+1}
\quad
\stackrel{\textrm{(c)}}{=} 
\quad
u_j^{n+1}\wedge k
+
\Delta t
q_j^{n+1}
\,.
\nonumber
\end{eqnarray}
(d) {\it Case } $u_j^{n+1} < k$:
\begin{eqnarray}
\lefteqn{
H
\left( 
l \in \underline{d}, \,
u_{j-\delta l}^{n+1} \wedge k, \,
u_j^{n+1} \wedge k, \,
u_{j+\delta l}^{n+1} \wedge k, \,
u_j^n \wedge k 
\right)
}\nonumber\\
& \stackrel{\textrm{(d)}}{=} &
u_j^n \wedge k
-
\sum_{l=1}^d
\frac{\Delta t}{\Delta x_l}
\left\{
g_l \left( u_j^{n+1}, \,
u_{j + \delta l}^{n+1} \wedge k \right)
-
g_l \left( u_{j-\delta l}^{n+1} \wedge k, \,
u_j^{n+1} \right)
\right\}
+
\Delta t
q_j^{n+1}
\nonumber\\
& \leq &
u_j^n
-
\sum_{l=1}^d
\frac{\Delta t}{\Delta x_l}
\left\{
g_l \left( u_j^{n+1}, \, u_{j + \delta l}^{n+1} \right)
-
g_l \left( u_{j-\delta l}^{n+1}, \, u_j^{n+1} \right)
\right\}
+
\Delta t
q_j^{n+1}
\nonumber\\
& = &
u_j^{n+1} 
\quad
\stackrel{\textrm{(d)}}{=} 
\quad
u_j^{n+1}\wedge k
\,.
\nonumber
\end{eqnarray}
By combining all these cases, we 
obtain from (\ref{breuss-25}) the inequality
\begin{eqnarray}
&&
-
\sum_{l=1}^d
\frac{\Delta t}{\Delta x_l}
\biggl\{
G_l \left(
u_j^{n+1}, \, u_{j+\delta l}^{n+1}; k
\right)
-
G_l \left(
u_{j-\delta l}^{n+1}, \, u_j^{n+1};k
\right)
\biggr\}
+
\mathrm{sgn}
\left[
u_j^{n+1} - k
\right]
\Delta t
q_j^{n+1}
\nonumber\\
& \geq &
u_j^{n+1} \vee k - u_j^{n+1} \wedge k
-
\mathrm{sgn}
\left[
u_j^{n+1}
-
k
\right]
\Delta t
q_j^{n+1}
\nonumber\\
&&
+
\mathrm{sgn}
\left[
u_j^{n+1} - k
\right]
\Delta t
q_j^{n+1}
-
\left|
u_j^n - k
\right|
\nonumber\\
& = &
\left|
u_j^{n+1}
-
k
\right|
-
\left|
u_j^n
-
k
\right|
\,.
\nonumber
\end{eqnarray}
By construction, the procedure is well defined.
Division by $\Delta t$ gives the desired 
discrete entropy inequality.
\end{proof}

In the case of {\it Scenario 2}, the validity
of the corresponding discrete entropy inequality
can be proven in the same way, resulting essentially
from the monotonicity of the method.
The difference between {\it Scenario 1} and {\it Scenario 2}
is made up by substituting
\begin{displaymath}
\sum_{l=1}^d
\mathrm{sgn}
\left[
u_j^{n+1} - k
\right]
\left[
q_j^{n+1}
-
f_{l_{x_l}}(j, n+1)
\right]
\quad \textrm{for } \quad
\mathrm{sgn}
\left[
u_j^{n+1} - k
\right]
q_j^{n+1}
\end{displaymath}
\begin{displaymath}
\textrm{with} \quad
f_{l_{x_l}}(j, n+1)
:=
f_{l_{x_l}} \left( \tilde J (j), (n+1)\Delta t, u_j^{n+1} \right)
\,.
\end{displaymath}

\section{Implicit numerical methods}
\label{breuss-section-4}

This section contains the theoretical investigation of 
a few selected implicit methods. These are: (1) An implicit
upwind scheme, (2) an implicit version of the Lax-Friedrichs
scheme and (3) an implicit Godunov-type method.

\subsection{An implicit upwind method}
\label{breuss-subsection-4-1}

The implicit formulation of the upwind method reads
\begin{equation}
u_j^{n+1}
=
u_j^n
-
\sum_{l=1}^d
\frac{\Delta t}{\Delta x_l}
\left\{
f_l \left( u_j^{n+1} \right)
-
f_l \left( u_{j-\delta l}^{n+1} \right)
\right\}
+
\Delta t
q_j^{n+1}
\,.
\label{breuss-26}
\end{equation}
We now employ the developed implicit notion of
monotonicity.\\ 
{\it To condition (\ref{breuss-18}):}
\begin{eqnarray}
\lefteqn{
\tilde H_l
\left( 
a + \Delta a, \,
b, \,
c    
\right)
-
\tilde H_l
\left( 
a, \,
b, \,
c 
\right)
}\nonumber\\
& = &
\left[
-
\frac{\Delta t}{\Delta x_l}
\left[ f_l(b)- f_l(a + \Delta a) \right]
\right]
-
\left[
-
\frac{\Delta t}{\Delta x_l}
\left[ f_l(b)- f_l(a) \right]
\right] \nonumber\\
& = &
\frac{\Delta t}{\Delta x_l}
\left[
f_l(a + \Delta a)- f_l(a)
\right] 
\,.
\nonumber
\end{eqnarray}
The condition (\ref{breuss-18}) is fulfilled if
$f_l$ grows monotonically for all $l = 1, \, \ldots, \, d$.\\
{\it To condition (\ref{breuss-19}):}
\begin{eqnarray}
\lefteqn{
\tilde H_l 
\left( a,b,c + \Delta c
\right)
-
\tilde 
H_l
\left( a,b,c \right)} \nonumber\\
& = &
\left[
-
\frac{\Delta t}{\Delta x_l}
\left[ f_l(b)- f_l(a) \right]
\right]
-
\left[
- \frac{\Delta t}{\Delta x_l}
\left[ f_l(b) - f_l(a) \right] 
\right]
= 
0 \; \left( \geq 0 \right)
\,.
\nonumber
\end{eqnarray}
Thus, the condition (\ref{breuss-19}) is always fulfilled
and the implicit upwind scheme is monotone 
if all the fluxes $f_l$ grow
monotonically. This is a nice property of the developed
notions, since also the implicit scheme
respects the direction of the flow.
Note that the $f_l$ do not need
to be Lipschitz continuous to ensure the monotonicity
of the scheme.

\subsection{The implicit Lax-Friedrichs method}
\label{breuss-subsection-4-2}

We investigate the
implicit Lax-Friedrichs scheme
\begin{displaymath}
u_j^{n+1} = 
u_j^n
+
\sum_{l=1}^d
\left\{
\frac{1}{2}
\left[
u_{j-\delta l}^{n+1}-2 u_j^{n+1}+u_{j+\delta l}^{n+1}
\right]
- \frac{\Delta t}{2 \Delta x_l}
\left[
f_l(u_{j+\delta l}^{n+1})-f_l(u_{j-\delta l}^{n+1})
\right]
\right\}
\,.
\end{displaymath}
{\it To condition (\ref{breuss-18}):}
\begin{equation}
\tilde H_l(a+\Delta a,b,c)- \tilde H_l(a,b,c)
=
\frac{1}{2}
\Delta a 
+ \frac{\Delta t}{2 \Delta x_l}
\left[
f_l(a+\Delta a)-f_l(a)
\right]
\,.
\label{breuss-27}
\end{equation}
This expression is not positive or equal to zero
without additional requirements.\\
{\it To condition (\ref{breuss-19}):}
\begin{equation}
\tilde H_l(a,b,c+\Delta c)- \tilde H_l(a,b,c)
=
\frac{1}{2}
\Delta c 
- \frac{\Delta t}{2 \Delta x_l}
\left[
f_l(c+\Delta c)-f_l(c)
\right]
\,.
\label{breuss-28}
\end{equation}
Again this expression is not automatically positive
or equal to zero. The requirements
(\ref{breuss-27}) and (\ref{breuss-28})
can be combined to 
\begin{displaymath}
\frac{
\left|
f_l(x+\Delta x)-f_l(x)
\right|}{\Delta x_l}
\; \leq \;
\frac{\Delta x_l}{\Delta t}
\quad
\forall \, l=1, \, \ldots, \, d
\textrm{ and }
\forall \, \Delta x \geq 0
\,.
\end{displaymath}
Therefore, the implicit Lax-Friedrichs scheme is monotone only
for Lipschitz-con\-ti\-nu\-ous flux functions 
with Lipschitz constants $L_l \leq (\Delta x_l / \Delta t)$.
Note that this can also be read as a condition on the time step
size which does not depend on the dimension,
since each single one of the $2l$ conditions 
(\ref{breuss-18}) and (\ref{breuss-19})
has to be satisfied and no coupling is involved. 
This is quite surprising (a) because it is normally suggested
that the numerical cha\-rac\-te\-ris\-tics include
the whole domain in the case of implicit methods,
and (b) since no dimensional influence on the monotonicity
property is obtained. 
In order to illuminate point (a), we briefly
review the discussion of the situation for the 
case of the linear advection equation without sources in one dimension which
is done in \cite{breuss02} in much more detail.
With respect to point (b), we demonstrate numerically a
similar behavior in two dimensions in order to illustrate
the noted missing dimensional dependence of the
implicit monotonicity criterion.

In the case of a linear flux $f(u)=vu$, the nonlinear system
defined by the implicit Lax-Friedrichs scheme
degenerates to a linear system with $\lambda=\Delta t/\Delta x$ given through
\begin{equation}
\left[ - \frac{1}{2} - v \frac{\lambda}{2} \right] u_{j-1}^{n+1}
+ 2 u_j^{n+1}
+
\left[ - \frac{1}{2} + v \frac{\lambda}{2} \right] u_{j+1}^{n+1}
=
u_j^n
\,.
\label{breuss-29}
\end{equation}
We investigate the structure of the
tridiagonal matrix $A=\left( a_{ij} \right)$
defined by (\ref{breuss-29}). Therefore, let
$v$ be positive with $v>(1/\lambda)$ so that the formal
monotonicity property of the scheme is lost.
Then the entries in the lower diagonal $a_{i+1,i}$
always take on negative values while the entries in the
upper diagonal $a_{i,i+1}$ are always positive. 

We at first eliminate the entries
in the lower diagonal $a_{i+1,i}$. The diagonal entries of the
matrix have to be modified accordingly, i.e. the diagonal
entry in the $i$-th row is modified via
\begin{displaymath}
a_{ii}^{new} = a_{ii}^{old} -
\frac{a_{i,i-1} }{a_{i-1,i-1}}a_{i-1,i}
\,.
\end{displaymath}
Thereby, note that we always have the situation
\begin{displaymath}
a_{i,i-1}<0\textrm{, }
a_{i-1,i-1}>0 \;, \textrm{ and } \;
a_{i-1,i}>0
\,,
\end{displaymath}
so that $a_{ii}^{new}>a_{ii}^{old}$ is always satisfied.
Since the right hand side $(b_i)$ of the investigated system
incorporating the given data is modified via
\begin{displaymath} 
b_i = u_i^n - \frac{a_{i,i-1} }{a_{i-1,i-1}} b_{i-1}
\,,
\end{displaymath}
data sets with $u_k^n \geq 0$ $\forall k$ imply
only positive possible changes in the values $b_i$.
In particular, the values in the upper diagonal
$a_{i,i+1}$ remain unchanged and positive.

We now investigate what happens at a jump in
given data $u_k^n$ from values $0$ to $1$ 
when backward elimination is applied in order to
solve the system. 
Therefore, we fix $u_j^n:=0 \, \forall j<i$ and 
$u_j^n:=1 \, \forall j \geq i$.
By the described procedure, it is clear that the corresponding
entries on the right hand side also show a jump from 
$0$ to $1$ after the modification due to elimination of the
lower diagonal since $b_{i-1}=u_{i-1}^n=0$, so that no positive
update in $b_i$ takes place.
Backward elimination results in
\begin{displaymath}
u_{i-1}^{n+1} = 
\underbrace{
\frac{1}{a_{i-1,i-1}^{new}}}_{>0}
\bigl(
\underbrace{u_{i-1}^n}_{=0}
-
\underbrace{a_{i-1,i}}_{>0}
\underbrace{u_i^n}_{=1}
\bigr)
\, <0
\,,
\end{displaymath}
so that the monotonicity is violated, as expected. 
The violation of the monotonicity
property can also be observed at jumps from high to lower
values within given data. 

Concerning the two-dimensional situation, we consider
the linear advection equation
\begin{displaymath}
\frac{\partial}{\partial t}
u(x,y,t)
+
\frac{\partial}{\partial x}
\left(
v u(x,y,t)
\right)
+
\frac{\partial}{\partial y}
\left(
v u(x,y,t)
\right)
= 
0
\end{displaymath}
with grid parameters $\Delta x=\Delta y=0.1$ and the
initial condition
\begin{displaymath}
u(x,y,0)=
\left\{
\begin{array}{ccc}
1 & \textrm{for} & 
(x,y)\in \left[ 0,1 \right] \times \left[0,1 \right]
\,, \\
0 & \textrm{else.}
\end{array}
\right.
\end{displaymath}
The monotonicity condition yields that the chosen
time step size $\Delta t = 0.1$ is the largest one
allowed for $v=1.0$ in order
to preserve the monotonicity of the scheme,
the same as would be in the one-dimensional case. See
Figure \ref{breuss-figure-lf1}
for a visualization of the monotone and monotonicity-violating
property of the method. Figure \ref{breuss-figure-lf2}
gives a more detailed picture of the latter case.

\begin{figure}[!ht] 
\centering
\includegraphics[height=4.5cm]{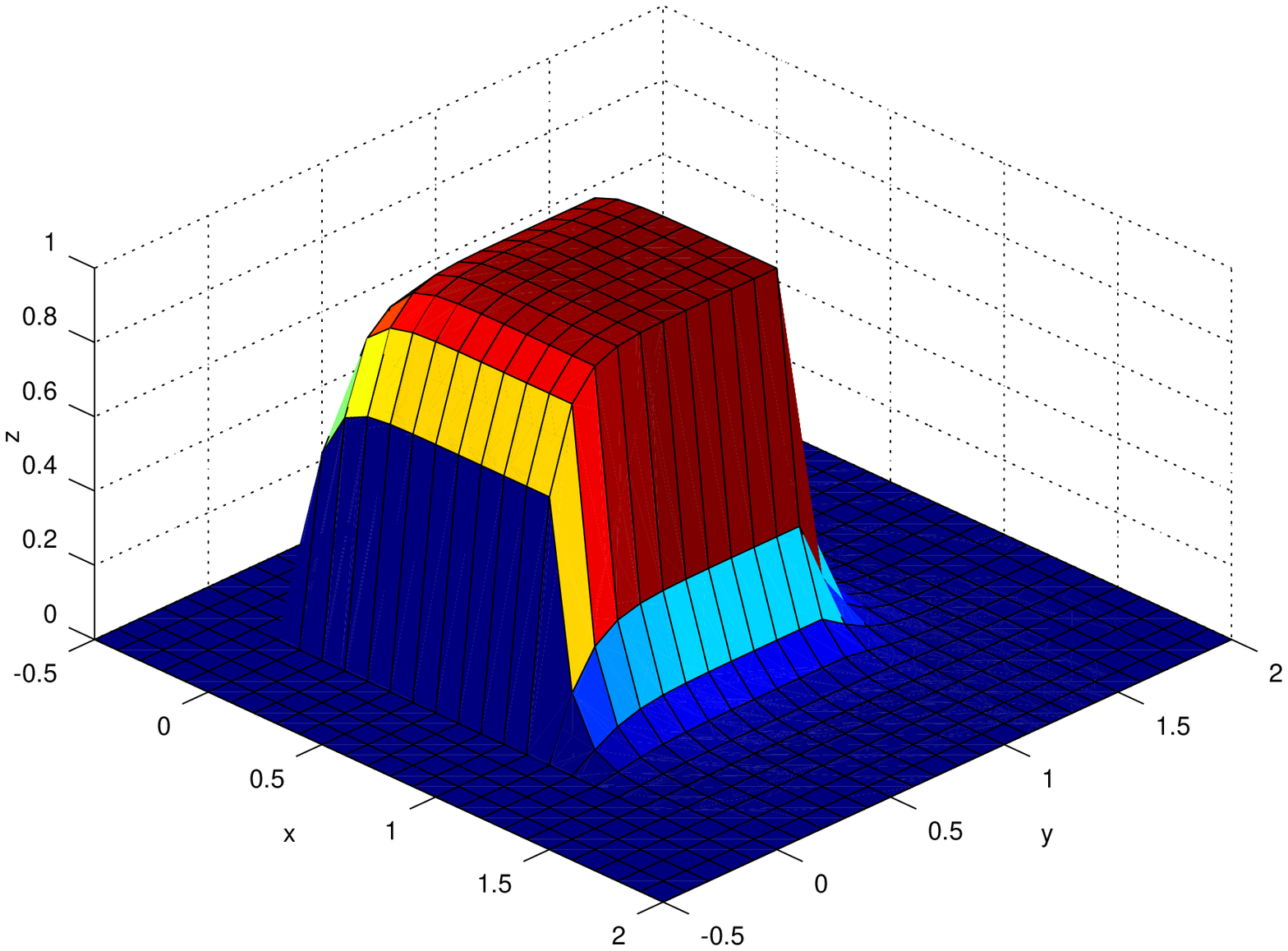}
\includegraphics[height=4.5cm]{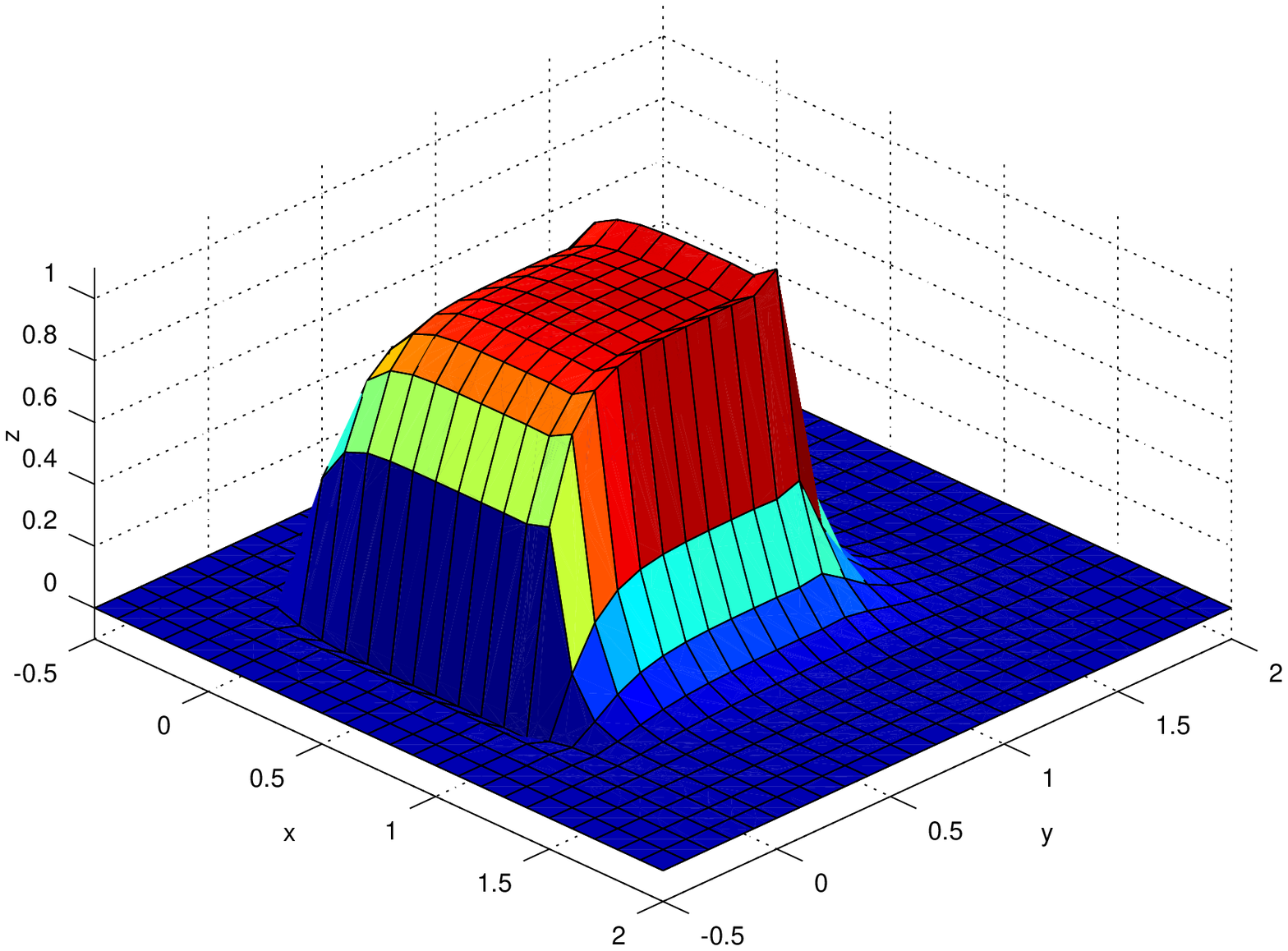}
\caption{Numerical solutions of the linear two-dimensional problem
after one time step with $v=1$ (left) satisfying
the monotonicity condition and with $v=1.5$ (right), 
resulting in a monotonicity violation as in the one-dimensional case.
The same behavior also occurs for
velocities $1<v<1.5$, resulting in
much less amplitudes of the violations.}
\label{breuss-figure-lf1}       
\end{figure}

\begin{figure}[!ht] 
\centering
\includegraphics[height=4.5cm]{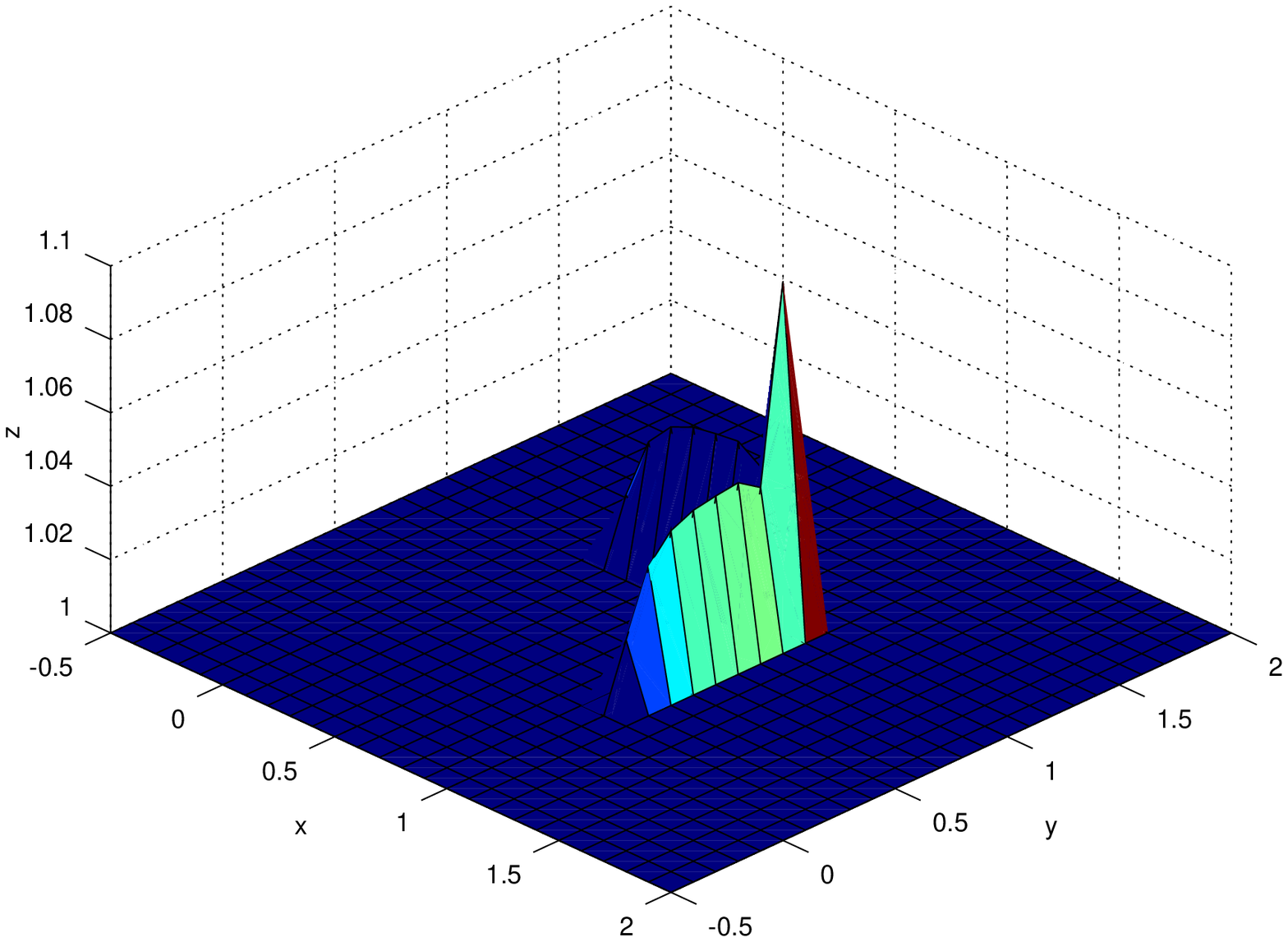}
\includegraphics[height=4.5cm]{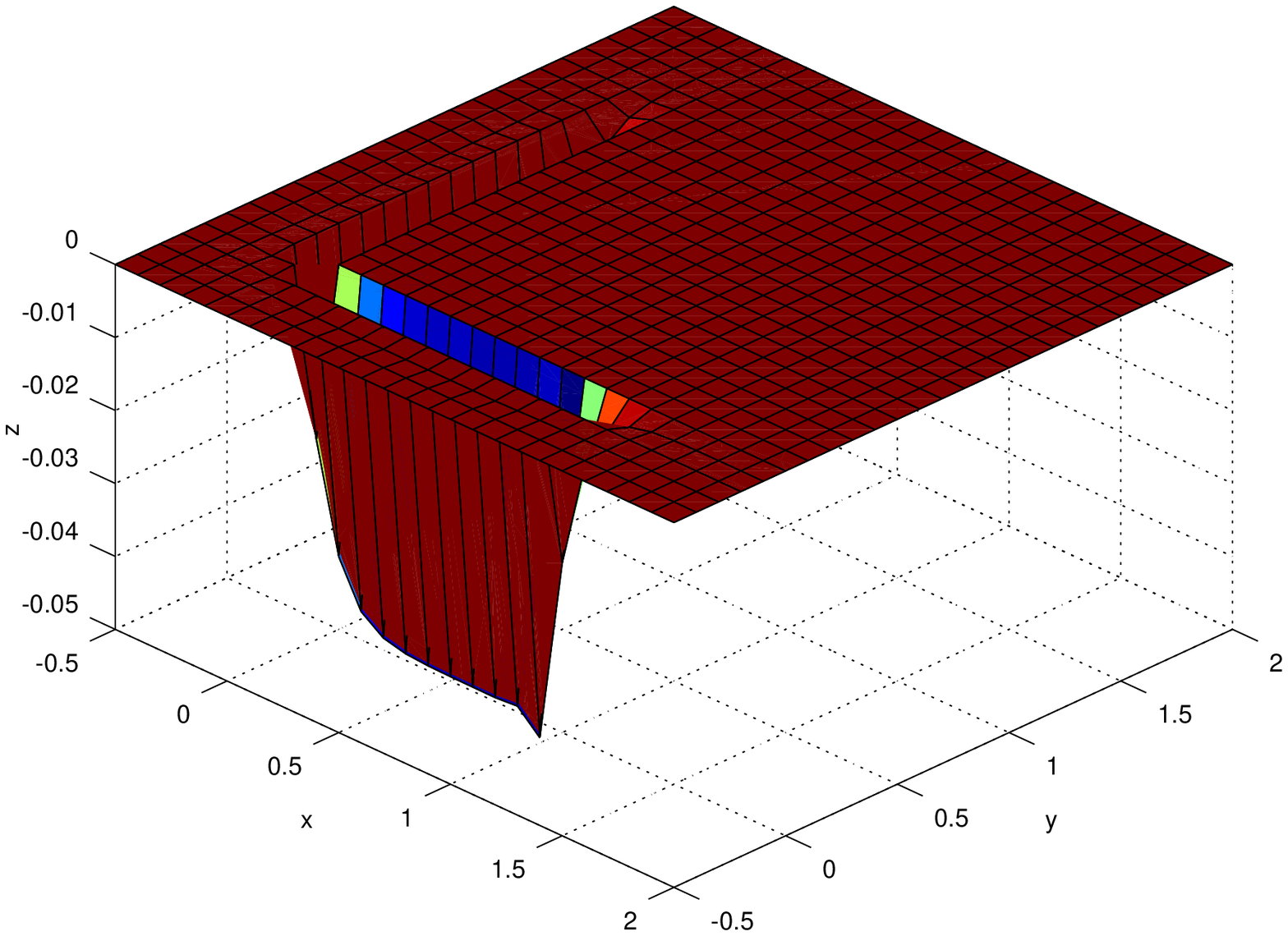}
\caption{Plots showing in detail the monotonicity violation
in the case $\Delta t=\Delta x=\Delta y=0.1$ and $v=1.5$,
obtained after one time step with the implicit Lax-Friedrichs scheme:
The maximum of $1.0$ and the numerical solution (left) and
the minimum of $0.0$ and the numerical solution (right).}
\label{breuss-figure-lf2}       
\end{figure}

\subsection{An implicit Godunov-type method}
\label{breuss-subsection-4-3}

In the scalar case, a closed form of the exact
solution of a Riemann-problem was described
by Osher \cite{osher84}.
Using this, a numerical
scheme can be defined via the $d$ numerical flux functions
\begin{displaymath}
g^G_l(v,w) \; = \;
\left\lbrace
\begin{array}{ccc}
\min_{v \leq u \leq w}f_l(u) & : & v \leq w \,,\\ 
\max_{w \leq u \leq v}f_l(u) & : & v > w \,.
\end{array}
\right.
\end{displaymath}
Since the relative values of the test variables have to be 
compared within the scheme, diversions by cases have to be
employed.\\
{\it To condition (\ref{breuss-18}):}\\
Generally, for $l=1, \, \ldots, \, d$,
\begin{displaymath}
\tilde H_l (a+\Delta a,b,c)
-
\tilde H_l (a,b,c) = 
\frac{\Delta t}{\Delta x_l}
\bigl[
g_l^G(a+\Delta a,b)
-
g_l^G(a,b)
\bigr]
\end{displaymath}
holds. Since only the values $b,\, a$ and $a+\Delta a$
are of importance, it is necessary
to investigate three cases for each $l \in \underline{d}$.
\begin{enumerate}
\item
Case: $b \leq a \leq a+\Delta a$
\begin{displaymath}
\frac{\Delta t}{\Delta x_l} 
\bigl[
g_l^G(a+\Delta a,b)-g_l^G(a,b) \bigr]
=
\frac{\Delta t}{\Delta x_l}
\left[ \,
\max_{b \leq u \leq a+\Delta a} f_l(u)
-
\max_{b \leq u \leq a} f_l(u) \,
\right] \; \geq 0
\,.
\end{displaymath}
\item
Case: $a \leq b \leq a+ \Delta a$
\begin{displaymath}
\frac{\Delta t}{\Delta x_l}
\bigl[
g_l^G(a+\Delta a,b)-g_l^G(a,b) \bigr]
=
\frac{\Delta t}{\Delta x_l} 
\left[ \,
\max_{b \leq u \leq a+\Delta a} f_l(u)
-
\min_{a \leq u \leq b} f_l(u) \,
\right] \; \geq 0
\,.
\end{displaymath}
\item
Case: $a \leq a+\Delta a \leq b$
\begin{displaymath}
\frac{\Delta t}{\Delta x_l}
\bigl[
g_l^G(a+\Delta a,b)-g_l^G(a,b) \bigr]
=
\frac{\Delta t}{\Delta x_l} 
\left[ \,
\min_{a+\Delta a \leq u \leq b} f_l(u)
-
\min_{a \leq u \leq b} f_l(u) \,
\right] \; \geq 0
\,.
\end{displaymath}
\end{enumerate}
Thus, the validity of the condition
(\ref{breuss-18}) is guaranteed without any additional condition
on the flux function. This can be verified analogously
for condition (\ref{breuss-19}), so that the investigated
Godunov-type scheme is monotone for general continuous flux
functions.

\subsection{Convergence}
\label{breuss-subsection-4-4}

Within this section, we prove convergence of the mentioned
schemes under the assumption that the conditions for
monotonicity are fulfilled.
We do this in some detail for the implicit upwind method,
since this is demonstrated in the easiest fashion, 
and we refer to the differences concerning the proofs 
of convergence with respect to the other methods afterwards.
The same holds true with respect to the type of sources
employed in {\it Scenario 2}. Since part of the convergence
proof is technically identical to the proofs in the one-dimensional
case without sources described in
\cite{breuss02}, we refer to that work for more details.

The basic idea of the convergence proofs is
the following. Corresponding to
sequences $\Delta x_l^k \downarrow 0$ for $k \to \infty$,
$l \in \underline{d}$, we construct
a monotonically growing
sequence of discrete initial data. Then by the monotonicity of
the method we get a monotonically growing sequence of
numerical solutions. Since we multiply
the initial function $u_0$ with an arbitrarily chosen but fixed
test function with compact support,
we only have to consider $u_0$ over a finite domain.
Because of the assumption $u_0 \in L_\infty$
and since we have
$L_\infty$-Stability, the corresponding function
sequence is integrable and
bounded from above. Then we can use the well-known
theorem of monotone convergence of Beppo Levi to show
convergence (almost everywhere) to a limit function.
More formally, we state the following
\vspace{0.5ex}

\begin{Theorem}\label{theorem-breuss-4}
Let $u_0(\mathbf{x})$ be in 
$L^{loc}_\infty \left( \mathbb{R}^d ; \, \mathbb{R} \right)$.
Consider a sequence of nested grids indexed by
$k=1,2, \ldots$, with mesh parameters
$\Delta t_k \downarrow 0$ and $\Delta x_l^k \downarrow 0$,
$l=1, \, \ldots, \, d$, as $k \to \infty$,
and let $u_k(\mathbf{x},t)$ denote the step function obtained via
the numerical approximation by a consistent, conservative
and monotone scheme in the form of the discussed methods.
Then $u_k(\mathbf{x},t)$ converges to the unique entropy solution of
the given conservation law as $k \to \infty$.\vspace{1.0ex}
\end{Theorem}

\begin{proof}
At first, the convergence
to a weak solution of the conservation law
is established, followed by
the verification that this weak solution
is the entropy solution.
For brevity of the notation, we omit the
arguments $( \mathbf{x},t)$ when appropriate.

We employ sequences
$\Delta t_k \downarrow 0$ and $\Delta x_l^k \downarrow 0$,
assuming that the resulting grids are nested in order to compare
data sets of values, i.e. refined grids always
inherit cell borders.

The most important technical detail is
the special discretization of the initial condition 
$u_0 \in L^{loc}_\infty \left( \mathbb{R}^d ; \, \mathbb{R} \right)$.
After a suitable modification on a set of
Lebesgue measure zero,
the initial condition is discretized on cell $j \in J$, i.e.
for
\begin{displaymath}
\mathbf{x} \in 
\left(
(j_1-1) \Delta x_1^0, j_1 \Delta x_1^0
\right]
\times
\ldots
\times
\left(
(j_d-1)\Delta x_d^0, j_d \Delta x_d^0 
\right]
\,,
\end{displaymath}
by
\begin{equation}
u_j^0 \; := 
\;
\inf_{
\mathbf{x}  
\textrm{ in cell }j
}
u_0(\mathbf{x})
\,.
\label{breuss-30}
\end{equation}
Corresponding to the initial data we also define
a piecewise continuous function
\begin{equation} 
u_k (\mathbf{x},0) := u_j^0 \textrm{, }
\mathbf{x} \textrm{ in cell }j
\,.
\label{breuss-31}
\end{equation}
It is a simple matter of classical analysis
to verify that the discretization (\ref{breuss-30}) together with
(\ref{breuss-31}) gives on any compact spatial domain
a monotonically growing function sequence with
\begin{equation}
\lim_{k \to \infty} u_k(\mathbf{x},0) \; = \; u_0(\mathbf{x})
\textrm{ almost everywhere}
\label{breuss-32}
\end{equation}
by application of the theorem of monotone convergence.
In the classical fashion using point values,
we extract discrete test elements $\phi_j^0$
out of a given test function
$\phi \in C_0^\infty \left( \mathbb{R}^{d+1} ; \,\mathbb{R} \right)$.
Additionally, we define for $n \geq 1$ the step function
\begin{displaymath}
u_k(x,t)=u_j^n
\textrm{, }
\mathbf{x} \textrm{ in cell }j
\textrm{, }
t^{n-1}<t \leq t^n
\,.
\end{displaymath}
In the following,
let the test function $\phi$ be chosen arbitrarily but fixed.

Multiplication of the implicit upwind scheme (\ref{breuss-26})
with $\Delta t^k \prod_{l=1}^d \Delta x_l^k$ as well as
with the discrete test element
$\phi_j^{n+1}$, summation over the spatial indices
$j \in J$ and the temporal indices 
$n \geq 0$, and finally summation by parts yields
\begin{eqnarray}
&&
\Delta t^k
\prod_{l=1}^d
\Delta x_l^k 
\left\{
\sum_{j \in J} \sum_{n \geq 0}
\left[
u_j^n \frac{\phi_j^{n+1}-\phi_j^n}{\Delta t^k}
+
\sum_{l=1}^d
f_l(u_j^{n+1})
\frac{ \phi_{j+\delta l}^{n+1}- \phi_j^{n+1}}{\Delta x_l^k}
\right]
\right\} \nonumber\\
& = & 
-
\prod_{l=1}^d
\Delta x_l^k 
\sum_{j \in J} u_j^0 \phi_j^0
+
\Delta t^k
\prod_{l=1}^d
\Delta x_l^k 
\sum_{j \in J} q_j^{n+1} \phi_j^{n+1}
\,.
\label{breuss-33}
\end{eqnarray}  
By the definition of the introduced step functions,
(\ref{breuss-33}) is equivalent to
\begin{eqnarray}
&&
\int_{\mathbb{R}_+} \int_{\mathbb{R}^d}
\left[
u_k( \mathbf{x},t)
\frac{\phi_k (\mathbf{x}, t + \Delta t^k )
-
\phi_k( \mathbf{x},t)}{\Delta t^k}
\right.
\nonumber\\
&&
+
\left.
\sum_{l=1}^d
f_l(u_k( \mathbf{x},t+\Delta t^k))
\frac{\phi_k(\mathbf{x}+\Delta x_l^k, t+\Delta t^k)
- \phi_k(\mathbf{x},t+\Delta t^k)}{\Delta x_l^k}
\right]
\, \mathrm{d}\mathbf{x}\,\mathrm{d}t 
\nonumber\\
& = &
-
\int_{\mathbb{R}^d} u_k(\mathbf{x},0)
\phi_k(\mathbf{x},0) \, \mathrm{d}\mathbf{x}
+
\int_{\mathbb{R}_+} \int_{\mathbb{R}^d}
q_k(\mathbf{x},t+\Delta t^k)
\phi_k(\mathbf{x},t+\Delta t^k) \, \mathrm{d}\mathbf{x}\,\mathrm{d}t
\,.
\label{breuss-34}
\end{eqnarray}
We now prove convergence of (\ref{breuss-34})
to the form which implies that
$u$ is a weak solution of the original problem, 
see (\ref{breuss-7}).

We first investigate the right hand side of (\ref{breuss-34}).
Set $\tilde \Delta:=\max_{l \in \underline{d}}\Delta x_l^0$
and let
\begin{displaymath}
K:= 
\left\{ 
\,
(\mathbf{x},t) 
\, | \, 
\exists (\mathbf{y},t) \in \mathrm{supp}(\phi):
\,
t = 0
\textrm{ and }
y_l 
-
\tilde \Delta
\leq 
x_l
\leq 
y_l + \tilde \Delta 
\textrm{ for all }
l \in \underline{d}
\,
\right\}
\,. 
\end{displaymath}
By construction, $K$ is compact and gives the largest possible
spatial domain where non-zero discrete initial data
may occur. Adding zeroes, we now cast the problem
into a more suitable form, namely
\begin{eqnarray}
\lefteqn{
\int_{\mathbb{R}^d} 
u_k(\mathbf{x},0) \phi_k(\mathbf{x},0) \, \mathrm{d}\mathbf{x}
= 
\int_K
u_0(\mathbf{x})
\phi(\mathbf{x},0) \, \mathrm{d}\mathbf{x}
} \nonumber\\
& &
\quad
+
\int_K u_k(\mathbf{x},0) 
\left[
\phi_k(\mathbf{x},0) - \phi(\mathbf{x},0)
\right]
\, \mathrm{d}\mathbf{x} 
+
\int_K
\left[
u_k(\mathbf{x},0)
- u_0(\mathbf{x})
\right]
\phi(\mathbf{x},0) \, \mathrm{d}\mathbf{x} 
\,.
\label{breuss-35}
\end{eqnarray}
Because of $u_0 \in L^\infty(\mathbb{R}^d; \, \mathbb{R})$
and by our construction,
we can estimate the absolute value of the second right hand side term in
(\ref{breuss-35}) by the help of a constant
$M_u<\infty$:
\begin{equation}
\left|
\int_K u_k(\mathbf{x},0) 
\left[
\phi_k(\mathbf{x},0) - \phi(\mathbf{x},0)
\right]
\, \mathrm{d}\mathbf{x} \right|
\leq
M_u
\mathopen{|} K \mathclose{|}
\sup_{x \in K}
\left|
\phi_k(\mathbf{x},0) - \phi(\mathbf{x},0)
\right|
\,.
\label{breuss-36}
\end{equation}
Since $\phi$ is a smooth test function, it is a simple but
technical exercise to show
\begin{equation}
\| 
\phi_k(\mathbf{x},0)
-
\phi(\mathbf{x},0)
\|_\infty 
\to 0
\quad
\textrm{for}
\quad
k \to \infty
\,.
\label{breuss-37}
\end{equation}
By (\ref{breuss-36}) and (\ref{breuss-37}),
the investigated term tends to zero with $k \to \infty$.
Since $\phi$ is continuous and since
$u_k(\mathbf{x},0)$ approaches $u_0(\mathbf{x})$ from below
by construction, we can estimate
the absolute of the third right hand side
term in (\ref{breuss-35})
with the help of a constant $M_\phi<\infty$ by
\begin{displaymath}
\left|
\int_K
\left[
u_k(\mathbf{x},0)
- u_0(\mathbf{x})
\right]
\phi(\mathbf{x},0) \, \mathrm{d}\mathbf{x}
\right|
\leq
M_\phi
\int_K
u_0(\mathbf{x})
-
u_k(\mathbf{x},0)
\, \mathrm{d}\mathbf{x}
\,.
\end{displaymath}
The theorem of monotone convergence implies that
\begin{displaymath}
\int_K
u_0(\mathbf{x})
- u_k(\mathbf{x},0)
\, \mathrm{d}\mathbf{x}
\end{displaymath}
vanishes in the limit for $k\to \infty$, i.e.
the corresponding term in (\ref{breuss-35})
goes to zero for $k \to \infty$.
To condense these results, we obtain
\begin{displaymath}
\lim_{k \to \infty}
\int_{\mathbb{R}^d} 
u_k(\mathbf{x},0)
\phi_k(\mathbf{x},0) \, \mathrm{d}\mathbf{x}
=
\int_{\mathbb{R}^d} 
u_0(\mathbf{x})
\phi(\mathbf{x},0) \, \mathrm{d}\mathbf{x}
\,.
\end{displaymath}
It remains to show
\begin{displaymath}
\int_{\mathbb{R}_+} \int_{\mathbb{R}^d}
q_k(\mathbf{x},t+\Delta t^k)
\phi_k(\mathbf{x},t+\Delta t^k) \, \mathrm{d}\mathbf{x}\,\mathrm{d}t
\stackrel{k \to \infty}{\longrightarrow}
\int_{\mathbb{R}_+} \int_{\mathbb{R}^d}
q(\mathbf{x},t)
\phi(\mathbf{x},t) \, \mathrm{d}\mathbf{x}\,\mathrm{d}t
\,.
\end{displaymath}
This result can easily be achieved 
by analogously introducing a compact domain 
$S \subset \mathbb{R}^d$
including the support of $\phi$ in space and time,
setting for $n \geq 1$
($n=0$ is not relevant since $q(\cdot, 0)\equiv 0$)
\begin{displaymath}
q_j^n \; := 
\;
\inf_{
(\mathbf{x},t) 
\textrm{ with }
\mathbf{x}
\textrm{ in cell }j
\textrm{ and $t$ in }
\left( t^n-\Delta t^0, t^n \right]
}
q(\mathbf{x},t)
\end{displaymath}
and using a similar manipulation as for the terms
involving $u_0$.

Concerning the left hand side of (\ref{breuss-35}),
adding zeroes and using the attributes of test functions
together with the $L_\infty$-stability of $u_k$
yields that we finally have to show
\begin{equation}
\lim_{k \to \infty}
\int_S
\bigl|
u(\mathbf{x},t)
-
u_k(\mathbf{x},t)
\bigr|
\left|
\phi_t(\mathbf{x},t)
\right|
\, \mathrm{d}\mathbf{x}\,\mathrm{d}t
\stackrel{k \to \infty}{\longrightarrow}
0
\label{breuss-38}
\end{equation}
and also for all $l\in \underline{d}$
\begin{equation}
\lim_{k \to \infty}
\int_S
\bigl|
f_l(u(\mathbf{x},t))
-
f_l(u_k(\mathbf{x},t+\Delta t^k))
\bigr|
\left|
\frac{\partial}{\partial x_l}
\phi(\mathbf{x},t)
\right|
\, \mathrm{d}\mathbf{x} \mathrm{d}t
\stackrel{k \to \infty}{\longrightarrow}
0
\label{breuss-39}
\end{equation}
in order to prove convergence to a weak solution.
Since $\phi_t$ is continuous on $S$, we can estimate
$\left| \phi_t \right|$ in (\ref{breuss-38}) by a
constant $M_t<\infty$.
Since $u_k(\mathbf{x},t)$ grows monotonically with $k \to \infty$
in the sense of pointwise comparison, and since it
is positive and bounded from above because of $u_0 \in L_\infty(S)$
and the monotonicity of the method,
the function sequence 
$\left( u_k( \mathbf{x},t) \right)_{k \in \mathbf{N}}$
converges almost everywhere to an integrable
limit function on $S$
by the theorem of monotone convergence due to Levi.
We set
\begin{displaymath}
u(\mathbf{x},t) \, := \, \lim_{k \to \infty} u_k(\mathbf{x},t)
\,.
\end{displaymath}
Introducing exactly this limit function as the function 
$u(\mathbf{x},t)$
used up to now, the corresponding term in (\ref{breuss-38})
becomes zero in the limit:
\begin{displaymath}
\lim_{k \to \infty}
\int_S
\bigl|
u(\mathbf{x},t)
-
u_k(\mathbf{x},t)
\bigr|
\left|
\phi_t(\mathbf{x},t)
\right|
\, \mathrm{d}\mathbf{x}\,\mathrm{d}t
\leq
M_t
\int_S
u(\mathbf{x},t)
-
\lim_{k \to \infty}
u_k(\mathbf{x},t)
\, \mathrm{d}\mathbf{x}\,\mathrm{d}t
=0
\,.
\end{displaymath}
Note that the pointwise convergence $u_k \to u$ almost everywhere
is now established and can be used in the following proofs.
For proving (\ref{breuss-39}), we need some further
simple manipulations.
We use again the continuity of the derivatives of $\phi$
to introduce constants $M^l_x<\infty$ to obtain
\begin{eqnarray} 
\lefteqn{
\int_S
\bigl|
f_l(u(\mathbf{x},t))
-
f_l(u_k(\mathbf{x},t+\Delta t^k))
\bigr|
\left|
\frac{\partial}{\partial x_l}
\phi(\mathbf{x},t)
\right|
\, \mathrm{d}\mathbf{x}\,\mathrm{d}t \leq }\nonumber \\
&&
M^l_x
\int_S
\left|
f_l(u_k(\mathbf{x},t))
-
f_l(u(\mathbf{x},t))
\right|
\, \mathrm{d}\mathbf{x}\,\mathrm{d}t
\nonumber\\
&&
\;
+
M^l_x
\int_S
\left|
f_l(u_k(\mathbf{x},t+\Delta t^k))
-
f_l(u_k(\mathbf{x},t))
\right|
\, \mathrm{d}\mathbf{x}\,\mathrm{d}t
\label{breuss-40}
\end{eqnarray}
for all $l\in \underline{d}$.
We now discuss the first right hand side 
term in (\ref{breuss-40}).
Since by construction $u_k$ and $u$ are
in $L_\infty(S)$, we can estimate every
$\left| f_l(u_k(\mathbf{x},t+\Delta t^k)) 
-
f_l(u_k(\mathbf{x},t)) \right|$
over $S$ from
above by a constant $M^l_f<\infty$ because of the continuity of 
the $f_l$ on the compact set of possible values.
Then the functions
\begin{displaymath}
M^l_f(\mathbf{x},t):=
\left\lbrace
\begin{array}{ccc}
M^l_f &, & (\mathbf{x},t) \in S \\
0 &, & \textrm{otherwise}
\end{array}
\right.
\,,
\end{displaymath}
are in $L_1(\mathbb{R}^d \times \mathbf{R_+}; \, \mathbb{R})$
and dominate
$\left| f_l(u_k(\mathbf{x},t)) - f_l(u(\mathbf{x},t)) \right|$
for all $l\in \underline{d}$ and all $k$.
Because of the established pointwise convergence $u_k \to u$
a.e., we can apply the theorem of dominated convergence by 
Lebesgue to obtain for all $l$
\begin{equation}
\lim_{k \to \infty}
M_x
\int_S
\left|
f_l(u_k(\mathbf{x},t))
-
f_l(u(\mathbf{x},t))
\right|
\, \mathrm{d}\mathbf{x}\,\mathrm{d}t 
= 
0
\,.
\label{breuss-41}
\end{equation}
Now we discuss the second right hand side
term in (\ref{breuss-40}).
Since by construction $u_k$ is a step function with finite
values on the compact domain $S$, $u_k$ is in $L_1(S)$.
Since the $f_l$ are continuous, also $f_l \circ u_k$ are 
in $L_1(S)$. By the continuity in the mean of $L_1$-functions,
there exist $\delta_l(\epsilon)$ for all $\epsilon>0$ with
\begin{displaymath}
\int_S
\left|
f_l(u_k(\mathbf{x},t+\Delta t^k))
-
f_l(u_k(\mathbf{x},t))
\right|
\, \mathrm{d}\mathbf{x}\,\mathrm{d}t \, < \, \epsilon\,,
\end{displaymath}
if $\Delta t^k < \delta_l(\epsilon)$. Since $\Delta t^k \downarrow 0$
for $k \to \infty$, $\epsilon$ can be chosen arbitrarily small,
i.e. 
\begin{equation}
M_x
\int_S
\left|
f_l(u_k(\mathbf{x},t+\Delta t^k))
-
f_l(u_k(\mathbf{x},t))
\right|
\, \mathrm{d}\mathbf{x}\,\mathrm{d}t
\, \to \,
0 \;
\textrm{ for } \; k \to \infty
\label{breuss-42}
\end{equation}
holds for all $l \in \underline{d}$.
By (\ref{breuss-41}) and (\ref{breuss-42}) the assertion
in (\ref{breuss-39}) is proven. Since the test element
$\phi$ was chosen arbitrarily, convergence to a weak
solution is established.

We have now to show that exactly this weak solution
is the unique entropy solution in the sense of Kru\v{z}kov.
Therefore, we derive in a similar fashion as in the derivation of
(\ref{breuss-33}) the weak form of the discrete entropy
condition (\ref{breuss-21}) connected with the implicit upwind
scheme using Lemma \ref{lemma-breuss-1} and
Theorem \ref{theorem-breuss-2}. It reads
\begin{eqnarray}
&&
-
\Delta t^k
\prod_{l=1}^d
\Delta x_l^k
\sum_{j \in J}
\mathopen{|}
u_j^0 - k
\mathclose{|}
\phi_j^0 
-
\Delta t^k
\prod_{l=1}^d
\Delta x_l^k
\sum_{j \in J}
\sum_{n \geq 0}
\mathrm{sgn}
\left(
u_j^{n+1} - k
\right)
q_j^{n+1}
\phi_j^{n+1}
\nonumber\\
& \leq &
\Delta t^k
\prod_{l=1}^d
\Delta x_l^k
\sum_{j \in J}
\sum_{n \geq 0}
\left[
\mathopen{|}
u_j^n - k
\mathclose{|}
\frac{\phi_j^{n+1}-\phi_j^n}{\Delta t^k}
\right.
\nonumber\\
& &
\qquad
+
\mathrm{sgn}
\left( 
u_j^{n+1} - k 
\right)
\sum_{l=1}^d
\left\{
\left[
f_l(u_j^{n+1}) - f_l(k)
\right]
\frac{ \phi_{j+1}^{n+1}- \phi_j^{n+1}}{\Delta x_l^k}
\right\}
\Biggr]
\,.
\label{breuss-43}
\end{eqnarray}
Using the established convergence $u_k \to u$ a.e.
of the function sequence
ge\-ne\-ra\-ted by the numerical
method for $\Delta t^k \downarrow 0$
and $\Delta x_l^k \downarrow 0$ for all $l \in \underline{d}$,
we now prove convergence of (\ref{breuss-43})
towards the form of the
entropy condition due to Kru\v{z}kov (\ref{breuss-8}).
Therefore, we have to consider
arbitrarily chosen but fixed test elements
composed of a test function $\phi$ with
$\phi \geq 0$, 
$\phi \in C_0^\infty(\mathbb{R}^{d+1}; \, \mathbb{R})$,
and a test number $k \in \mathbb{R}$.

Using the same notation and applying a similar
procedure as in the case of the
convergence proof to a weak solution, we first
want to prove
\begin{equation}
\lim_{k \to \infty}
M_\phi
\int_K
\bigl|
\left| u_k(\mathbf{x},0)-k \right|
-
\left| u_0(\mathbf{x}) - k \right|
\bigr|
\, \mathrm{d}\mathbf{x}
= 0
\,.
\label{breuss-44}
\end{equation}
Since $k$ is fixed and $u_k( \mathbf{x},0)$ and $u_0$ bounded,
one can find a constant function over the compact interval
$K$ which dominates the integrand
for all $k$.
Then (\ref{breuss-44}) follows
by the use of the already
established convergence $u_k(x,0) \to u_0(x)$
a.e. and the theorem of dominated convergence by Lebesgue.
We also have to treat
\begin{displaymath}
\lim_{k \to \infty}
\int_{\mathbb{R}_+}
\int_{\mathbb{R}^d}
\mathrm{sgn}
\left(
u_k(\mathbf{x},t+\Delta t^k) - k
\right)
q_k(\mathbf{x},t+\Delta t^k)
\phi(\mathbf{x},t+\Delta t^k)
\,
\mathrm{d}\mathbf{x}\,\mathrm{d}t
\,.
\end{displaymath}
Therefore, we expand the factor $\phi(\mathbf{x},t+\Delta t^k)$
by adding zeroes in the form
\begin{displaymath}
\phi(\mathbf{x},t+\Delta t^k) =
\phi(\mathbf{x},t+\Delta t^k)
-
\phi(\mathbf{x},t)
+
\phi(\mathbf{x},t)
\,.
\end{displaymath}
Convergence of the integrals involving the factor
$\phi(\mathbf{x},t+\Delta t^k) - \phi(\mathbf{x},t)$ tend to zero. This
follows by estimating $\mathrm{sgn}$, $u_k$ and $q_k$ from above and using
the usual properties of test functions.
In a similar fashion, we expand the factor
$\mathrm{sgn} \left( u_k(\mathbf{x},t+\Delta t^k) - k \right)$, adding
zero in the form 
$-\mathrm{sgn} \left( u_k(\mathbf{x},t) - k \right)
+\mathrm{sgn} \left( u_k(\mathbf{x},t) - k \right)$.
The proof that the integrals involving 
$\mathrm{sgn} \left( u_k(\mathbf{x},t+\Delta t^k) - k \right)
- \mathrm{sgn} \left( u_k(\mathbf{x},t) - k \right)$
vanish follows from the continuity in the mean of $L_1$-functions. 
Again similarly, we expand in the form
$q_k(\mathbf{x},t+\Delta t^k)=
q_k(\mathbf{x},t+\Delta t^k)
-
q(\mathbf{x},t+\Delta t^k)
+
q(\mathbf{x},t+\Delta t^k)$
and use the concept of monotone convergence due to
Beppo Levi to obtain convergence to zero of the integrals involving 
$q_k(\mathbf{x},t+\Delta t^k) - q(\mathbf{x},t+\Delta t^k)$.
Lastly, the proof of convergence of 
$q(\mathbf{x},t+\Delta t^k)$ to $q(\mathbf{x},t)$
under the integral follows from the continuity in the mean
of $L_1$-functions. The technical details only require
to take all expansions obtained by taking suitable zeroes
into account and eliminating all integrals
which involve discrete notions.
The other terms left to investigate are
\begin{eqnarray}
&&
\int_S
\left|
u_k(\mathbf{x},t)-k
\right|
\phi_t( \mathbf{x},t)
\, \mathrm{d}\mathbf{x}\,\mathrm{d}t \quad \textrm{and} \nonumber\\
&&
\int_S
\sum_{l=1}^d
\mathrm{sgn} \left[
u_k(\mathbf{x},t+\Delta t^k)-k
\right]
\bigl[
f_l(u_k(\mathbf{x},t+\Delta t^k))
-
f_l(k)
\bigr]
\frac{\partial}{\partial x_l}
\phi(\mathbf{x},t)
\, \mathrm{d}\mathbf{x}\,\mathrm{d}t
\,. \nonumber
\end{eqnarray}
The procedure is the same in both cases. Since the occurring
derivatives of $\phi$ are continuous, we can estimate these
over the compact domain $S$ by finite constants.
Since $k$ is a fixed value 
(and so is $f_l(k)$ for all $l\in \underline{d}$),
since $u_k(\mathbf{x},t)$ is bounded and
because the $f_l$ are
continuous over the bounded interval of possible
values of $u_k$ (due to the established $L_\infty$-stability),
we can also give constants which
estimate all the expressions involving $u_k$ from above.
Using the product of these finite constants as dominating
function over $S$ as well as $u_k \to u$ a.e., we employ 
the theorem of dominated convergence to receive the desired result
for the implicit upwind scheme.

In the case of the implicit Lax-Friedrichs method,
we have to assume Lipschitz continuity with a Lipschitz 
constant $L \leq 1/\lambda$ of the flux functions
so that the method is monotone.
In comparison to the implicit upwind method,
the difference in the
corresponding weak forms are made up from 
\begin{displaymath}
-\frac{\Delta t_k}{2}
\int_S 
u_k(\mathbf{x}, t+\Delta t^k) 
\hat \phi \, \mathrm{d}\mathbf{x}\,\mathrm{d}t
\textrm{ and }
-\frac{\Delta t_k}{2} \int_S 
\left| u_k(\mathbf{x}, t+\Delta t^k)
 - k 
\right| \hat \phi 
\, \mathrm{d}\mathbf{x}\,\mathrm{d}t
\,,
\end{displaymath}
respectively. Thereby, $\hat \phi$ converges in the $L_\infty$-Norm
to $\partial^2_{x_l^2}\phi(\mathbf{x},t)$ which is continuous since
$\phi \in C_0^\infty(\mathbb{R}^{d+1}; \, \mathbb{R})$.
Thus, the corresponding term can be estimated from above
by a constant over the compact domain $S$.
Since $k$ is fixed and $u_k(\mathbf{x},t)$ is bounded as usual,
both expressions vanish with $\Delta t^k \downarrow 0$.

With respect to the described
implicit Godunov-type method, we use a similar procedure
as in the case of the implicit Lax-Friedrichs methods,
namely to write down the differences in the weak forms to 
the case of the implicit upwind method.
These are made up from
\begin{eqnarray}
& &
\Bigl\{
\bigl[
g_l^G(u_j^{n+1} \vee k, u_{j+\delta l}^{n+1} \vee k)-
g_l^G(u_j^{n+1} \wedge k, u_{j+\delta l}^{n+1} \wedge k)
\bigr]
\Bigr. \nonumber\\
&& 
\quad
\Bigl.
- \mathrm{sgn}
\left(
u_j^{n+1}
-k
\right)
\bigl[
f_l(u_j^{n+1})- f_l(k)
\bigr]
\Bigr\}
\frac{\phi_{j+\delta l}^{n+1}-\phi_j^{n+1}}{\Delta x_l^k}
\label{breuss-45}\\
&&
\textrm{and}
\quad
\bigl[
g_l^G(u_j^{n+1}, u_{j+1}^{n+1})-f_l(u_j^{n+1})
\bigr]
\frac{\phi_{j+\delta l}^{n+1}-\phi_j^{n+1}}{\Delta x_l^k}
\,.
\label{breuss-46}
\end{eqnarray}
Since $g_G$ is continuous in the components and
$u_k(\mathbf{x},t) \in L_1(S)$, the $g_l^G \circ u_k$ are also in
$L_1(S)$. After introducing step functions as usual,
the expressions incorporating $g_l^G$ from (\ref{breuss-45})
give values $f_l(\xi_l)$ with
\begin{displaymath}
\begin{array}{cc}
 & \xi_l \in \left[ u_k(\mathbf{x}, t+\Delta t^k),
u_k(\mathbf{x}+\Delta x_l^k,t+\Delta t^k) \right] \\
\textrm{or} &
\; 
\xi_l \in \left[ u_k(\mathbf{x} + \Delta x_l^k, t+\Delta t^k),
u_k(\mathbf{x},t+\Delta t^k) \right] \,,
\end{array}
\end{displaymath}
respectively. The integrals over the terms 
corresponding to (\ref{breuss-45})
then go to zero with $k \to \infty$ because of the
continuity in the mean of $g_l^G \circ u_k$. The idea for proving
convergence to zero concerning the integral of
the expressions corresponding to (\ref{breuss-46}) is the same.

Concerning {\it Scenario 2}, the described strategy is fully
transferable by employing accordingly the notions developed
in section \ref{breuss-section-3}. 
\end{proof}

\section{Numerical tests}
\label{breuss-section-5}

In order to show the applicability of the developed notions,
we investigate numerically a number of test cases 
which were employed within the literature in 
various contexts. In contrast to the cited examples from
\cite{santosoliveira99,grlebano97} we do not employ any further
manipulations of the problem or on the numerical side, we simply
rely on the straightforward application of the implicit
Godunov-type method in all cases.

We remark that the notions we developed reduce in the
one-dimensional case without sources to the notions described
within \cite{breuss01,breuss02}. In that works, especially the 
applicability of implicit schemes with respect to
a conservation law given in \cite{krupan91}
was shown where the solution features a rarefaction
wave extending in an arbitrarily small time step
to infinity. Thus, the applicability of the described concept
is already established in the case without sources,
where the flux is merely continuous and where
a meaningful CFL-condition does not exist. With the
numerical tests documented in this section, we focus on the theoretical
extensions developed in this paper.

First, we employ a
one-dimensional conservation law featuring as a particular problem
a point source depending on space and time. This test case 
was used in \cite{santosoliveira99}
to show experimentally convergence to the entropy solution.
In contrast to the scheme used in their work, for our
scheme convergence to the entropy solution is guaranteed. 

At second, we consider a couple of model problems
featuring spatial dependent sources having the form of the
derivative of certain functions. These model problems were used in 
\cite{grlebano97} in order to show numerically
convergence to steady state solutions featuring various
difficulties which is in contrast
to the first unsteady example. Moreover, since we use
the described Godunov-type scheme, we do not rely on a
CFL-like condition as in \cite{grlebano97}
which greatly restricts the time step size,
an annoying aspect in steady state calculations.
Note also that we simply employ the described implicit method
without any further improvements as done with the 
explicit method employed
in \cite{grlebano97} for which convergence was not guaranteed.

While the first two examples could be identified as belonging
to {\it Scenario 1}, the third test case refers to 
{\it Scenario 2}. It consists of a one-dimensional model problem
featuring a parameter dependent source term depending also 
in a nonlinear way on the solution. This model problem
was used by LeVeque and Yee \cite{levequeyee90}
to illuminate numerical difficulties in the case of stiffness.

As fourth and last example, we show numerical results
of a two-dimensional problem used in \cite{grsome02} which exhibits
all principal difficulties encountered when dealing
with hyperbolic equations. As in the second example,
the implicitness of our methods is advantageous 
in order to calculate the steady state solution.

In all examples, nonlinear systems of equations arise
which have been solved numerically with an iterative solver. Precisely, we used the MINPACK subroutine \textit{hybrd}.

\subsection*{Example 1}

The one-dimensional scalar conservation law under consideration is 
\begin{displaymath}
\frac{\partial}{\partial t} u(x,t) 
+
\frac{\partial}{\partial x}
u (x,t) = \sin(\pi t)\delta(x-0.1)
\,, 
\quad
x \in (0,1)
\,, \;
t>0
\,, 
\end{displaymath}
\begin{displaymath}
\textrm{with} \quad
u_0(x)=u(x,0)=0 \; \forall x \in (0,1)
\quad
\textrm{and}
\quad
u(0,t)=0 \; \forall t \geq 0
\,.
\end{displaymath}
\begin{figure}[!ht]
\centering
\includegraphics[height=4.5cm]{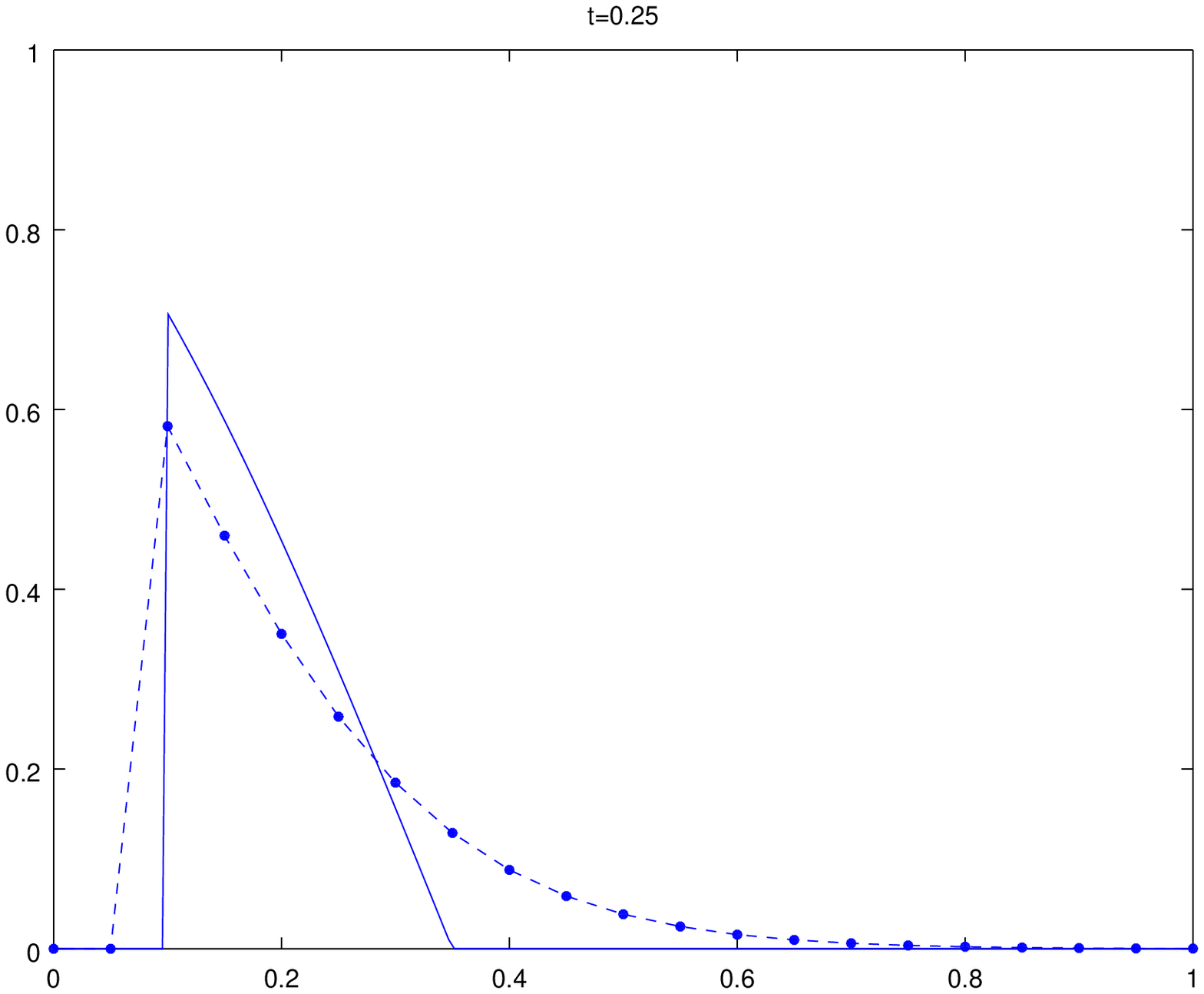}
\includegraphics[height=4.5cm]{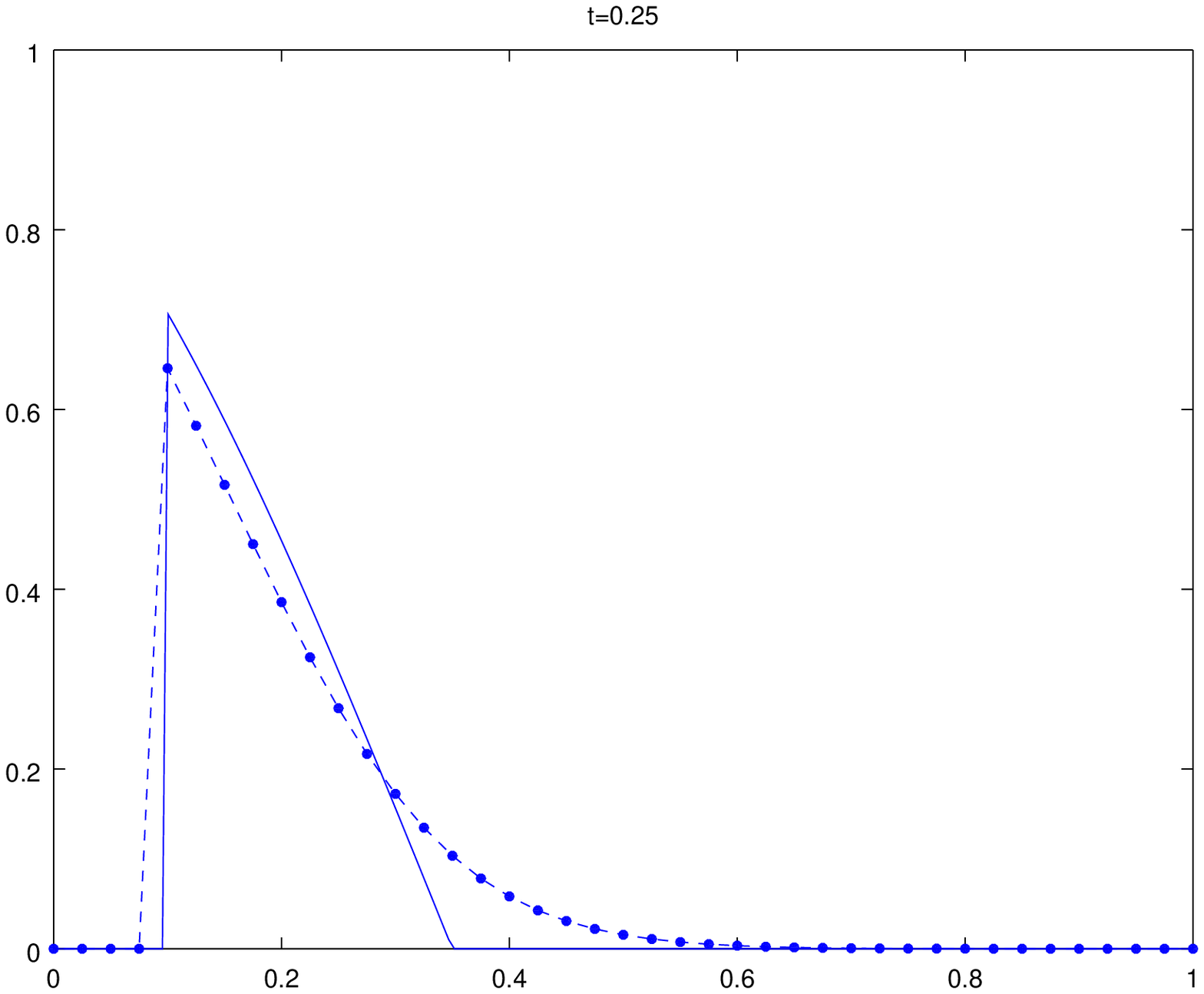}
\includegraphics[height=4.5cm]{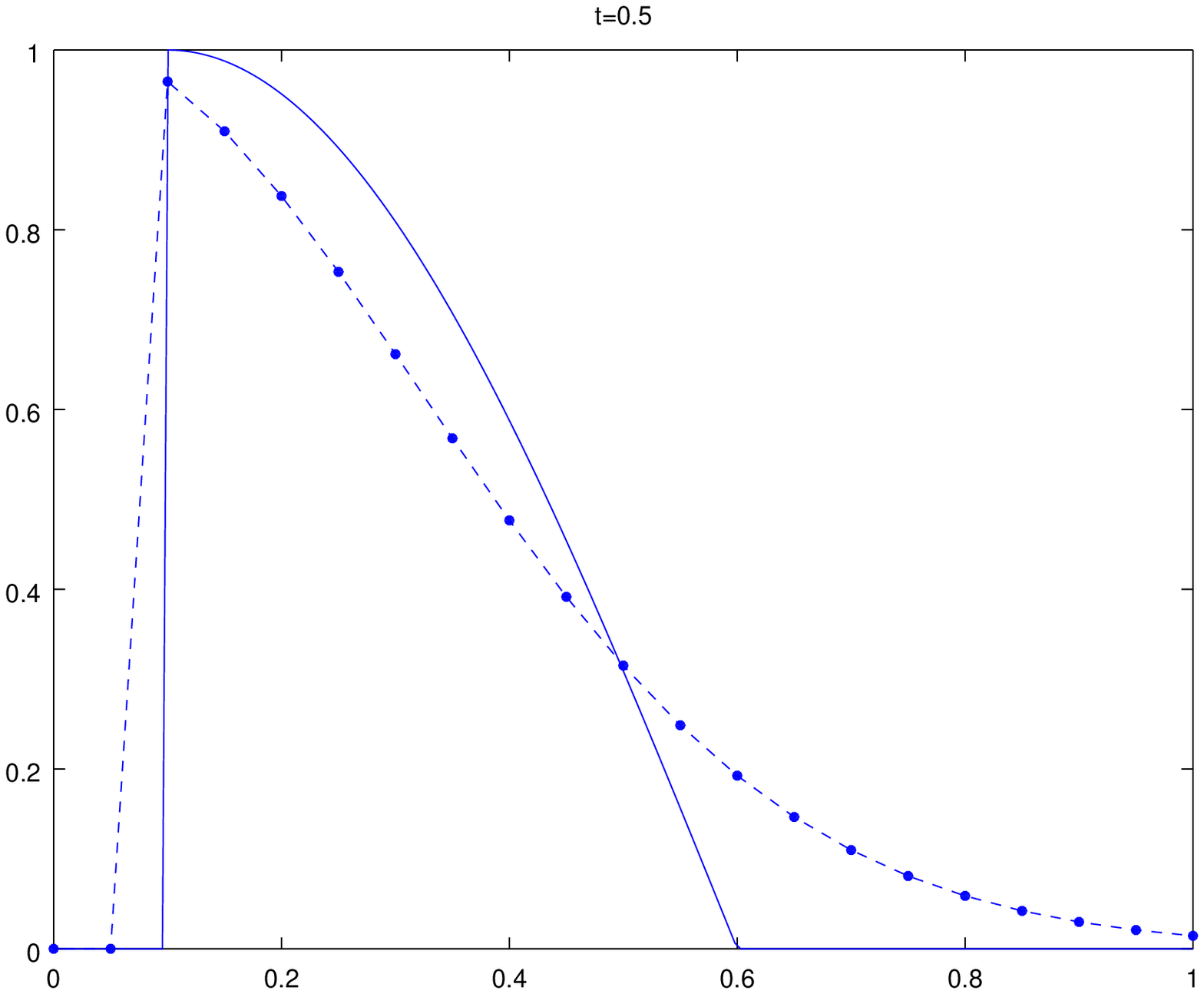}
\includegraphics[height=4.5cm]{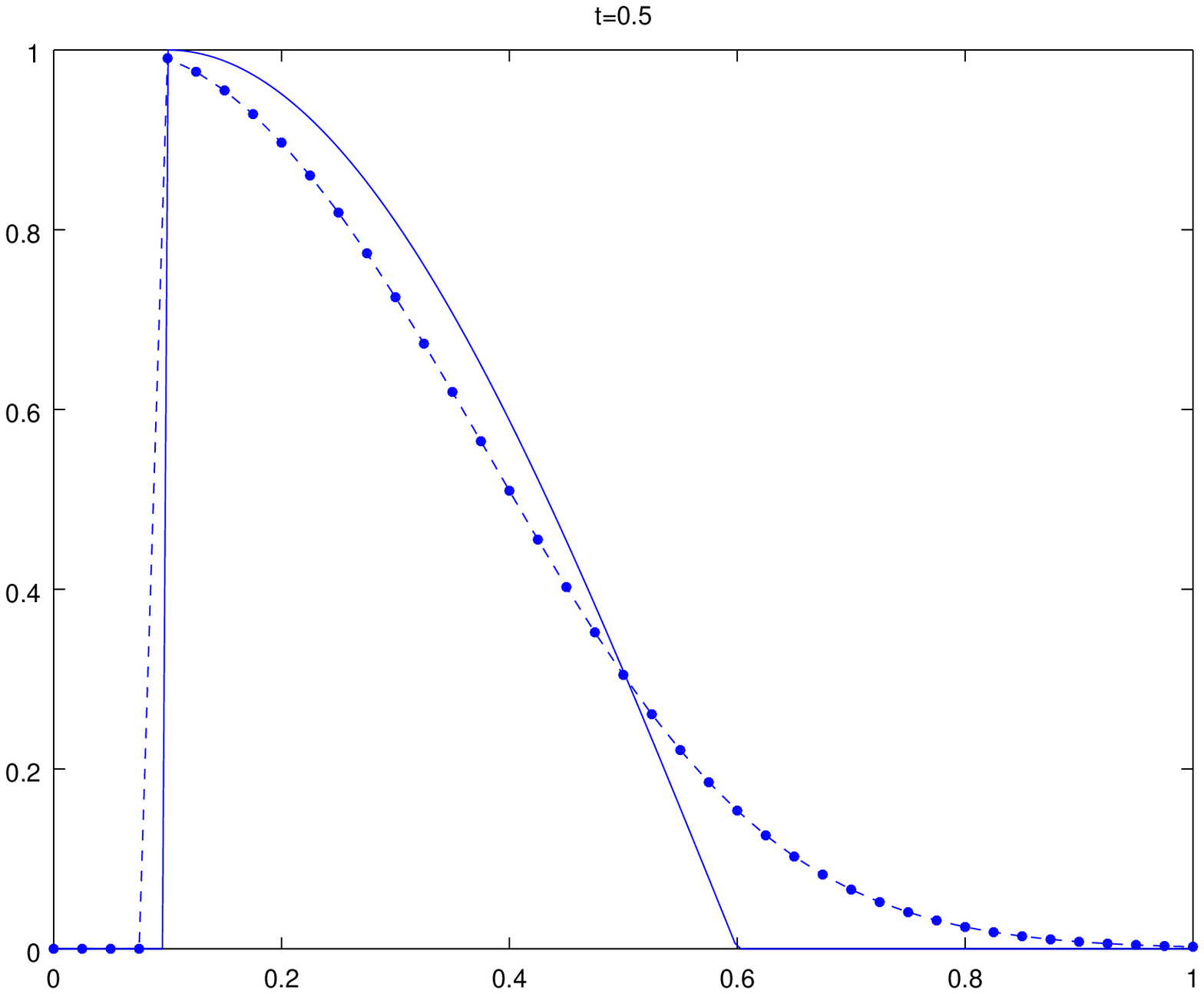}
\includegraphics[height=4.5cm]{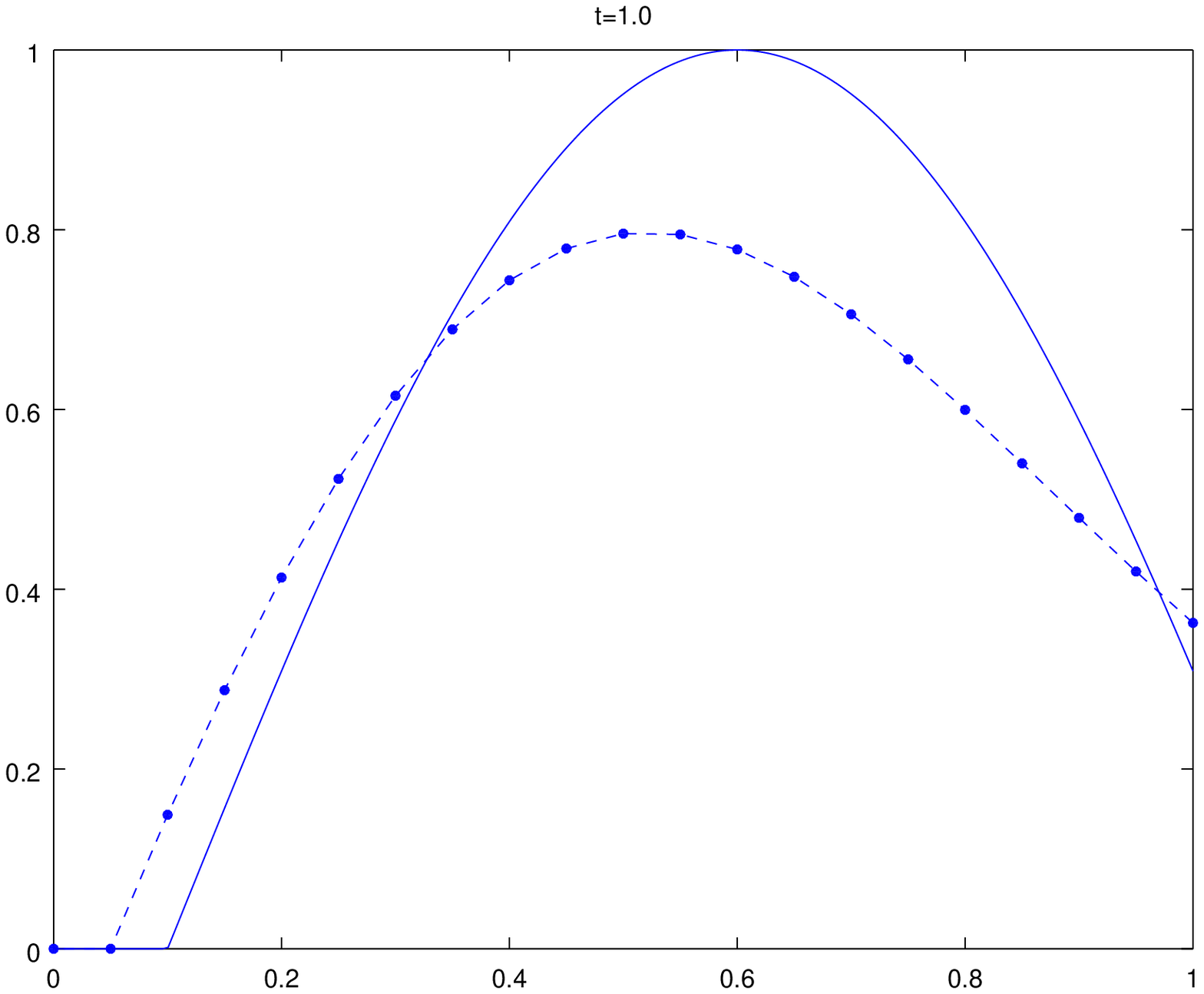}
\includegraphics[height=4.5cm]{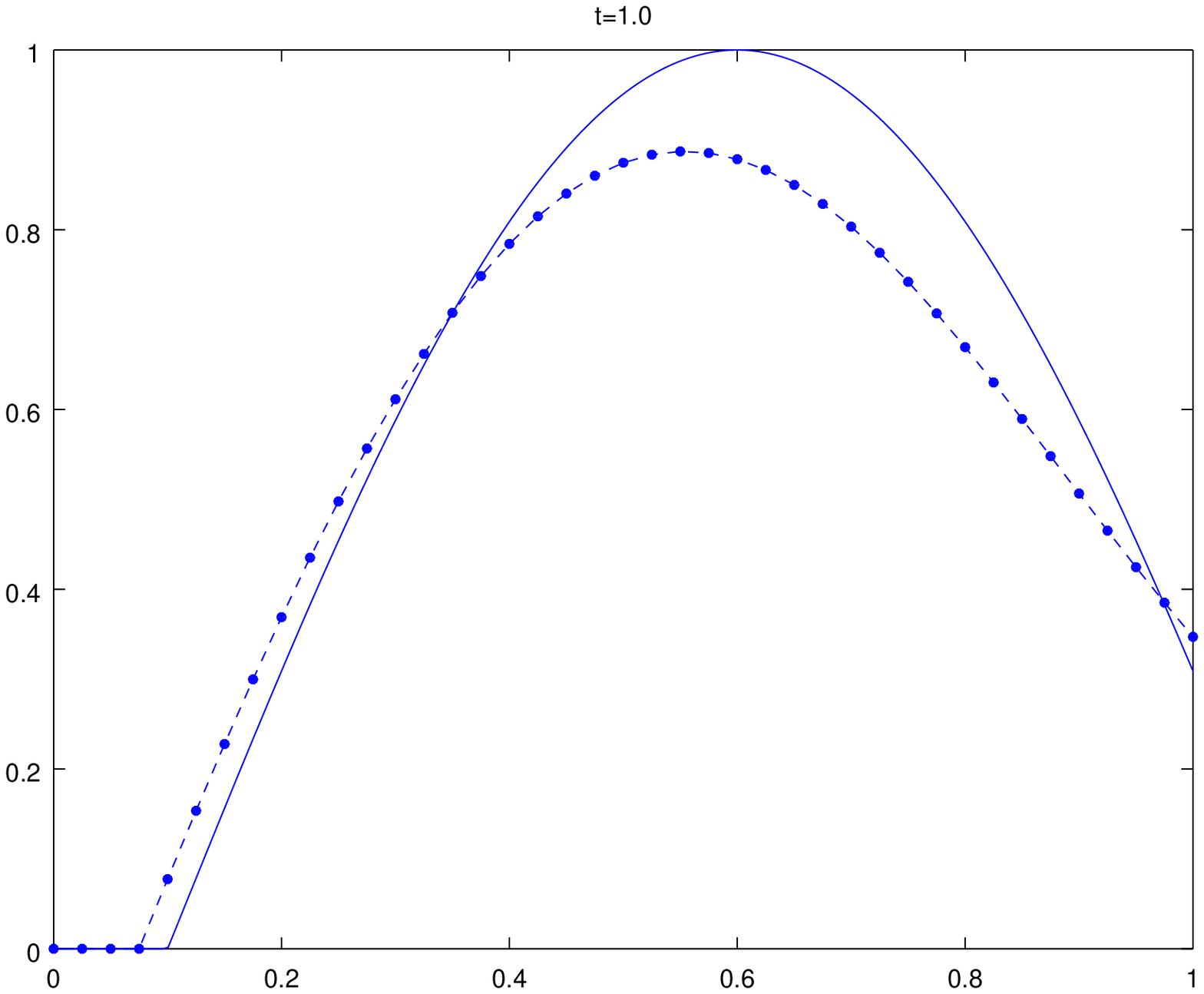}
\caption{The columns show the numerical solutions (dashed lines)
in comparison with the exact solution (continuous lines).
The situations displayed in the right column are obtained using 
a grid twice as fine as in the left column.}
\label{breuss-figure-so1}
\end{figure}
The exact solution is given in \cite{santosoliveira99}
and reads
\begin{displaymath}
u(x,t) \; = \;
\left\lbrace
\begin{array}{ccc}
u_0(x-t) 
& : & 
x<0.1 \textrm{ or } x \geq 0.1+t
\,,\\
\sin \left( \pi (0.1+t-x) \right)
+ u_0(x-t)
& : & 
0.1 \leq x < 0.1+t
\,.
\end{array}
\right.
\end{displaymath}
By Figure \ref{breuss-figure-so1}, we can compare
the exact and numerical solutions obtained with the implicit
upwind method in the same situations as displayed 
in \cite{santosoliveira99}, using also exactly the same
grid parameters. They used $\Delta x=\Delta t=1/20$ and $\Delta x=\Delta t=1/40$ in the three moments $t=1/4$, $t=1/2$, and $t=1$.

Relating to the
method used in \cite{santosoliveira99}, our scheme
is overall much more viscous. This is as expected since
Santos and Oliveira especially sought a good accuracy
of their method. We also observe experimentally
convergence to the correct solution by our method, 
documented by the bottom pictures
within Figure \ref{breuss-figure-so2}
showing results of analogous computations with a
more refined grid.
\begin{figure}[!ht]
\centering
\includegraphics[height=4.5cm]{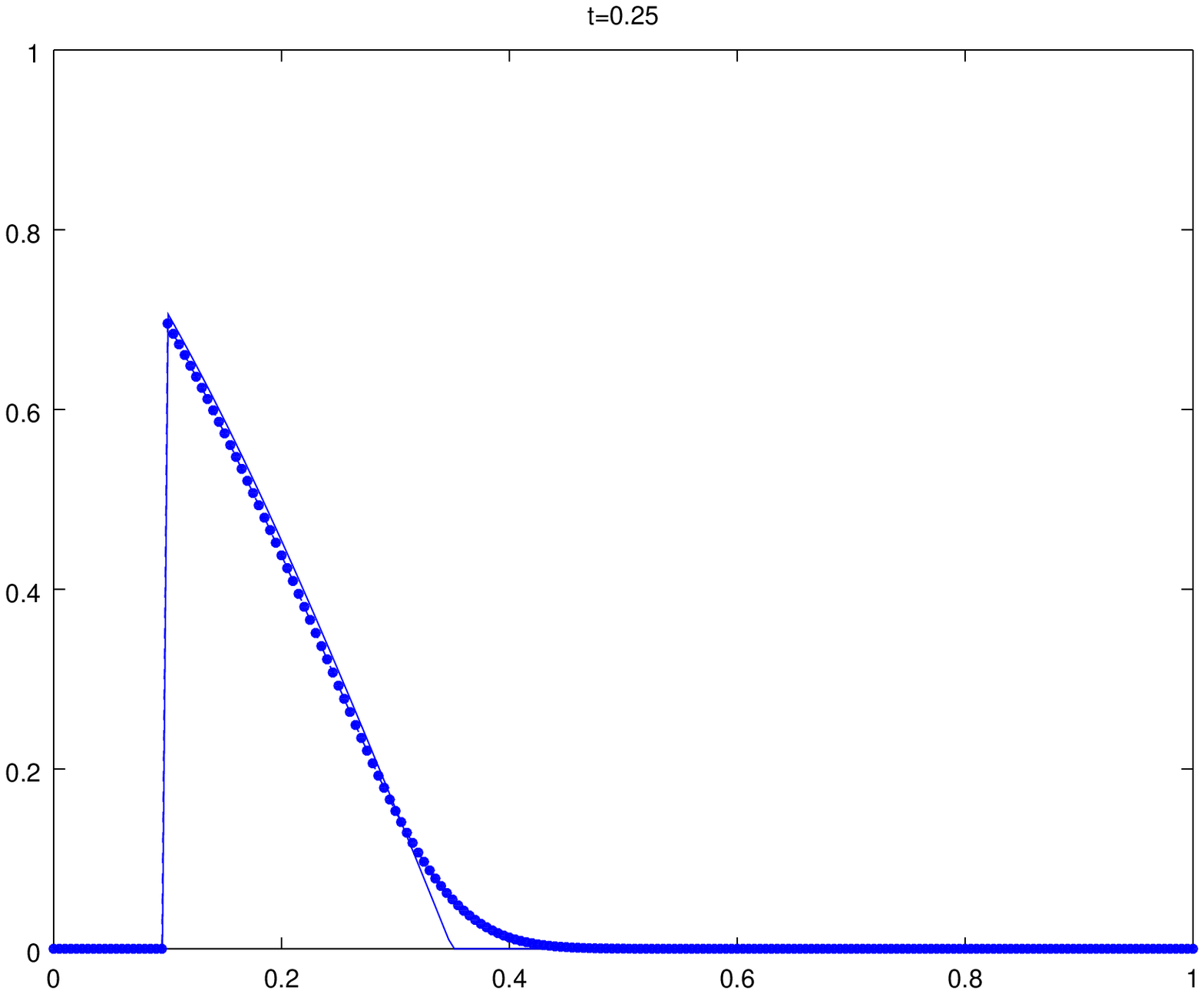}
\includegraphics[height=4.5cm]{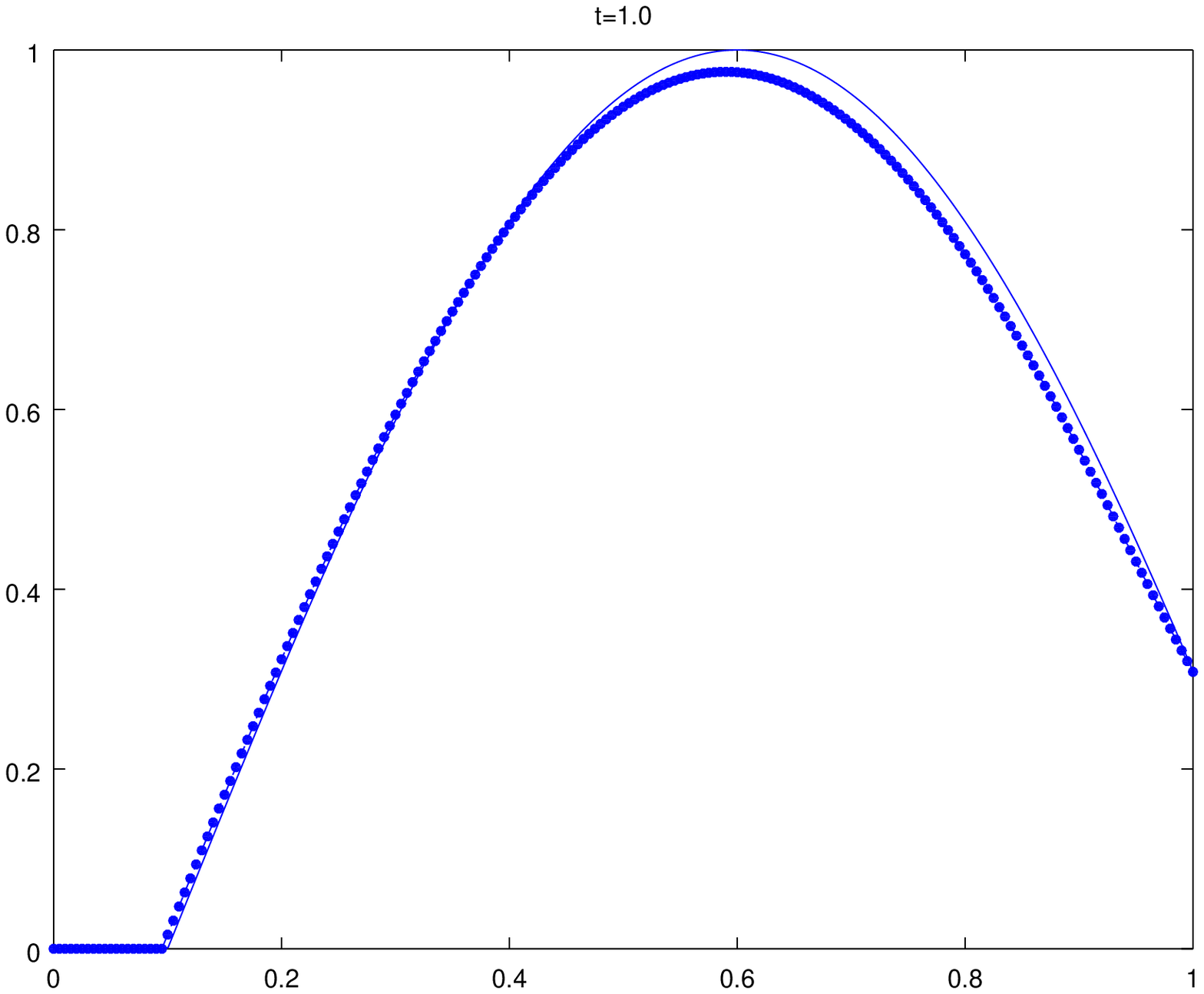}
\caption{The first and third state from Figure
\ref{breuss-figure-so1} revisited,
this time obtained via numerical approximation
using a grid ten times as fine as in the
left column of Figure \ref{breuss-figure-so1}.}
\label{breuss-figure-so2}
\end{figure}

\subsection*{Example 2}

The conservation law generally under consideration is 
\begin{displaymath}
\frac{\partial}{\partial t}
u(x,t) + 
\frac{\partial}{\partial x}
\left( \frac{1}{2} u(x,t)^2 \right) = q_x(x)
\,, 
\quad
x \in \mathbb{R}
\,, \;
t>0
\,, 
\end{displaymath}
which is used featuring different sources and initial conditions
resulting in various difficulties. The source terms in use are
\begin{displaymath}
q(x)
=
\left\lbrace
\begin{array}{cc}
0 \,,
&  
x<-1 \,,\\
\cos^2 \left( \pi x/2 \right)
\,,
& 
-1 \leq x \leq 1 \,,\\
0 \,,
& 1<x 
\end{array}
\right.
\;
\textrm{and}
\;
q(x)
=
\left\lbrace
\begin{array}{cc}
0 \,,
&  
x<-1 \,,\\
-\cos^2 \left( \pi x/2 \right)
\,,
& 
-1 \leq x \leq 1 \,,\\
0 
\,,
&  1<x \,.
\end{array}
\right.
\end{displaymath}
The initial conditions in use are
\begin{eqnarray*}
 &&u(x,0^+)=0,\qquad -\infty<x<\infty\,,\qquad
u(x,0^+)=
\left\lbrace
\begin{array}{cc}
0 \,,
&  
x<-1 \,,\\
1
\,,
& 
-1 \leq x \leq 1 \,,\\
0 
\,,
&  1<x \,,
\end{array}\right.\;
\textrm{and}\\
&&u(x,0^+)=
\left\lbrace
\begin{array}{cc}
0 \,,
&  
x<-1 \,,\\
-1
\,,
& 
-1 \leq x \leq 1 \,,\\
0 
\,,
&  1<x \,.
\end{array}\right.
\end{eqnarray*}

The following four experiments are analogous
to the ones in \cite{grlebano97}, using exactly 
the same grid parameters as initially in \cite{grlebano97}
where later on spatial regridding was used in order to
obtain sharp shock profiles. We use the implicit Lax-Friedrichs method with $\delta x=0.025$ and $\delta t=0.0125$.

In the Figures \ref{breuss-figure-glbn1} and
\ref{breuss-figure-glbn2},
we show in all test cases from top to bottom
the numerical solutions obtained by using our method at times
$t=0.2, \, 0.5, \, 1.0$ and $3.0$ (line featuring small circles)
together with the stationary solution (continuous line). 
Thereby, different experiments correspond to different
columns of pictures.

Concerning the first experiment, the numerical solution
is almost identical to the exact one except at the point $x=0.0$
where a grid point is located exactly on the shock front.
With respect to the other
experiments, the numerical solutions exhibit slightly smeared shocks
while they are otherwise quite accurate. 
Comparing with the numerical results shown in 
\cite{grlebano97}, not employing a regridding
procedure results in slightly smeared shocks. Further numerical 
experiments have shown that we can employ much larger
time steps --- usually of about $20$--$30$ times the one used for the
presented experiments --- without degrading our numerical solution.

\begin{figure}[!ht]
\centering
\includegraphics[height=4.5cm]{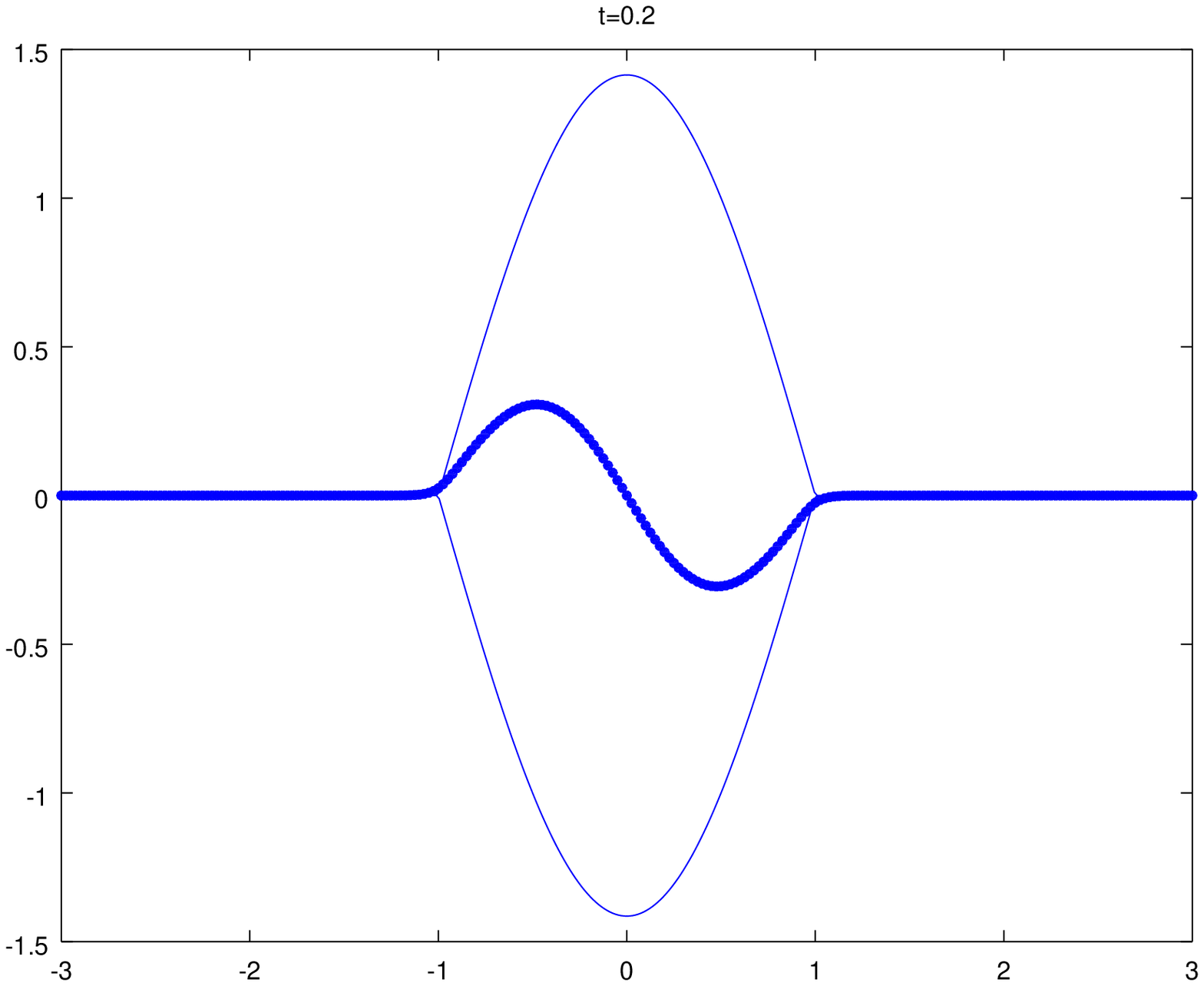}
\includegraphics[height=4.5cm]{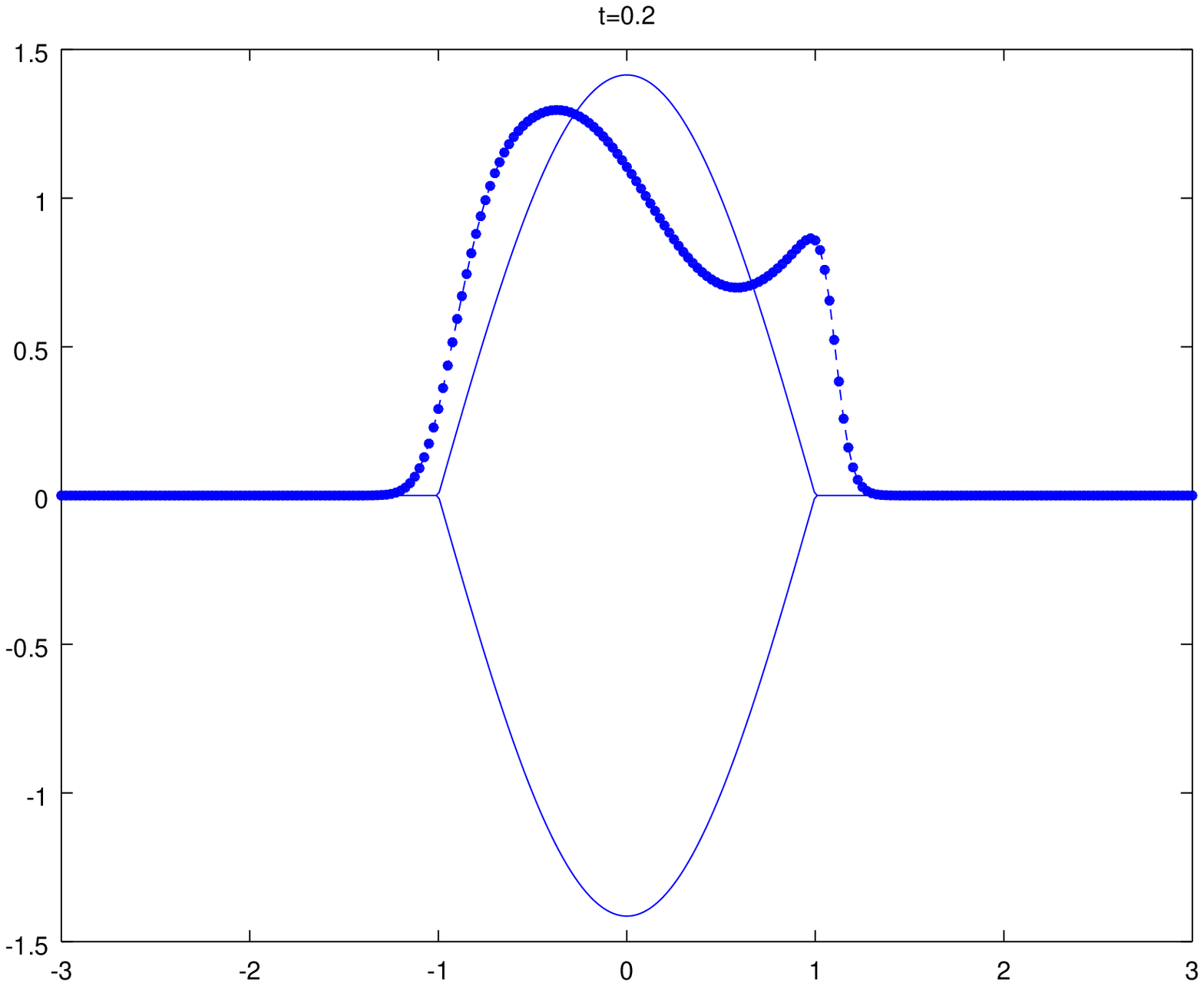}
\includegraphics[height=4.5cm]{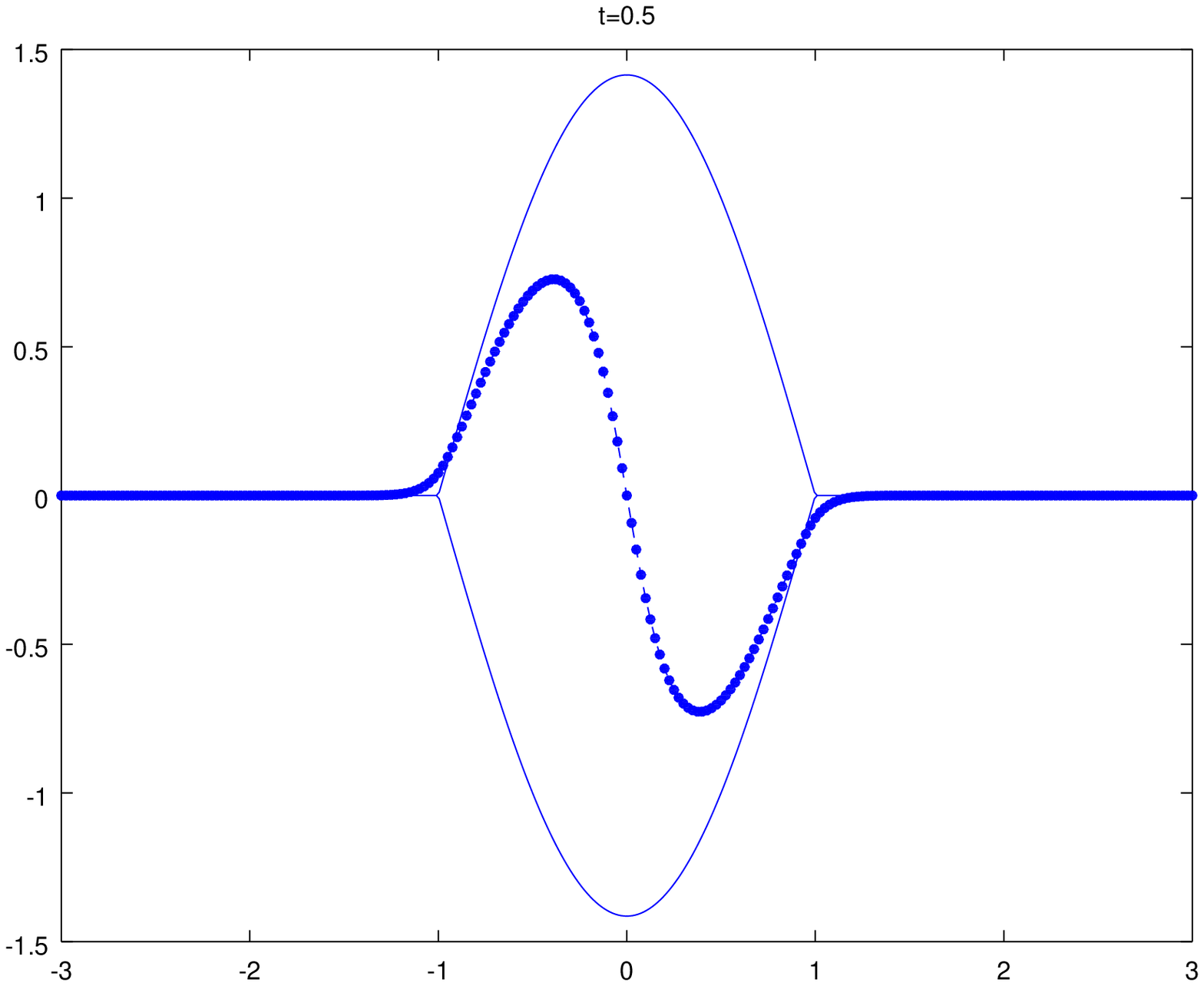}
\includegraphics[height=4.5cm]{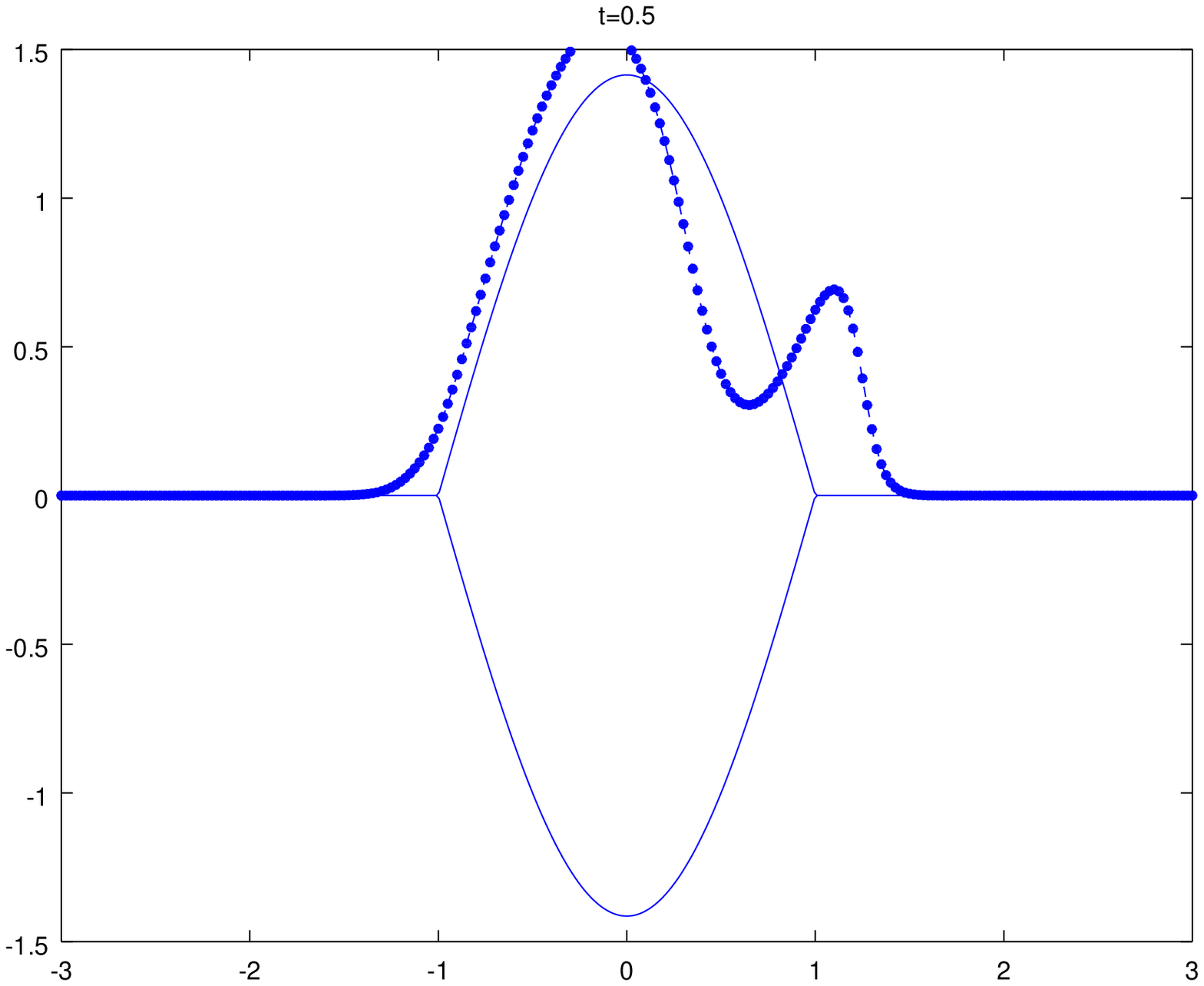}
\includegraphics[height=4.5cm]{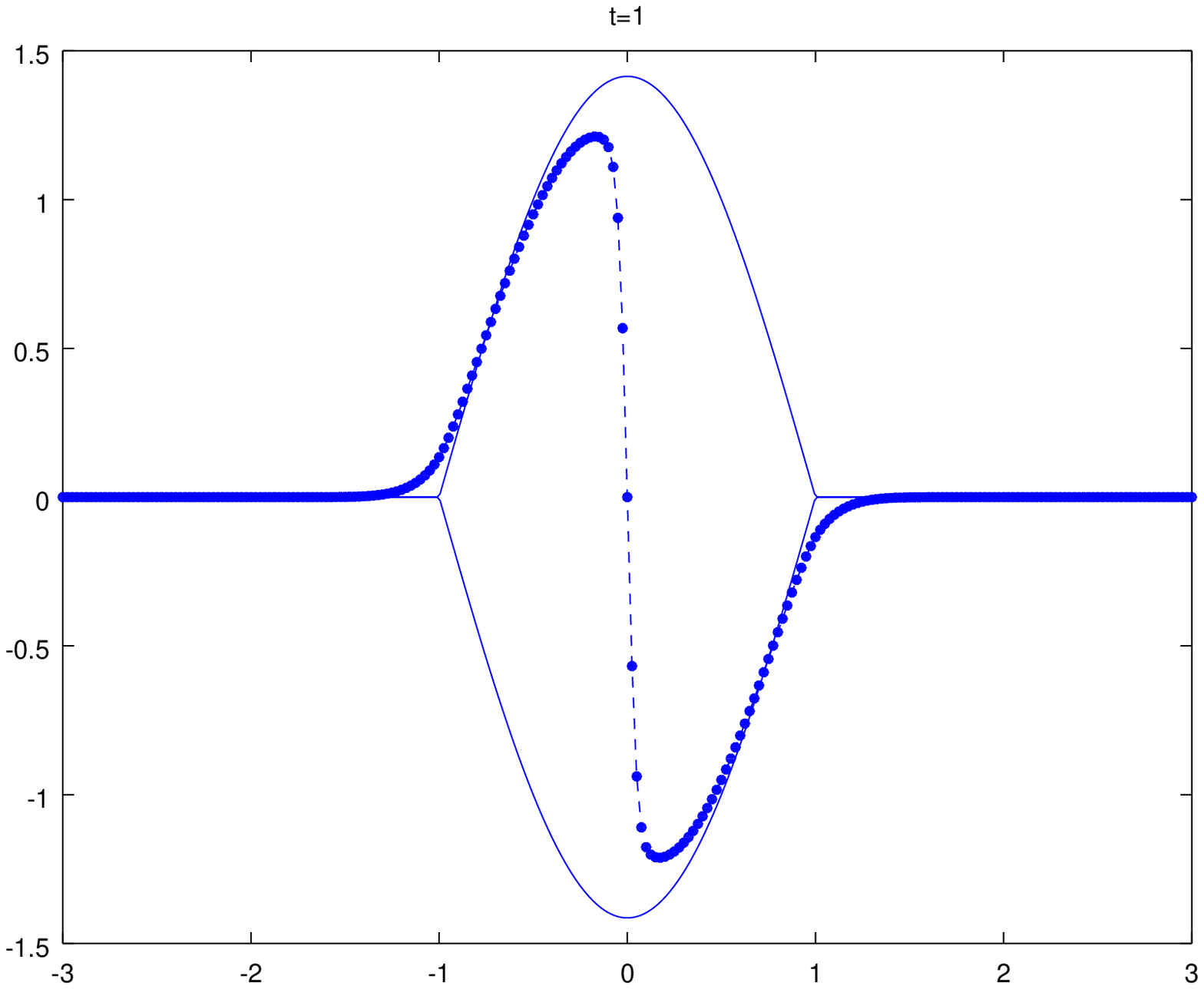}
\includegraphics[height=4.5cm]{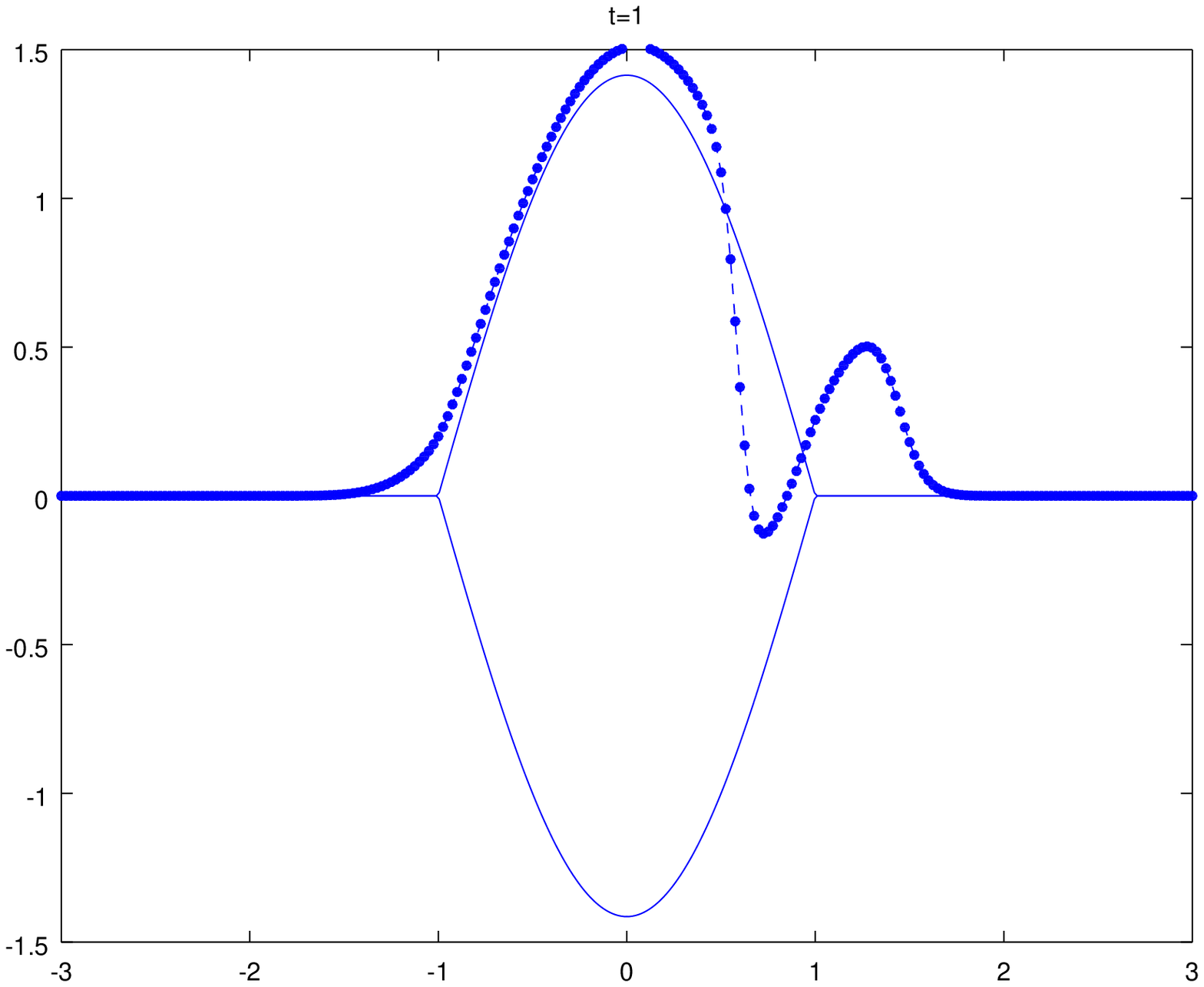}
\includegraphics[height=4.5cm]{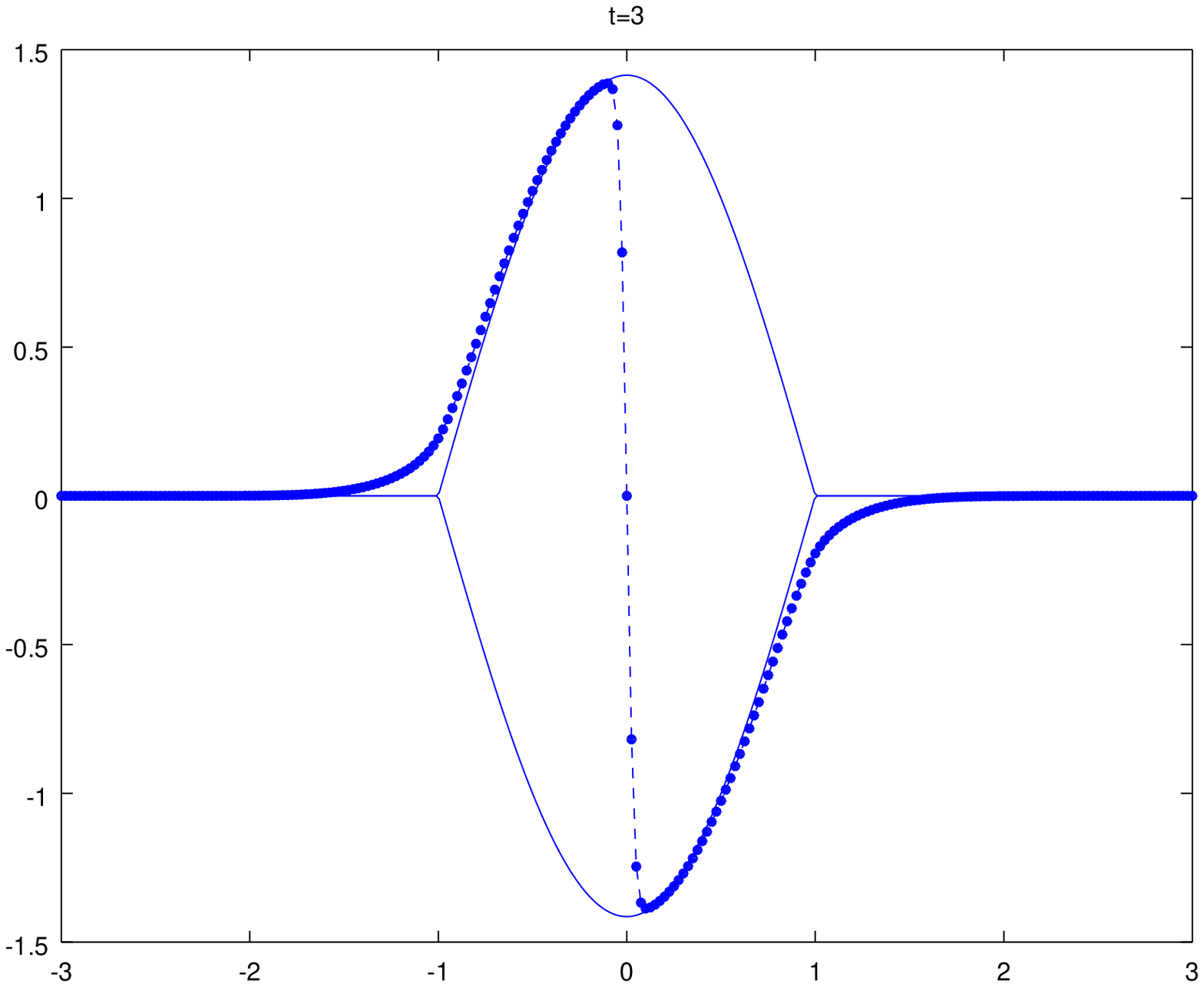}
\includegraphics[height=4.5cm]{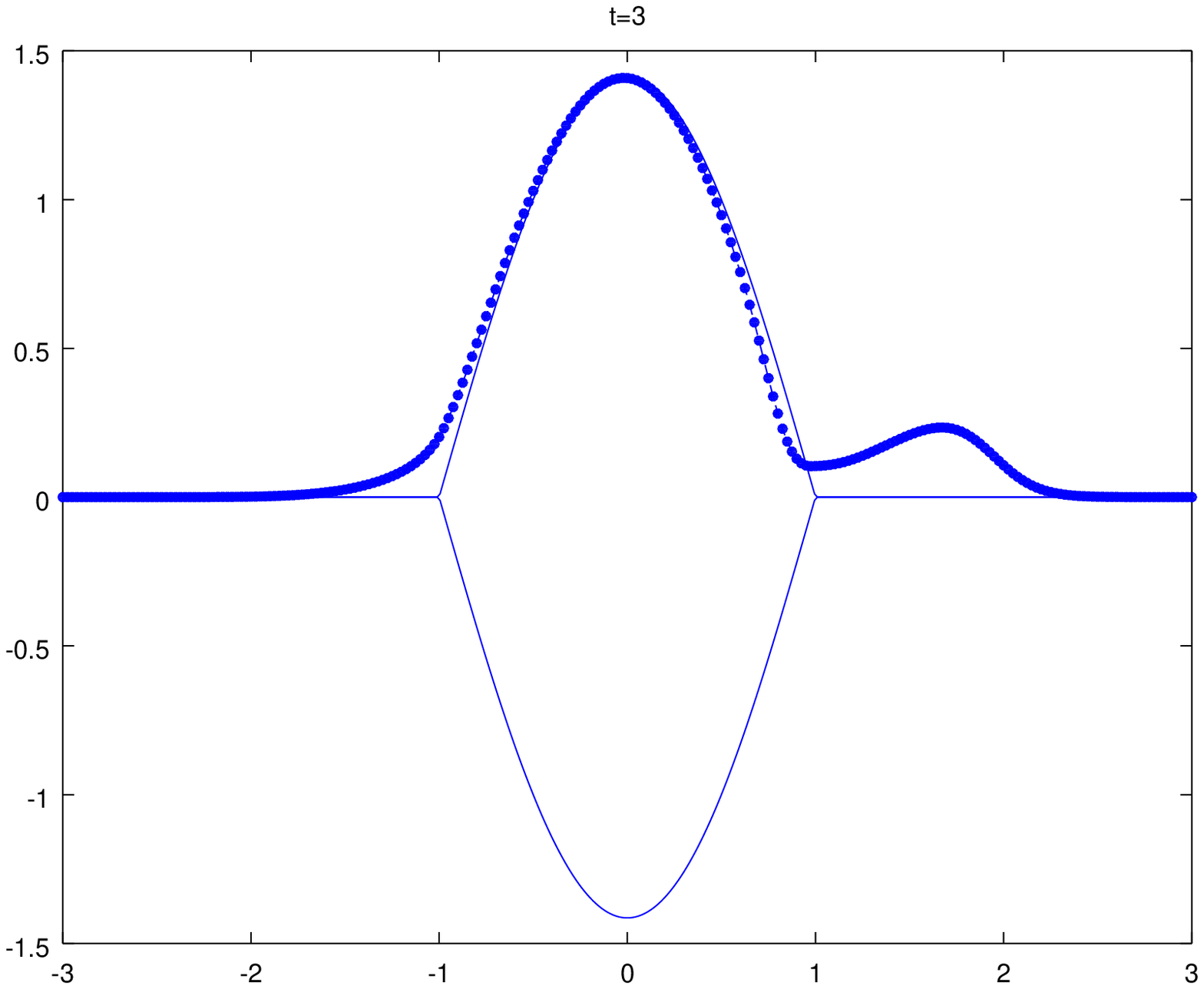}
\caption{Experiments one (left) and two (right).}
\label{breuss-figure-glbn1}
\end{figure}

\begin{figure}[!ht]
\centering
\includegraphics[height=4.5cm]{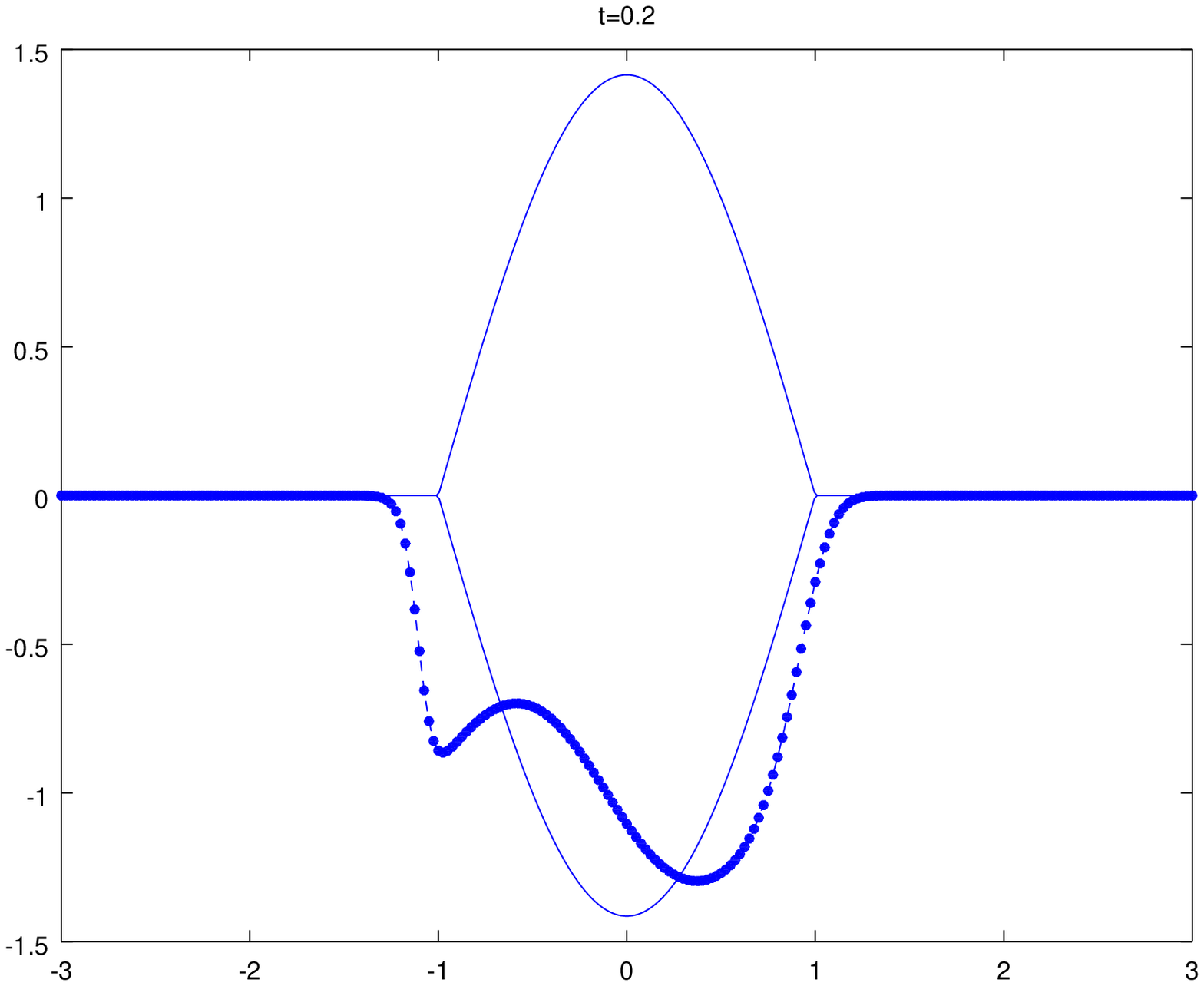}
\includegraphics[height=4.5cm]{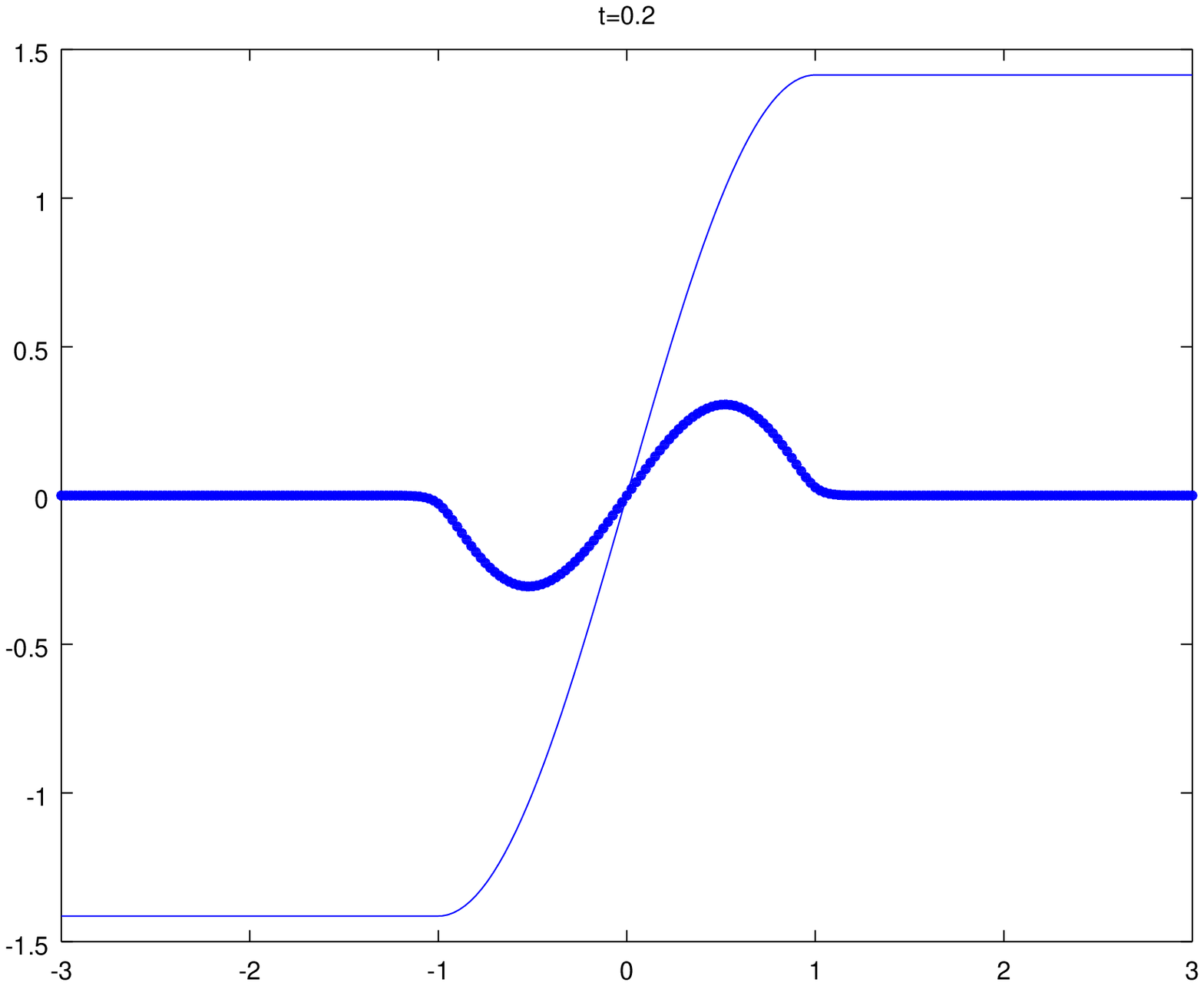}
\includegraphics[height=4.5cm]{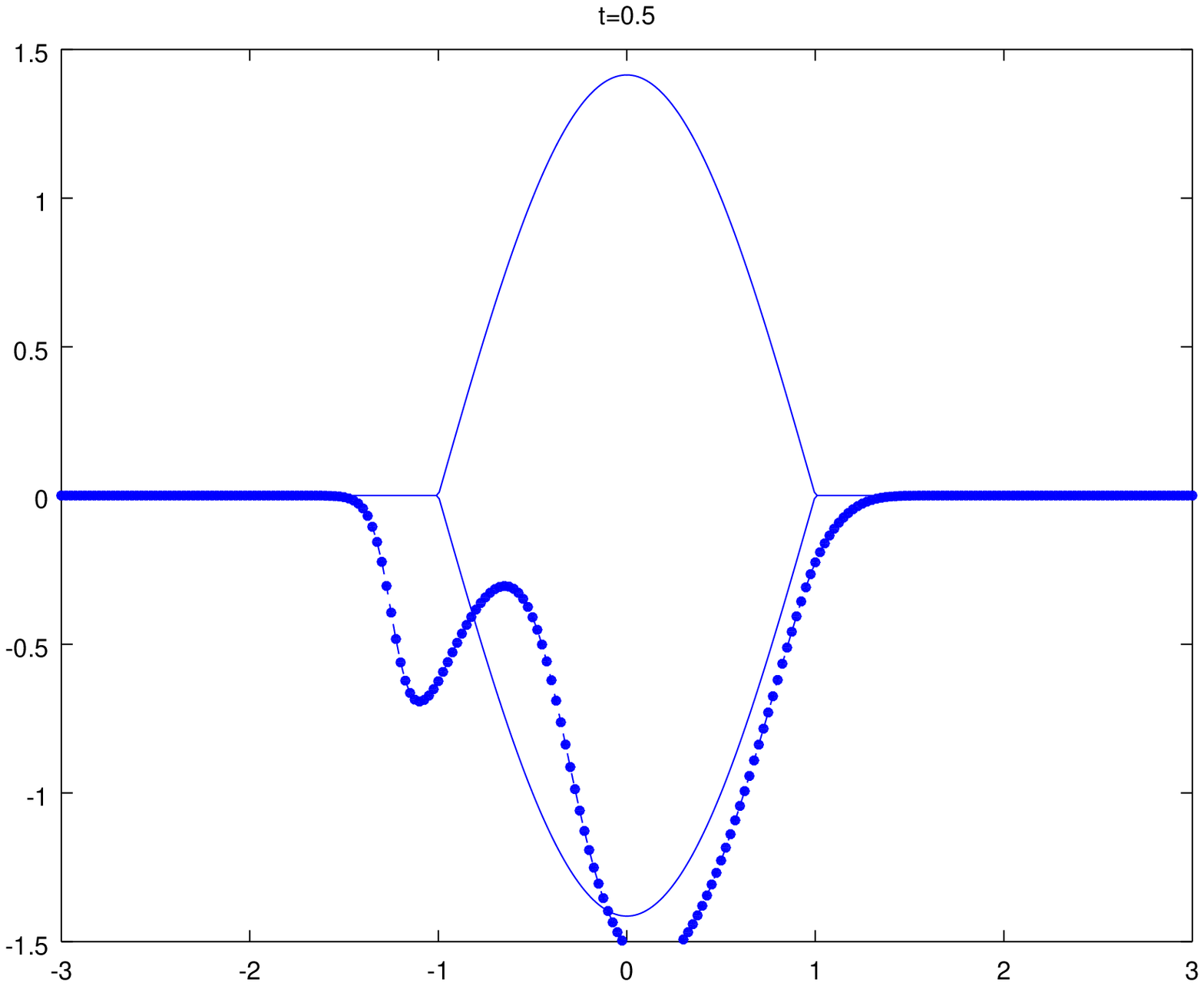}
\includegraphics[height=4.5cm]{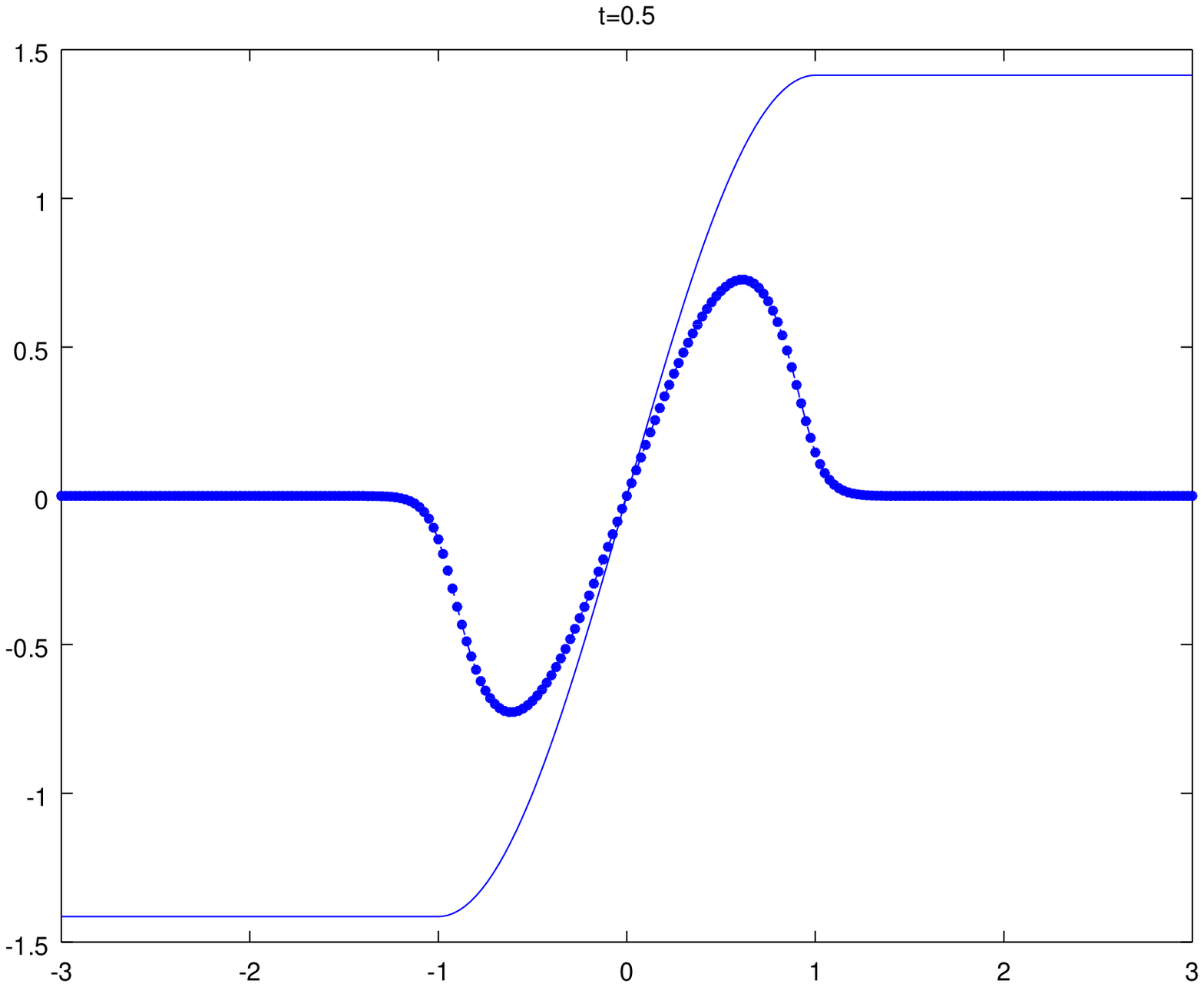}
\includegraphics[height=4.5cm]{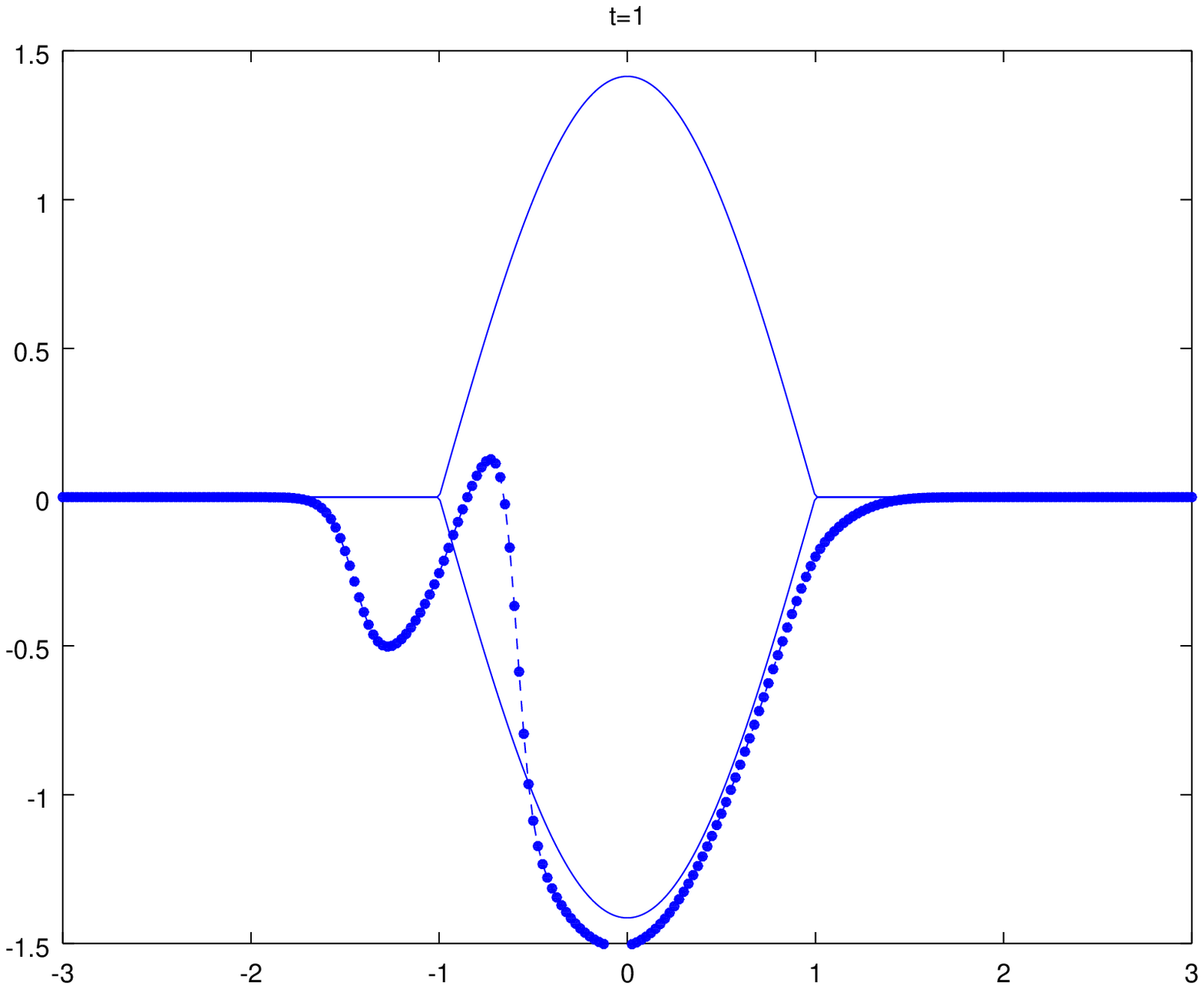}
\includegraphics[height=4.5cm]{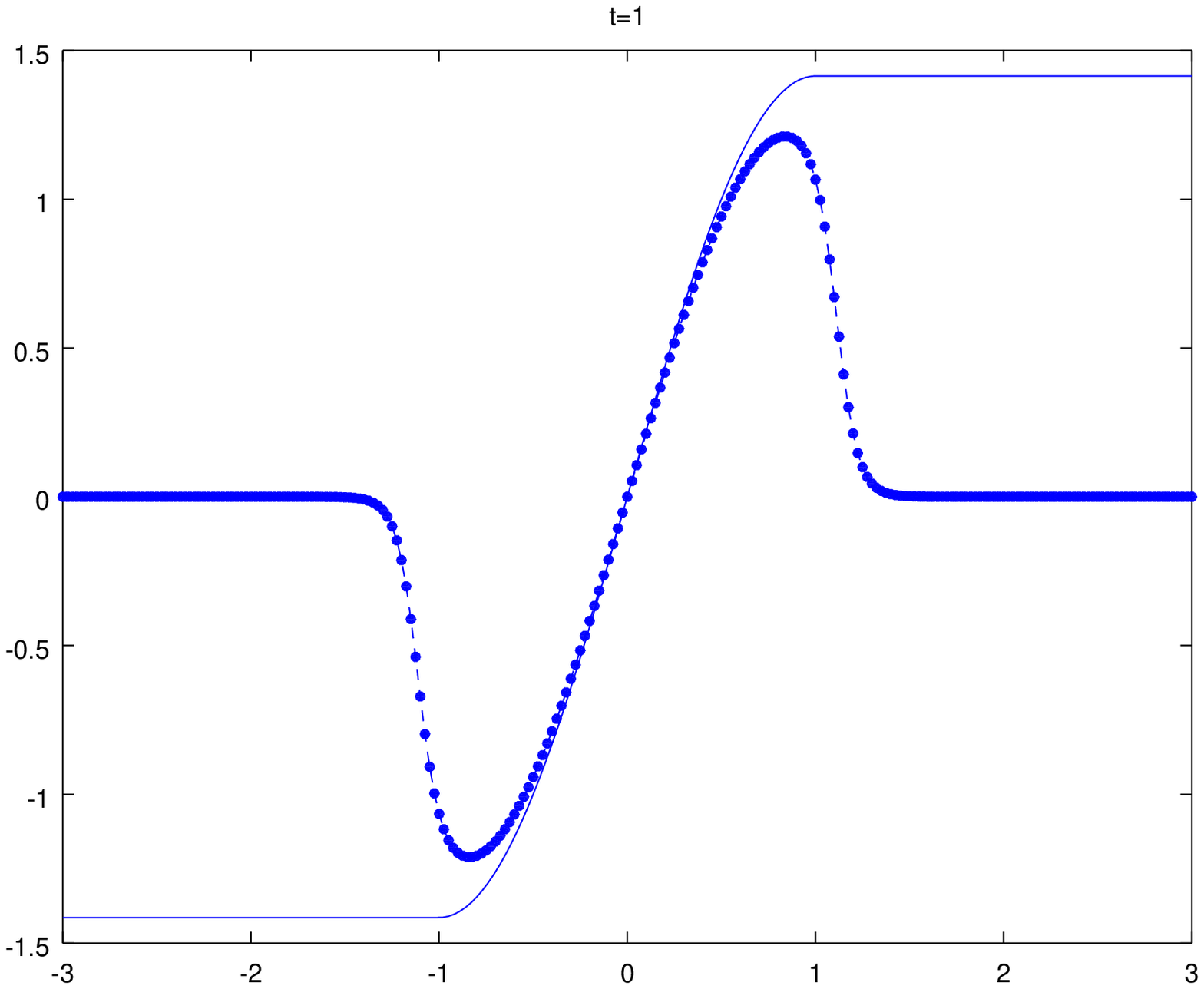}
\includegraphics[height=4.5cm]{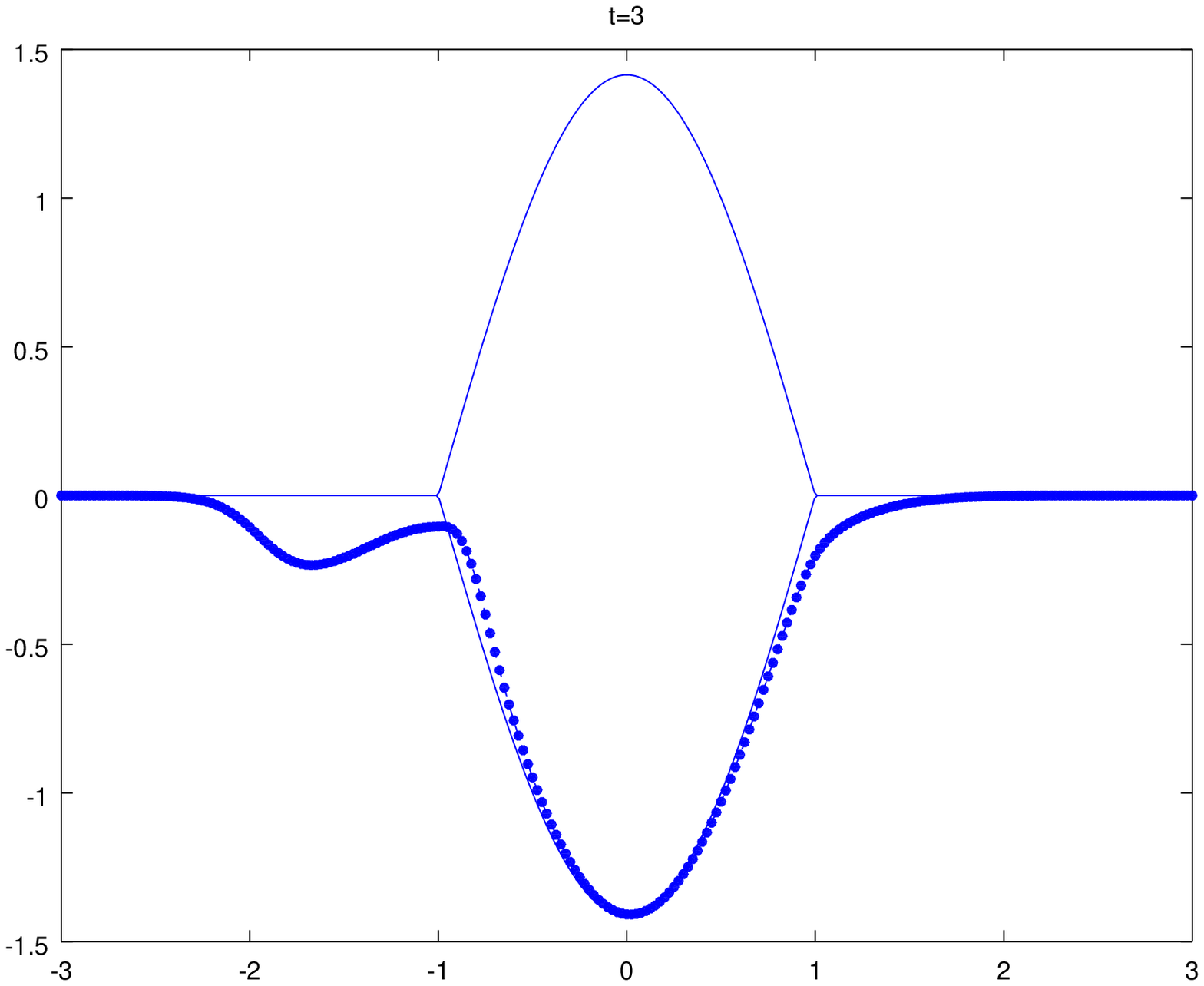}
\includegraphics[height=4.5cm]{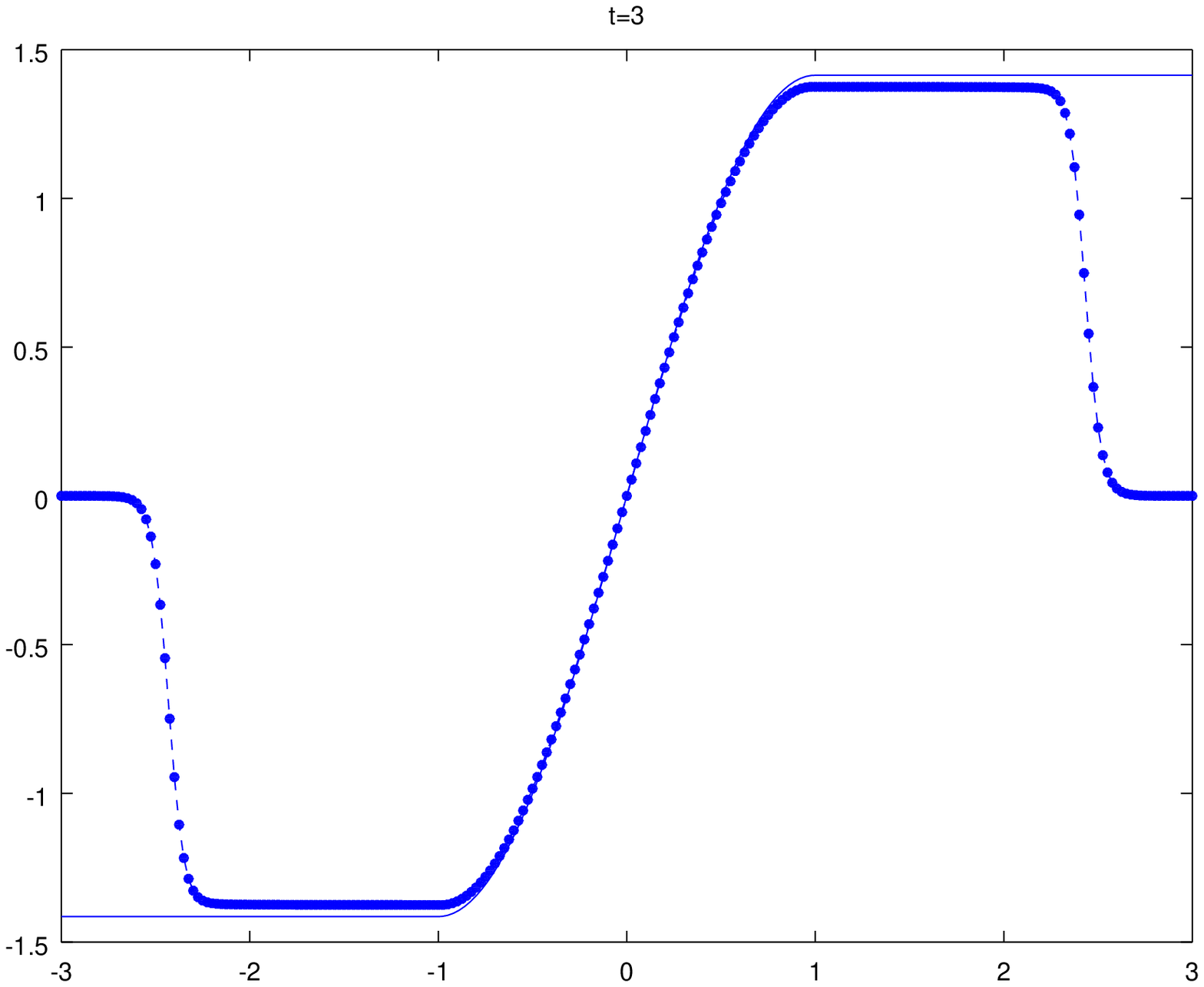}
\caption{Experiments three (left) and four (right).}
\label{breuss-figure-glbn2}
\end{figure}

\subsection*{Example 3}

The conservation law under consideration is 
\begin{displaymath}
\frac{\partial}{\partial t}u(x,t)
+
\frac{\partial}{\partial x} u(x,t)
=
- \mu  u(x,t) \bigl( u(x,t)-1 \bigr) \left( u(x,t)- \frac{1}{2}
\right) \end{displaymath}
which exhibits a nonlinear source term with
an increasingly stiff behavior for $\mu$ growing large. 

Our numerical investigations
are completely analogous to the ones in \cite{levequeyee90}.
The experiments consist of the numerical
solution of a Riemann problem whose
exact solution features a shock front
moving from $x=0.3$ to $x=0.6$ after a couple
of time steps, see Figure \ref{breuss-figure-ly}.
For small and medium $\mu$, the numerical solution shows
the correct behavior incorporating numerical viscosity; note
the sharpening and slight translation
of the shock approximation for $\mu=100$,
an effect of the increasing stiffness of the source term.
For $\mu=1000$ the usual problem is faced, 
see \cite{levequeyee90} for details.
This experiment shows that although the discussed methods
are generally capable to deal with non-linear sources,
they are not recommended without modification for
stiff problems even though they are fully implicit.
Of course, grid refinement results in the 
approximation of the correct solution as expected.

\begin{figure}[!ht]
\centering
\includegraphics[height=4.5cm]{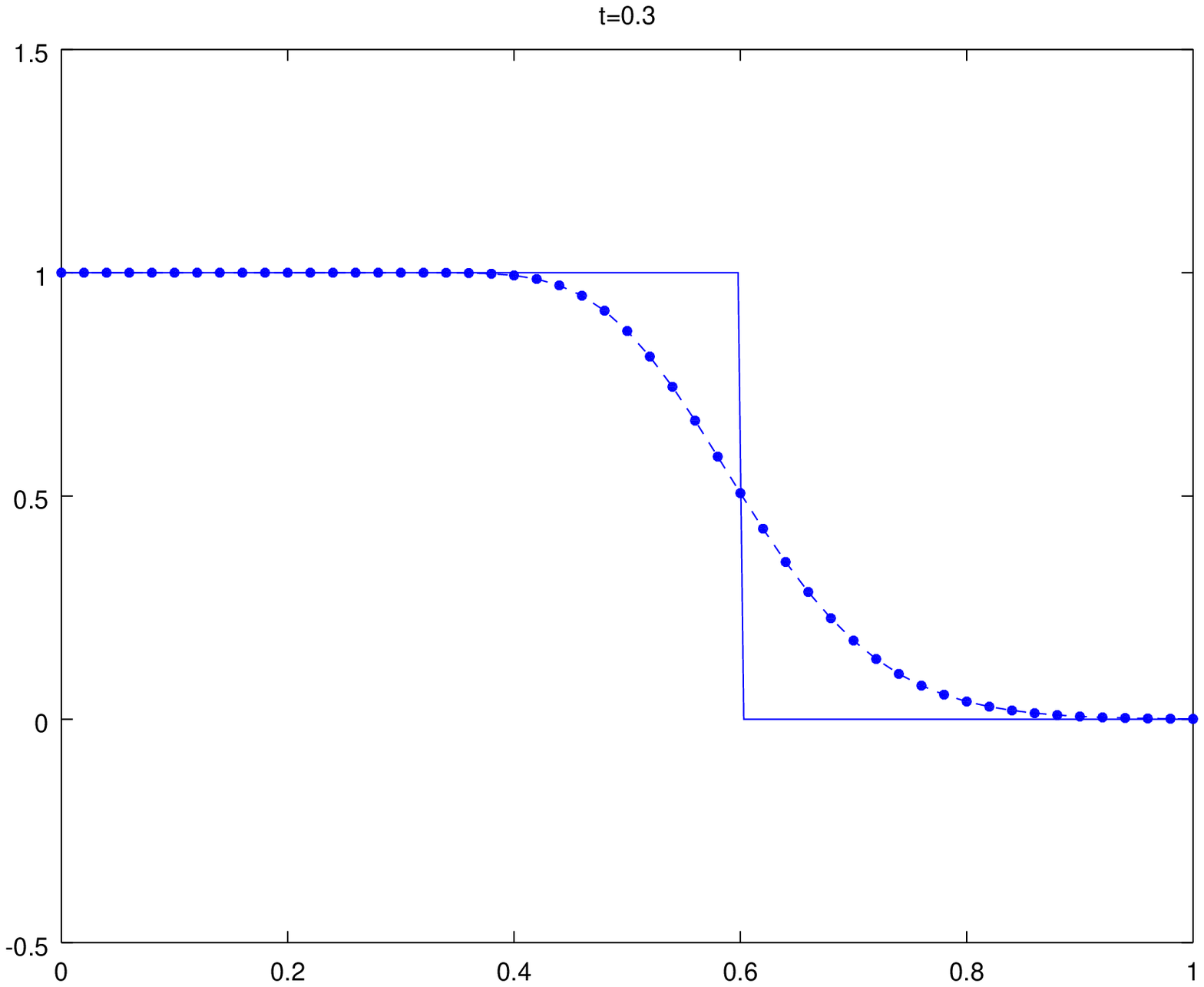}
\includegraphics[height=4.5cm]{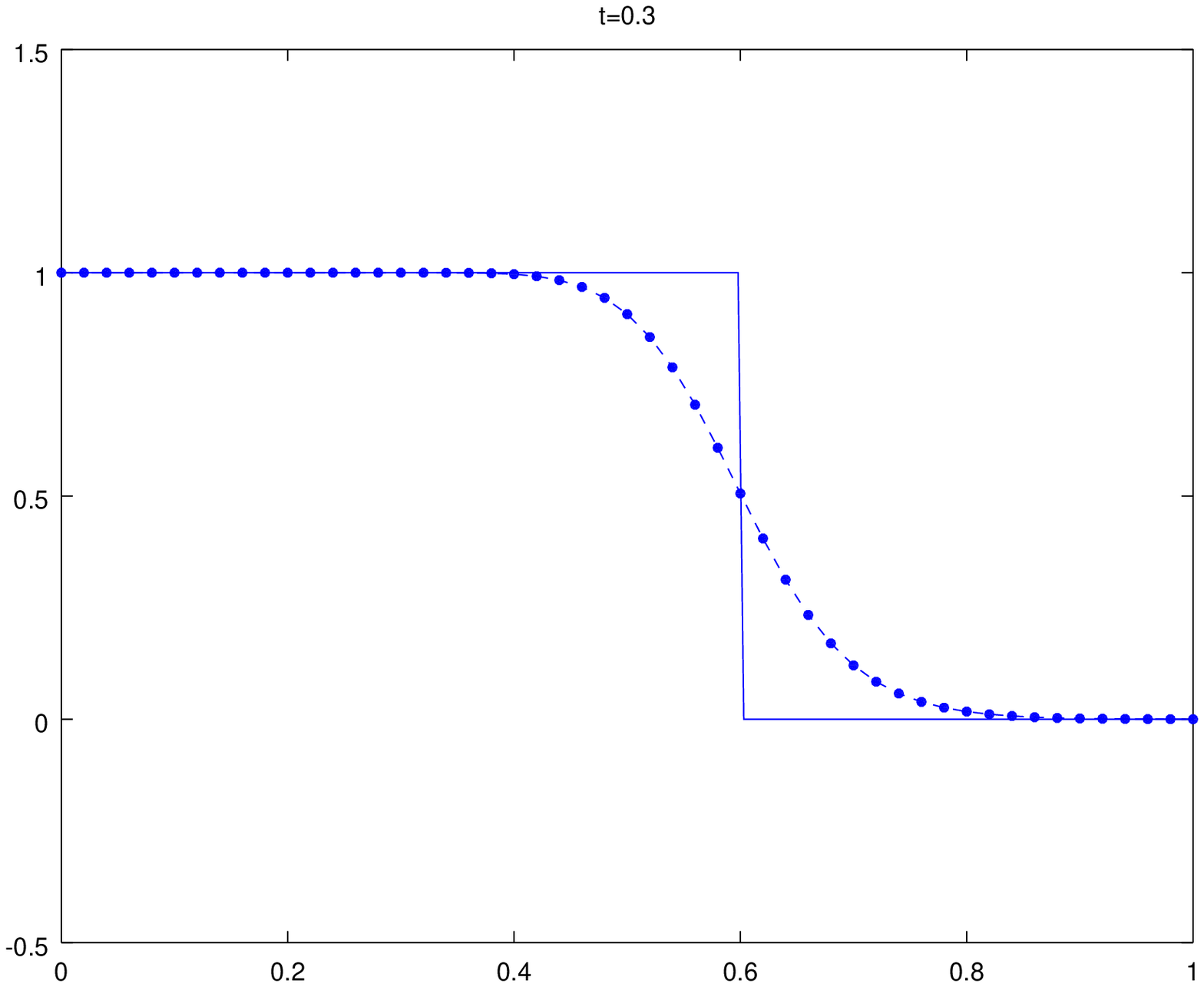}
\includegraphics[height=4.5cm]{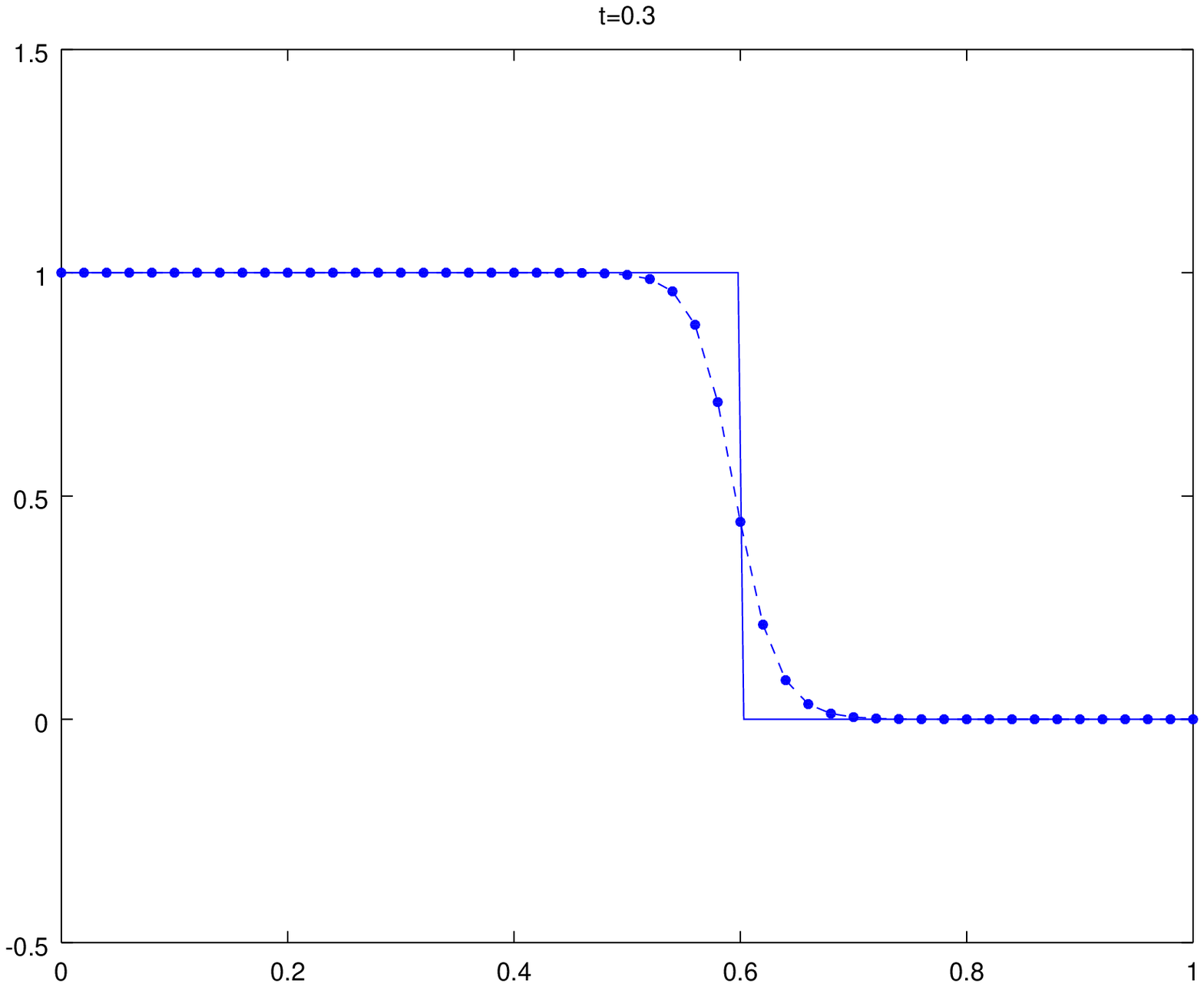}
\includegraphics[height=4.5cm]{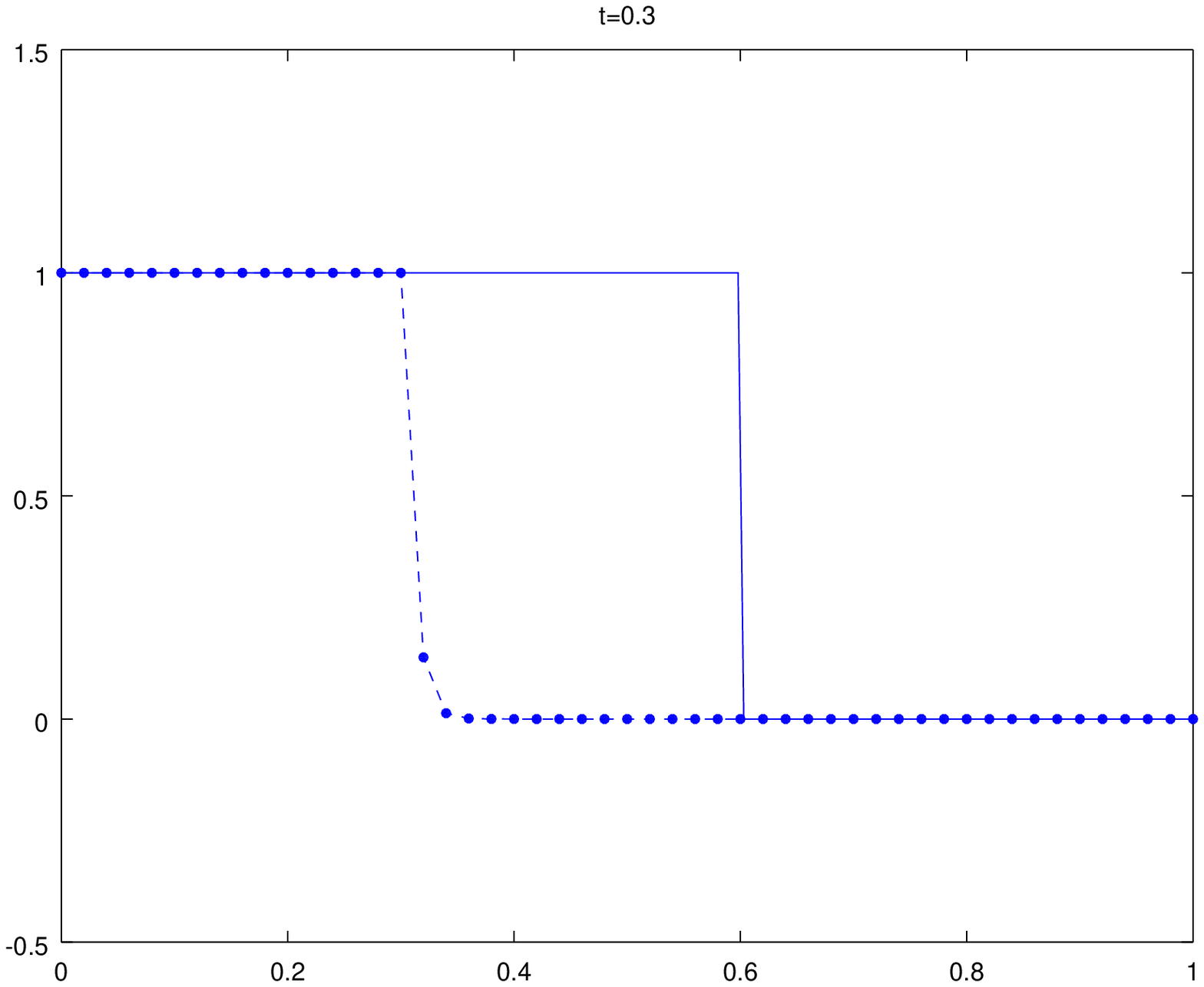}
\caption{Numerical solutions (dashed lines) compared with the exact
solution (continuous lines). The pictures correspond from left to right
and top to bottom to the choices $\mu=1,10,100,1000$.}
\label{breuss-figure-ly}
\end{figure}

\section{Summary and conclusive remarks}
\label{breuss-section-6}

In this paper, we have introduced a new concept for implicit methods for scalar conservation laws in one or more spatial dimensions which may also include source terms of
different type. We developed implicit notions that are centered around a monotonicity criterion and show the relation between a numerical scheme and a discrete entropy inequality.
We investigate in detail three implicit methods and give a convergence proof. Hence, we extend the rigorously verified
range of applicability of those three implicit numerical methods.
By numerical experiments we have shown the validity
and usefulness of our theoretical results.

\bibliographystyle{plain}
\bibliography{bibo}

\end{document}